\documentclass[suppldata]{interact}

\usepackage{epstopdf}
\usepackage[caption=false]{subfig}
\usepackage{hyperref}
\usepackage[dvipsnames]{xcolor}
\usepackage{stackengine}
\usepackage{amsmath}   
\usepackage{amssymb}
\usepackage{bbm}
\usepackage{upgreek}
\usepackage{comment}
\usepackage{algorithm}
\usepackage{algpseudocode}
\usepackage{mathtools}

\usepackage{pifont}
\usepackage{scalerel}

\usepackage{multirow}

\usepackage[numbers,sort&compress]{natbib}
\bibpunct[, ]{[}{]}{,}{n}{,}{,}
\renewcommand\bibfont{\fontsize{10}{12}\selectfont}

\theoremstyle{plain}
\newtheorem{theorem}{Theorem}[section]
\newtheorem{lemma}[theorem]{Lemma}
\newtheorem{corollary}[theorem]{Corollary}
\newtheorem{proposition}[theorem]{Proposition}

\theoremstyle{definition}
\newtheorem{definition}[theorem]{Definition}
\newtheorem{example}[theorem]{Example}

\theoremstyle{remark}
\newtheorem{remark}{Remark}
\newtheorem{assumption}{Assumption}

\usepackage{cleveref}

\crefformat{section}{\S#2#1#3} 
\crefformat{subsection}{\S#2#1#3}
\crefformat{subsubsection}{\S#2#1#3}

\usepackage{tikz}

\definecolor{lime}{HTML}{A6CE39}
\DeclareRobustCommand{\orcidicon}{%
	\begin{tikzpicture}
	\draw[lime, fill=lime] (0,0) 
	circle [radius=0.16] 
	node[white] {{\fontfamily{qag}\selectfont \tiny ID}};
	\draw[white, fill=white] (-0.0625,0.095) 
	circle [radius=0.007];
	\end{tikzpicture}
	\hspace{-2mm}
}

\foreach \x in {A, ..., Z}{%
	\expandafter\xdef\csname orcid\x\endcsname{\noexpand\href{https://orcid.org/\csname orcidauthor\x\endcsname}{\noexpand\orcidicon}}
}

\usepackage{multirow}

\usepackage[numbers,sort&compress]{natbib}
\bibpunct[, ]{[}{]}{,}{n}{,}{,}
\renewcommand\bibfont{\fontsize{10}{12}\selectfont}

\begin{document}


\title{Global and local approaches for the minimization of a sum of pointwise minima of convex functions}

\author{
\name{Guillaume Van Dessel\textsuperscript{a}\orcidA{}\thanks{CONTACT Guillaume Van Dessel. Email: guillaume.vandessel@uclouvain.be} and Fran\c{c}ois Glineur\textsuperscript{a,b}\orcidB{}}
\affil{\textsuperscript{a}UCLouvain, ICTEAM (INMA), 4 Avenue Georges Lema\^{\i}tre, Louvain-la-Neuve, BE; \\\textsuperscript{b}UCLouvain, CORE, 34
Voie du Roman Pays, Louvain-la-Neuve, BE}
}

\maketitle

\begin{abstract}

\noindent Numerous machine learning and industrial problems can be modeled as the minimization of a sum of $N$ so-called \emph{clipped} (or \emph{truncated}) convex functions (SCC), i.e. each term of the sum stems as the pointwise minimum between a constant and a convex function. In this work, we extend this framework to capture more problems of interest. Specifically, we allow each term of the sum to be a pointwise minimum of an arbitrary number of convex functions, called \emph{components}, turning the objective into a sum of pointwise minima of convex functions (SMC). \\ \vspace{-5pt}\\
\textbf{Local}. As emphasized in dedicated works, problem (SCC) is already NP-hard, highlighting an appeal for scalable local heuristics. In this spirit, one can express (SMC) objectives as the difference between two convex functions to leverage the possibility to apply (DC) algorithms to compute critical points of the problem. Our approach does not rely on the above (DC) decomposition but rather on a bi-convex reformulation of the problem. From there, we derive a family of local methods, dubbed as relaxed alternating minimization (\texttt{r-AM}) methods, that include classical alternating minimization (\texttt{AM}) as a special case. We prove that every accumulation point of \texttt{r-AM} is critical. In addition, we show the empirical superiority of \texttt{r-AM}, compared to traditional \texttt{AM} and (DC) approaches, on \emph{piecewise-linear regression} and \emph{restricted facility location} problems. \\ \vspace{-5pt}\\
\noindent \textbf{Global}. Under mild assumptions, (SCC) can be cast as a mixed-integer convex program (MICP) using perspective functions. This approach can be generalized to (SMC) but introduces many copies of the primal variable. In contrast, we suggest a compact \emph{big-M} based (MICP) equivalent formulation of (SMC),  free of these extra variables. Finally, we showcase practical examples where solving our (MICP), restricted to a neighbourhood of a given candidate (i.e. output iterate of a local method), will either certify the candidate's optimality on that neighbourhood or providing a new point, strictly better, to restart the local method.

\end{abstract}

\vspace{-20pt}

\section{Introduction}
\label{intro_notes}
In this paper, we are concerned with (constrained) optimization problems of the form
\begin{equation}
    F^* = \min_{x \in \mathcal{X}}\hspace{3pt}\Bigg\{F(x):= \bar{h}(x) + \frac{1}{N}\,\sum_{s=1}^{N}\,\underbrace{\min_{l=1,\dots,n_s}\,h_{l}^{(s)}(x)}_{h^{(s)}(x)}\Bigg\} 
\label{eq:min_problem}
\tag{SMC}
\end{equation}

\noindent where $N \in \mathbb{N}$ and $\mathcal{X} \subseteq \mathbb{R}^d$. For every index\footnote{For any natural number $n \in \mathbb{N}$, we use the shorthand $[n]$ to depict the set $\{1,\dots,n\}$.} $s \in [N]$, we consider $n_s \in \mathbb{N}$ continuous functions $h^{(s)}_1,\dots,h^{(s)}_{n_s}$ called \emph{component functions}. Together with $\bar{h}$ lower semicontinuous, dubbed as the \emph{main function}, they are assumed to be closed, proper and convex. Throughout the sequel, we further assume that problem \eqref{eq:min_problem} is \emph{feasible} (i.e. $\mathcal{X} \cap \text{dom}\,F = \mathcal{X} \cap \text{dom}\,\bar{h} \not=\emptyset$ hence $F^*<\infty$) and \emph{bounded} (i.e. $F^*>-\infty$).
\\

\noindent Problems of the form \eqref{eq:min_problem} frequently arise in applications such as clustering \cite{Piccialli22}, minimization of truncated convex functions \cite{LiuJian19} (e.g. regression with capped polyhedral norm penalty \cite{Ong13} and regression robust to outliers \cite{Baratt20}). More generally, it encompasses the wide class of Piecewise-Linear-Quadratic optimization programs \cite{Cui20}. 

\paragraph*{Goals} At this stage, we already point out the intrinsic combinatorial nature of \eqref{eq:min_problem}. As thoroughly explained later on in the paragraph dedicated to its geometry, the problem at hand is nonconvex and nonsmooth. Yet, in theory, there exists a deterministic algorithm that solves it globally. Let $\sigma =(\sigma_1,\dots,\sigma_N) \in \bigtimes_{s=1}^{N}\,[n_s]$, we define,
\begin{equation}
\nu(\sigma) := \min_{x \,\in\,\mathcal{X}}\,\Bigg\{F_\sigma(x) := \bar{h}(x) + \frac{1}{N}\,\sum_{s=1}^N\,h^{(s)}_{\sigma_s}(x)\Bigg\} .
\label{eq:F_sigma_sel}
\tag{Oracle}
\end{equation}
The convexity of $F_\sigma$ ensures the convexity of the (sub)problems mobilized in \eqref{eq:F_sigma_sel}.\\We consider that $\nu(\sigma)$ is computed \emph{exactly}, setting numerical accuracy concerns aside. By feeding \eqref{eq:F_sigma_sel} every possible $\sigma$-selection, one is guaranteed to solve \eqref{eq:min_problem} as 
\begin{equation}F^* = \min_{\sigma \,\in \,\bigtimes_{s=1}^N\,[n_s]}\,\nu(\sigma)=\min_{\sigma \,\in \,\bigtimes_{s=1}^N\,[n_s]}\,\min_{x\,\in\,\mathcal{X}}\,F_\sigma(x) = \min_{x\,\in\,\mathcal{X}}\,\min_{\sigma \,\in \,\bigtimes_{s=1}^N\,[n_s]}\,F_\sigma(x). \label{eq:equiv_set} \end{equation}
    \noindent Obviously, this enumerative approach quickly becomes impractical since \begin{equation} n:=\textbf{card}\bigg(\bigtimes_{s=1}^{N}\,[n_s]\bigg) = \Pi_{s=1}^{N}\,n_s \in \mathcal{O}\bigg(\Big(\max_{s=1,\dots,N}\,n_s\Big)^N\bigg) \label{eq:n_scale} \end{equation} grows exponentially with $N$, e.g. $\bar{n}^N$ if $n_s=\bar{n}$ for every $s\in [N]$. Unfortunately, our motivation essentially comes from problems exhibiting a large $N$, e.g. $N \sim 10^2\to 10^4$. In the absence of simplifications, there always exists a problem instance that requires a \emph{scan} of the $n$ possible \emph{pieces} $F_\sigma$ when it comes to global optimization, as we clarify in the paragraph about \eqref{eq:min_problem}'s \textbf{difficulty}. In other words, the naive enumeration is optimal in the worst case. Actually, \cite{Baratt20} points out that finding $F^*$ is at least as hard as the (NP-hard) subset sum problem (SSP), since (SSP) is a particular instance of \eqref{eq:min_problem}. Global optimality out of reach, we are mainly driven by empirical results.\\

\noindent Therefore, our primary goal is to develop heuristic methods, yet covered by sound convergence guarantees towards \emph{critical} points of \eqref{eq:min_problem}, making use of \eqref{eq:F_sigma_sel} to obtain the best solutions possible for a prescribed amount of computational time.\\ For most of the works in nonconvex nonsmooth optimization, \emph{criticality} represents the standard satisfaction requirement. Nevertheless, with all the structure at hand regarding the objective function, one can hope for better. As a secondary goal, we also aim at filtering out mere \emph{critical} points, ultimately striving for \emph{local-minima}. 
\paragraph*{Data driven}
Aside from purely academic toy examples, \eqref{eq:min_problem} problems are typically data driven. Starting from a collection of $N \in \mathbb{N}$ elements from $\mathbb{R}^p$, i.e. $\{\beta^{(s)}\}_{s=1}^{N}$, one tries to minimize a (regularized) average loss over the dataset, i.e. \begin{equation} \min_{x\,\in\,\mathcal{X}}\,\bar{h}(x) + \frac{1}{N}\,\sum_{s=1}^{N}\,h(x\,|\,\beta^{(s)}) \label{eq:SMC_data} \tag{data-SMC} \end{equation} where $h(\cdot\,|\,\beta)$ is the pointwise minimum among $\bar{n}\in \mathbb{N}$ fixed proper convex individual losses $\{h_l(\cdot\,|\,\beta)\}_{l=1}^{\bar{n}}$ that solely depend on $\beta \in \mathbb{R}^p$, i.e. \begin{equation}h(\cdot\,|\,\beta) := \min_{l\, \in \,[\bar{n}]}\,h_l(\cdot\,|\,\beta). \label{eq:pos_choice}\end{equation} In other words, about \eqref{eq:SMC_data}, $n_s = \bar{n}$ for every $s \in [N]$. Moreover, a term $h(\cdot\,|\,\beta^{(s)})$ matches with $h^{(s)}$ whereas for every $l \in [\bar{n}]$, $h_l(\cdot\,|\,\beta^{(s)})$ corresponds to the \emph{component} $h_{l}^{(s)}$. Two main types of situations, both of which we illustrate by a meaningful example, lead to $h(\cdot\,|\,\beta)$ as described in \eqref{eq:pos_choice}.\\ \vspace{-3pt}
\begin{enumerate}
    \item[(I)] \textbf{Multiple choices}.\\
        \begin{itemize} 
    \item In multifacility \cite{Hamacher09} problems, one needs to decide where to install $\bar{n}\in \mathbb{N}$ new \emph{facilities} $x_l \in \mathbb{R}^2$ for every $l \in [\bar{n}]$, taking into account territory constraints, i.e. 
   $x_l \in \mathcal{X}_l$ for every $l \in [\bar{n}]$. The overall decision variable is $x =(x_1,\dots,x_{\bar{n}})$ and $\mathcal{X} = \mathcal{P} \,\cap\,\bigtimes_{l=1}^{\bar{n}}\,\mathcal{X}_l$ where $\mathcal{P} \subseteq \mathbb{R}^{2\cdot\bar{n}}$ implements binding constraints for \emph{facilities}' locations. Here, $\beta \in \mathbb{R}^2$ ($p=2$) represents a customer's position. This customer incurs a proximity loss regarding the $l$-th \emph{facility} set by 
        \begin{equation} h_l(x\,|\,\beta) = ||x_l-\beta||.
            \label{eq:multiple_choices_1}
    \end{equation}
    where $||\cdot||$ depicts a metric over $\mathbb{R}^2$. Hence, his loss becomes the minimum over all the installed \emph{facilities}, i.e. the loss induced by the closest one, 
    $$ h(x\,|\,\beta) = \min_{l\,\in\,[\bar{n}]}\,||x_l-\beta||.$$
    \end{itemize}
\vspace{-8pt}
     \item[(II)] \textbf{Range cuts}.  \\
     \begin{itemize}
         \item In SVM problems, each element $\beta = (\gamma,\bar{\beta})$ includes a class $\gamma \in \{-1,1\}$ and features $\bar{\beta} \in \mathbb{R}^{p-1}$ ($p>1$ here). The goal is to find a vector $x \in \mathbb{R}^{p-1}$ ($d=p-1$) that correlates positively with $\bar{\beta}$ when $\gamma=1$ and negatively otherwise. Hence, a sound loss can be $\ell_\beta(x) = \max\{0,1-\gamma \,\langle \bar{\beta},x\rangle\}$ with lower values of $\ell_\beta$ indicating a good classification of the data point $\beta$. However, in the presence of outliers, the decision vector $x$ can be highly influenced by non-representative data. Therefore, in order to tame the impact of outliers, it is common to \emph{truncate} \cite{LiuJian19} or \emph{clip} \cite{Baratt20} the loss $\ell_{\beta}$ when it exceeds a fixed threshold quantity $\lambda>0$. With $\bar{n} = 2$, $\mathcal{X} = \mathbb{R}^{p-1}$, it comes \,
         $$ h(x\,|\,\beta) = \min\{\max\{0,1-\gamma \,\langle \bar{\beta},x\rangle\},\lambda\},$$
        \begin{equation} h_1(x\,|\,\beta) = \max\{0,1-\gamma \,\langle \bar{\beta},x\rangle\},\quad h_2(x\,|\,\beta) = \lambda.\label{eq:range_cut_1} \end{equation}
    \end{itemize}
 \end{enumerate}
\begin{remark}
Whether regarding \eqref{eq:SMC_data} or the more general \eqref{eq:min_problem}, the \emph{main function} $\bar{h}$ primarily serves as a regularizing term promoting some desired structure about the decision variable $x$. Also, as it is the case for our motivational Example \ref{DC_fitting}, it can capture the common convex part to all the \emph{components}. Indeed, considering a posed \eqref{eq:min_problem} problem. If for every $s \in [N]$ and $l \in [n_s]$, $h^{(s)}_l = \bar{h}^{(s)} + \Delta h^{(s)}_l$  with both $\bar{h}^{(s)}$ and $\Delta h^{(s)}_l$ proper convex then \begin{equation} F^* = \min_{x\,\in\,\mathcal{X}}\,\underbrace{\bar{h}(x)+\frac{1}{N}\,\sum_{s=1}^{N}\,\bar{h}^{(s)}(x)}_{\text{new \emph{main} term}} + \frac{1}{N}\,\sum_{s=1}^{N}\,\min_{l\,\in\,[n_s]}\,\underbrace{\Delta h^{(s)}_l(x)}_{\text{new \emph{components}}}. \label{eq:SMC_shift} \end{equation}
Formulation \eqref{eq:SMC_shift} is preferred to \eqref{eq:min_problem} since it only retains the discriminative part of the \emph{components}. Thereby, it becomes more obvious where on $\mathcal{X}$ a \emph{component}'s value $h^{(s)}_{l_+}(x)$ is smaller than its alternatives, i.e. $h^{(s)}_l(x)$ for $l \in [n_s]\backslash \{l_+\}$.
\end{remark}

\paragraph*{Difficulty}

    By default, problem \eqref{eq:min_problem} is nonconvex. However, consciously or not, sometimes it degenerates into a convex program. It is trivially convex if $n_s=1$ for every term $s \in [N]$. More subtly, it will also be convex if there exists a selection $\sigma \in \bigtimes_{s=1}^{N}\,[n_s]$ such that $F(x) = F_\sigma(x)$ for every $x \in \mathcal{X}$, see Equation \eqref{eq:equiv_set}. Otherwise, the objective $F$, restricted to $\mathcal{X}$, must be expressed as the pointwise minimum of a number $\hat{n}>1$ (and $\hat{n}\leq n$) of convex functions, turning \eqref{eq:min_problem} into a nonconvex program for sure. We note that in some (very) favorable settings with $d\leq 2$, it is shown in \cite{LiuJian19} that one can cope globally with problem \eqref{eq:min_problem} in polynomial-time by solving $\mathcal{O}(N^d)$ convex subproblems, i.e. \eqref{eq:F_sigma_sel}, even though $n_s=2$ for every $s\in [N]$. However, as suggested, their methodology is not applicable for any convex \emph{components} and especially not suited for $d>2$. Even if it were applicable more broadly, we highlight the exponential complexity in the dimension $d$, corroborating the NP-hardness result stated previously in the paragraph about our \textbf{goals}. Actually, \eqref{eq:min_problem} objectives can exhibit an exponential number of \emph{local minima}. To prove this, we propose a family of worst-case instances (Proposition \ref{prop:fa}). They embrace the full complexity of our problem in the sense that for any selection $\sigma \in \bigtimes_{s=1}^{N}\,[n_s]$, there exists a non-empty region $\mathcal{R}$ of $\mathbb{R}^d$ over which $\sigma$ is the unique selection leading to $F(x) = F_\sigma(x)$ for every $x \in \mathcal{R}$. We call these objectives \emph{fully-active} since every piece $F_\sigma$ is useful (i.e. becomes \emph{active} somewhere) in $F$'s enumerative description \eqref{eq:equiv_set}. 
    \begin{proposition}[Fully-active \eqref{eq:min_problem} instance] \label{prop:fa}
    Let $N \in \mathbb{N}$ be fixed and the number of components $n_s\in \mathbb{N}$ be chosen for every $s \in[N]$. For any $d \geq N$, let $\bar{h} : \mathbb{R}^d \to \mathbb{R}$ be a proper convex function. The function $\mathcal{W}$ defined for every $x \in \mathbb{R}^d$ by 
    \begin{equation}
    \mathcal{W}(x) = \bar{h}(x) + \frac{1}{N}\, \sum_{s=1}^{N}\, \min_{l\, \in\, [n_s]}\,\Big\{h^{(s)}_l(x) := l\cdot x_s+\frac{l\cdot (l-1)}{2}\Big\}
    \label{eq:resisting_oracle}
    \end{equation}
    is fully-active. I.e., for every $\sigma \in \bigtimes_{s=1}^{N}\, [n_s]$, $\sigma$ is the unique selection leading to 
    $$\mathcal{W}(x) = \mathcal{W}_{\sigma}(x) := \bar{h}(x) + \frac{1}{N}\,\sum_{s=1}^{N}\,\bigg(\sigma_s\cdot x_s + \frac{\sigma_s\cdot(\sigma_s-1)}{2}\bigg) \quad \forall x \in \bigtimes_{s=1}^{N}\,]-\sigma_s,-\sigma_s+1[.$$
    
    \end{proposition} 
    \begin{proof}
   We notice the independence of terms $h^{(s)} = \min_{l\,\in\,[n_s]}\,h^{(s)}_l$ for every $s \in [N]$. At any $x \in \mathbb{R}^d$ and $(s_1,s_2) \in [N]^2$, this implies that an \emph{active} \emph{component} $h^{(s_1)}_{\sigma_{s_1}}$, i.e., $$h^{(s_1)}_{\sigma_{s_1}}(x) \leq \min_{l_1 \in [n_{s_{1}}]}\,h^{(s_1)}_{l_1}(x)$$ has no influence on the possible activation of $h^{(s_2)}_{\sigma_{s_2}}$ regarding the term $h^{(s_2)}$. Indeed, the $s$-th dimension of vector $x$ defines on its own which \emph{component} among $h^{(s)}_{1},\dots,h^{(s)}_{n_{s}}$ will be \emph{active}. Hence, every combination $\sigma \in \bigtimes_{s=1}^{N}\,[n_s]$ can achieve 
   $\mathcal{W}(x) = \mathcal{W}_{\sigma}(x)$, conditional on the fact that each $h^{(s)}_{\sigma_{s}}$ might become \emph{active} for some $x_s \in \mathbb{R}$.\\ One can observe that this is the case, i.e. $h^{(s)}_{\sigma_{s}}$ is \emph{active} for every $x_s \in [-\sigma_s,-\sigma_s+1]$.\\Finally, on $]-\sigma_s,-\sigma_s+1[$, $h^{(s)}_{\sigma_s}$ is the only \emph{active component} for the term $h^{(s)}$.
    \end{proof}
    \noindent To summarize, in what concerns \emph{fully-active} objectives, globally minimizing $F$ using \eqref{eq:F_sigma_sel} suggests that, indeed, one should solve $n$ convex subproblems, possibly leading to as many \emph{local-minima} among which up to $n-1$ are spurious, i.e. non-global.

\paragraph*{Geometry \& Structure} Let us first underline a geometrical consideration particular to problem \eqref{eq:min_problem}, thereby shared by convex problems. We denote for every $s \in [N]$, 
\begin{equation} 
\mathcal{A}^{(s)} : \mathbb{R}^d \rightrightarrows \{1,\dots,n_s\},\hspace{5pt} x\to \mathcal{A}^{(s)}(x) := \big\{l \,|\, h^{(s)}(x) = h_l^{(s)}(x)\big\}. 
\label{eq:active}
\end{equation}
Let $x \in \mathbb{R}^d$, $\mathcal{A}^{(s)}(x)$ collects which \emph{components} $h^{(s)}_{l}$ are \emph{active} at $x$ for the term $h^{(s)}$.
If we assume that the \emph{main function} and \emph{component functions} are twice continuously differentiable on their respective domains, then for every $x \in \text{int}(\text{dom}\,\bar{h})$ fulfilling $$\mathcal{A}^{(s)}(x) = \{\sigma_s\}\quad \forall s \in[N],$$ the matrices $\nabla^2 h^{(s)}(x) = \nabla^2 h_{\sigma_s}(x)$ are well-defined and positive semidefinite so that 
\begin{equation}\nabla^2 F(x)= \nabla^2 \bar{h}(x) + \frac{1}{N}\sum_{s=1}^{N}\nabla^2 h_{\sigma_s}^{(s)}(x) \succeq \mathbf{0}_{d\times d}.
\label{eq:no_curvature}
\end{equation}
There exists no point in the effective domain of $F$ where, when existing, its Hessian matrix admits a strictly negative eigenvalue (no negative curvature). More generally, even when the \emph{components} and/or the \emph{main function} fail to be differentiable on their domains, \eqref{eq:min_problem} objectives are convex on subregions of $\mathcal{X}$ that we define now.\\ Every $\sigma$-selection induces two (possibly nonconvex or empty) regions\vspace{-5pt}\begin{align}
\mathcal{R}(\sigma) &:= \big\{x \in \mathcal{X}\,|\,h_{\sigma_s}^{(s)}(x) \leq h_l^{(s)}(x) \hspace{5pt}\forall s \in [N],\hspace{2pt}\forall l \in [n_s]\big\},
\label{eq:packing}\\
\overset{\circ}{\mathcal{R}}(\sigma) &:= \big\{x \in \mathcal{X}\,|\,h_{\sigma_s}^{(s)}(x) < h_l^{(s)}(x) \hspace{5pt}\forall s \in [N],\hspace{2pt}\forall l \in [n_s]\backslash \{\sigma_s\}\big\}.
\label{eq:packing_open}
\vspace{-15pt}
\end{align}
On $\mathcal{R}(\sigma)$, $F = F_\sigma$ and every \emph{stationary point} $\hat{x} \in \overset{\circ}{\mathcal{R}}(\sigma)$, i.e. $\mathbf{0}_d\in\partial F_\sigma(\hat{x})+\mathcal{N}_{\mathcal{X}}(\hat{x})$, must be a \emph{local-minimum} of $F$ on the subregion $\mathcal{R}(\sigma)$, there is no interior \emph{local-maximum}. 
\newpage \noindent Indeed, by convexity of $F_\sigma$, $\hat{x}$ would then be a global minimum of $F_\sigma$ over any subset $\mathcal{S} \subseteq \mathbb{R}^d$ in which it is contained. Moreover, by continuity of the \emph{components}, there must exist a radius $\hat{\alpha}>0$ for which every $x \in \mathcal{X} \,\cap\,\mathbb{B}_2(\hat{x}\,;\,\hat{\alpha})$ satisfies $h_{\sigma_s}^{(s)}(x) \leq h_l^{(s)}(x)$ for every $s \in [N]$ and $l \in [n_s]$. On $\mathcal{S} = \mathcal{X} \,\cap\,\mathbb{B}_2(\hat{x}\,;\,\hat{\alpha})$, $F=F_\sigma$ is minimized at $\hat{x}$.

\begin{figure}[h]
\vspace{-20pt}
\centering
\includegraphics[width=0.8\textwidth]{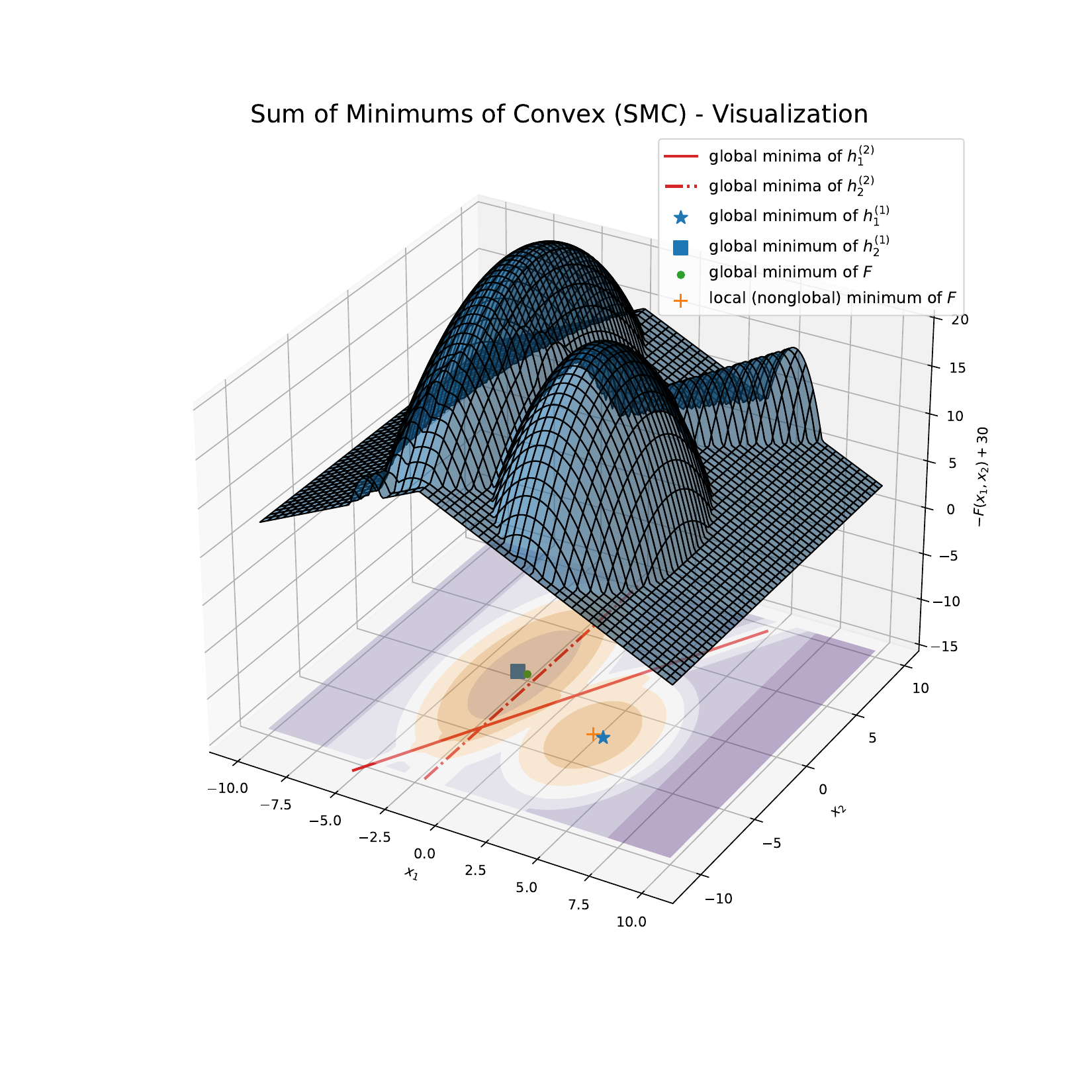}
\vspace{-50pt}
\caption{$F(x_1,x_2) = \min\{(x_1-3)^2+\frac{1}{3}(x_2+3)^2,(x_1+3)^2+\frac{1}{6}x_2^2,15\}+\min\{(x_2-2 x_1 + 1)^2,|x_1+2|\}$}
\label{fig:SMC_intro}
\end{figure}

\remark On Figure \ref{fig:SMC_intro}, we have displayed the graph of $-F$ as well as contour lines for an example with $d=2$, $N=2$ ($n_1=3$, $n_2=2$), $\mathcal{X} = [-10,10]^2$ and $\bar{h}(x)=0$. \\ $F$ is nonconvex on substantial portions of the basic feasible domain $\mathcal{X}$, i.e. there exists $x,x_+ \in \mathcal{X}$ and $q \in [0,1]$ such that $F(q\cdot x+(1-q)\cdot x_+) > q\cdot F(x) + (1-q)\cdot F(x_+)$. However, one can spot a global optimum of $F$ as being the global optimum of a convex function $(1/2)\cdot(h_2^{(1)}+h_2^{(2)})$, i.e. $F^* = \min_{x\,\in\,\mathcal{X}}\,F_{(2,2)}(x)$ and $\mathcal{X}^* = \{-(\frac{5}{2},0)\}$.\\

\noindent To conclude this paragraph, akin to \cite{Bagirov08},we also mention the difference-of-convex (DC) nature of \eqref{eq:min_problem}. Indeed, let $f_\star$ be convex for $\star \in \{1,2\}$. For any $x\in \mathbb{R}^d$,\begin{equation}
F(x) = \underbrace{\bar{h}(x) + \frac{1}{N}\,\sum_{s=1}^{N}\,\sum_{l\,\in\,[n_s]}\,h^{(s)}_l(x)}_{f_1(x)} - \underbrace{\frac{1}{N}\,\max_{\tilde{l}\,\in\,[n_s]}\,\sum_{l \,\in\,[n_s]\backslash\{\tilde{l}\}}\,h^{(s)}_l(x)}_{f_2(x)}.
\label{eq:DC_decomposition}
\end{equation}
This decomposition will help us to define the concept of \emph{critical point} of \eqref{eq:min_problem}. Moreover, \eqref{eq:DC_decomposition} allows the use of (DC) algorithms (see e.g. \cite{LeThi18}) as a possible baseline.
\newpage 
\begin{remark}
As suggested by Equation \eqref{eq:equiv_set}, there must exist a global minimum of $F$ over $\mathcal{X}$ that is the global minimum of a $\sigma^*$-selection over $\mathcal{X}$, i.e. $\sigma^* \in \bigtimes_{s=1}^{N}\,[n_s]$ with 
$$F^* = \min_{x\,\in\,\mathcal{X}}\,\bar{h}(x) + \frac{1}{N}\sum_{s=1}^{N}\,h^{(s)}_{\sigma_s^*}(x).$$ Nevertheless, we emphasize that there might exist global minima of $F$ that are not the result of the minimization of $\sigma$-selection. For instance, for every $x \in \mathbb{R}$ let
\begin{equation}
 F(x) =  -\frac{1}{4} + \frac{1}{2}\Big(\min\{(x-1)^2,1/2\} + \min\{x^2,1/2\}\Big).
 \label{eq:other_glob_min}
\end{equation}$F$ fits \eqref{eq:min_problem} with $\bar{h}(x)=-1/4$, $d=1$, $N=2$ ($n_1=2$, $n_2=2$). By inspection, $F^* = F(1) = F(0) = 0$. Yet, one also has $F(1/2)=0=F^*$ but $1/2$ is neither the minimizer of $h^{(1)}_1(x)=(x-1)^2$ nor $h^{(1)}_2(x)=x^2$. This toy example also highlights that there might be multiple optimal $\sigma^*$-selections, i.e. here $\sigma = 1$ and $\sigma = 2$.
\end{remark}
\paragraph*{Motivation}
Since the brute force enumeration approach \eqref{eq:equiv_set} quickly becomes impractical as $N$ grows ($n=\Pi_{s=1}^N\,n_s$,  recalling Eq. \eqref{eq:n_scale}), one question naturally emerges. 
\begin{center}
    \textit{Can we exploit the structure in \eqref{eq:min_problem} to (i) implement an efficient/scalable local-search heuristic, (ii) possibly endowed with some local optimality guarantees ?} \vspace{10pt}
\end{center}

About $(i)$, we answer by the affirmative. Rewriting \eqref{eq:min_problem} as a specific bi-convex problem, the possibility to apply alternating minimization (\texttt{AM}) procedures is unveiled. One only requires \emph{off-the-shelf} numerical tools, e.g.\@ \texttt{CVXPY} \cite{Diamond16}, to solve convex subproblems, i.e. \eqref{eq:F_sigma_sel}, and the ability to evaluate \emph{component functions} at any produced iterate. Relying on both these oracles, we come up with a family of new local-search methods, which we call relaxed alternating minimization (\texttt{RAM}), empirically providing better solutions than (\texttt{AM}) on our benchmark tests. (\texttt{RAM}), but also (\texttt{AM}) as a limiting case, scale well with $N$ in the sense that we prove their convergence in $O(\delta^{-1})$ iterations, each of which involving one call to both oracles, to reach any level $\delta>0$ of approximate optimality measure. This latter coincides with \emph{criticality} for $\delta=0$. We precise the notion of \emph{criticality} in Definition \ref{def:cp} (Section \ref{sec:prelim}).
\\ 

About $(ii)$, we nuance our answer. We leverage a new lemma tailored for \eqref{eq:min_problem} (Lemma \ref{lemma_lo_smc}), from which we deduce sufficient conditions (Corollary \ref{coro:restart_check}) for a candidate $\hat{x}\in\mathcal{X}$ to be \emph{locally optimal}. We show that it amounts to solving a reduced scale instance of the original \eqref{eq:min_problem} problem on a local neighbourhood of $\hat{x}$. This scale relates to the number $\Pi_{s=1}^{N}\,\big|\mathcal{A}^{(s)}(\hat{x})\big|$ (recall Eq. \eqref{eq:active}), called \emph{degeneracy factor}. \begin{itemize}
\item If the \emph{degeneracy factor} is small enough, one can implement a direct exhaustive scanning procedure to check for \emph{local optimality}. \\
\item Otherwise, we propose to use a mixed-integer formulation for \eqref{eq:min_problem} problems whose continuous relaxations are convex. Hopefully, the existence of dedicated \textit{branch-and-cut} algorithms (see e.g. \cite{Kronqvist18}) might help in practice to solve the aforementioned reduced scale instance, also allowing to check for \emph{local optimality}. \\
\end{itemize}

\noindent That being stated, the reduced scale problem might be as hard to solve as the original one in some unlucky circumstances, e.g. the order of the \emph{degeneracy factor} is $\mathcal{O}(n)$. Without further assumptions on the feasible set $\mathcal{X}$ and objective $F$, guaranteeing \emph{local optimality} remains difficult from a theoretical point of view. \\ \\
In order for the reader to get further insight, we present hereafter our most exciting example, instance of \eqref{eq:SMC_data}, closely related to \textbf{multiple choice} type.

\example (\emph{Piecewise-Linear $L_1$-Regression}, \cite{Bagirov22}) \label{DC_fitting} Due to their appealing universal approximator property \cite{Relu20}, piecewise-linear models provide an excellent trade-off between complexity and performance. Considering a $L_1$-regression task, given a dataset of training examples from $\mathbb{R}^{p+1}$, i.e. $\{(\gamma^{(s)},\bar{\beta}^{(s)})\}_{s=1}^{N}$, one would like build a piece-wise linear model $\Delta \ell(\cdot\,;\,x)$ that predicts the most accurately possible $\gamma^{(s)} \in \mathbb{R}$ as $\Delta \ell(\bar{\beta}^{(s)}\,;\,x)$ by minimizing its \emph{mean absolute deviation} regarding the examples, i.e. \begin{equation}
\min_{x \,\in \,\mathcal{X}}\,\frac{1}{N}\,\sum_{s=1}^{N}\,\big|\gamma^{(s)} - \Delta \ell(\bar{\beta}^{(s)}\,; \,x)\big|.
\label{eq:raw_DC_fitting}
\end{equation}
The prediction of the piecewise-linear model $\Delta \ell(\cdot\,;\,x)$ is defined for every $\bar{\beta} \in \mathbb{R}^p$ as
\begin{equation}
\Delta \ell(\bar{\beta} \, ;\, x) = \underbrace{\max_{e_1\,\in\,[B_1]}\,\big\{\langle \bar{\beta}, x_{e_1}^{(1)} \rangle\big\}}_{\ell_1(\bar{\beta}\,;\,x)} - \underbrace{\max_{e_2\,\in\,[B_2]}\,\big\{\langle \bar{\beta}, x_{e_2}^{(2)} \rangle\big\}}_{\ell_2(\bar{\beta}\,;\,x)}
\label{eq:DC_fit_model}
\end{equation}
The model is parametrized by a vector of coefficients $x$ of size $p\cdot(B_1+B_2)$, i.e. $$x = \big(x_1^{(1)},\dots,x_{B_1}^{(1)},x_1^{(2)},\dots,x_{B_2}^{(2)}\big).$$
Usually, one sets $\mathcal{X} = \big\{u\in\mathbb{R}^p\,|\,||u||_{\infty} \leq L\big\}^{B_1 + B_2}$ to ensure the $2L$ Lipschitz continuity of $\Delta \ell(\cdot\,;\,x)$. As we prove in Appendix \ref{dev_SMC}, problem \eqref{eq:raw_DC_fitting} fits in \eqref{eq:SMC_data} with 
  \begin{align}
  \bar{h}(x) = &\, \frac{1}{N}\,\sum_{s=1}^N 
 \max\Big\{\gamma^{(s)} + \ell_2(\bar{\beta}^{(s)}\,;\,x),\ell_1(\bar{\beta}^{(s)}\,;\,x)\Big\} \nonumber \\&\,+ \max\Big\{-\gamma^{(s)} + \ell_1(\bar{\beta}^{(s)}\,;\,x),\ell_2(\bar{\beta}^{(s)}\,;\,x)\Big\}, \nonumber \\
h_l(x\,|\,\bar{\beta}) =& -\big\langle \bar{\beta}, x^{(1)}_{e_1(l)}+x^{(2)}_{e_2(l)}\big\rangle,  \nonumber
\end{align}
as well as $\bar{n} = B_1 \cdot B_2$, $e_1(l)  =  \lceil l/B_2\rceil$ and finally $e_2(l)  =  1+\big(l-1) \,\textbf{mod}\,B_2$.

\subsection{Related work}
The present framework entails many optimization problems of interest. Among them, the \textbf{multiple choices} type of \eqref{eq:SMC_data}, seems to have attracted the most attention. To the best of our knowledge, authors of \cite{Rubinov05} first used the designation \emph{sum of minima of convex} functions. The scope of their paper is to run the (zero-order) discrete gradient method (see e.g. \cite{Bagirov08} for an updated version) on tasks like generalized clustering as well as Bradley-Mangasarian approximation of finite sets \cite{Ghosh05}. They also tried the hybrid discrete gradient/simulated annealing method from \cite{Bagirov03} as heuristic to escape from bad sub-optimal valleys. Roughly speaking, this approach only relies on function evaluations and does not fully take advantage of (DC)'s structure nor \eqref{eq:min_problem}'s. In contrast, \cite{Artacho22}, and later \cite{Tran23}, devised algorithms fundamentally based on (DC)'s decomposition \eqref{eq:DC_decomposition} to tackle clustering with constraints. Both papers displayed good empirical performances of their algorithms and proved that every converging subsequence of their generated iterates leads to an accumulation point that is \emph{critical}. Closer to our framework, authors in \cite{LiuJian19} address \eqref{eq:min_problem} problems of \textbf{range cuts} type, e.g.\@ signal-restoration, clipped linear regression or more generally, statistical learning subject to outlier detection. Their methodology aims at solving \eqref{eq:min_problem} globally thanks to a spatial segmentation of the ambient space $\mathcal{X}$ into regions $\mathcal{R}$,
 as in \eqref{eq:packing}. However, as soon as $d\geq 3$, this method hardly remains useful in practice. Moreover, one needs an efficient tool to enumerate all the intersections of $\mathcal{R}$-boundaries, a very complicated task in the absence of closed-forms, even when $d=2$. \\ \\The terminology of \emph{min-convex} also appears in \cite{Dao19} where it is shown that the the proximal operator of a \eqref{eq:min_problem} objective, which can be cast into a \eqref{eq:min_problem} problem, satisfies a property of \textit{union averaged nonexpansiveness}. Nevertheless, authors do not provide specific algorithms\footnote{Other than the naive \texttt{enumeration approach} \eqref{eq:equiv_set}, not suitable when $N$ is medium to large.} to compute the associated proximal mapping and their findings remain theoretical in what concerns our applications, e.g. Example \ref{DC_fitting}.\\ Very recently, \cite{Ding24} proposed a new inertial variant of a generalized Lloyd's algorithm\footnote{In the context of clustering, Lloyd's algorithm corresponds to plain alternating minimization (\texttt{AM}).}, named \texttt{inLloyd}, applicable on \eqref{eq:min_problem} instances that are of \textbf{multiple choices} type, $$x =(x_1,\dots,x_{\bar{n}})\to h^{(s)}(x) = \min_{l\,\in\,[\bar{n}]}\,H^{(s)}(x_l)\quad \forall s \in [N],$$
 with $H^{(s)}$ strongly convex. In this precise situation, in the spirit of \texttt{kmeans++} \cite{Arthur07}, they come up with a randomized strategy to set up the initial values of the replicas $\{x_l\}_{l=1}^{\bar{n}}$ with proven quality guarantees. Unfortunately, we were not able to extend their methodology to initialize instances of \eqref{eq:min_problem} in the general case. \\ \\The present work places itself in the direct continuity of \cite{Baratt20}. Just like \cite{LiuJian19}, this latter work is devoted to \eqref{eq:min_problem} instances for which every term in the sum is the pointwise minimum between a convex function and a simple constant value $\lambda$ (i.e. $n_s=2$ for every $s \in [N]$). Therein, authors call such terms \emph{clipped convex functions}.
 \begin{equation}
F^*_{\text{clipped}} = \min_{x\, \in\, \mathcal{X}}\, \bar{h}(x) + \frac{1}{N}\, \sum_{s=1}^N\, \min\{h_1^{(s)}(x), \lambda\}.
\label{eq:clipped_convex_simple}
\tag{SCC}
\end{equation}
Their contribution is twofold. First, they highlight a bi-convex reformulation of \eqref{eq:clipped_convex_simple}. \begin{equation}
F^*_{\text{clipped}} = \min_{Q= (q^{(1)},\dots,q^{(N)})\,\in\, [0,1]^N}\,\underbrace{\min_{x\, \in\, \mathcal{X}}\, \bar{h}(x) + \frac{1}{N}\, \sum_{s=1}^N\, q^{(s)}\cdot h_1^{(s)}(x) +(1-q^{(s)})\cdot \lambda}_{\bar{F}^*_{|Q}}
\label{eq:clipped_convex_reformulation}
\end{equation}
Out of \eqref{eq:clipped_convex_reformulation}, one usually alternates between an exact minimization with respect to $Q$ and $x$ (\texttt{AM}). However, authors of \cite{Baratt20} suggest to perform inexact $Q$-updates, 
\begin{equation}
q^{(s)}_+ = \Pi_{[0,1]^N}\big(q^{(s)}-\kappa\cdot \text{sign}(h_1^{(s)}(x)-\lambda)\big) \quad \forall s \in [N], \label{eq:iAM_original}
\end{equation}
for a fixed $\kappa \in (0,1]$ ($\kappa=1$ implementing usual \texttt{AM}). We note that where $\bar{F}^*_{|Q}$ is locally differentiable, one can interpret updates \eqref{eq:iAM_original} as projected signed gradient steps \cite{Moulay19}. 

\noindent Second, they devise a mixed-integer convex program (MICP) equivalent to \eqref{eq:clipped_convex_simple}, 
\begin{align} F^*_{\text{clipped}} = \min_{x\,\in\,\mathcal{X}} \quad &\, \frac{1}{N}\,  \sum_{s=1}^N\,(\bar{h})^\dagger(x-z^{(s)},1-q^{(s)})+(\bar{h})^\dagger(z^{(s)},q^{(s)}) \label{eq:clipped_micp} \\ & +\frac{1}{N}\,  \sum_{s=1}^N\,(h_1^ {(s)})^\dagger(z^{(s)},q^{(s)})+(1-q^{(s)})\cdot \lambda \nonumber\\
\textrm{s.t.} \quad & q^{(s)} \in \{0,1\},\, z^{(s)}\in\mathcal{X}\quad \forall s \in [N] \nonumber 
\end{align}
where a (convex) perspective function $(h)^\dagger$ is obtained from any closed convex\footnote{We recall that  \emph{components} as well as the \emph{main function} are all assumed proper, closed and convex.} $h$ as,
\begin{equation}
(h)^\dagger : \mathbb{R}^d \times \mathbb{R}_+,\quad (x,q) \to (h)^\dagger(x,q) = \begin{cases} 
q\cdot h(x/q) & q>0 \\
0& q=0,\hspace{2pt}x=\mathbf{0}_d\\
+ \infty & \text{otherwise}.
\end{cases}
\label{eq:perspective_redef}
\end{equation} 
To obtain \eqref{eq:perspective_redef} and the property that $(h)^\dagger$ is well-defined and convex, it is assumed iin \cite{Baratt20} that $\bar{h}$ is superlinear and $\mathbf{0}_d \in \text{dom}(\bar{h})$. Note these requirements slightly decrease the generality of the approach. As the authors noticed, aside from being intrinsically hard as a (MICP), formulation \eqref{eq:clipped_micp} exhibits $N\cdot d$ extra continuous variables $\{z^{(s)}\}_{s=1}^{N}$ (i.e. compared to the raw problem \eqref{eq:clipped_convex_reformulation}) in addition to the seemingly unavoidable $N$ binary variables $\{q^{(s)}\}_{s=1}^{N}$. From a practical aspect, this renders the resolution of \eqref{eq:clipped_micp} quite time consuming (or impossible) even if the problem is only restricted to a smaller domain portion $\mathcal{S} \subset \mathcal{X}$. As argued in \cite{Baratt20}, solving \eqref{eq:clipped_micp} while relaxing integrality constraints on $Q$ yields a lower-bound on $F^*_{\text{clipped}}$ and a possible starting iterate in the $x$-coordinates. In this case, to circumvent the dimensionality drawback, it is possible to apply \texttt{ADMM} procedures  \cite{Boyd11}  to deal with a \emph{consensus form} of the continuous relaxation of \eqref{eq:clipped_micp}. Both the aforementioned contributions initiated the present work. Indeed, our first goal had been to provide extensions of models \eqref{eq:clipped_convex_reformulation} and \eqref{eq:clipped_micp} to fit in the general \eqref{eq:min_problem} picture. \\

\noindent Finally, we also mention that a last related paper \cite{Zuo23} appeared in the literature. Therein, a first step towards more versatility is undertaken.  Indeed, the authors consider again a \eqref{eq:min_problem} instance for which $n_s=2$ for every $s\in [N]$ but this time, $h^{(s)}_2$ needs not to be a constant. However, they implicitly assume that $\bar{h} \in \mathcal{C}^{1}(\mathbb{R}^d)$. Their methodology works as follows. By approximating each pointwise minimum in the \eqref{eq:min_problem} sum by its own smooth minimum surrogate, they leverage the power of \texttt{(L-)BFGS} (unconstrained case) or any black-box method dedicated to constrained smooth nonconvex optimization. Overall, according to the numerical experiments conducted in \cite{Zuo23}, this new approach performed equally as well as \cite{Baratt20} on similar problems.  
\subsection{Contributions}
Now that the framework of \eqref{eq:min_problem} has been thoroughly introduced, we summarize below our own specific contributions. We emphasize that all our numerical experiments (and, soon, a free to use new dedicated \texttt{CVXPY} bundle called \texttt{smc}, improving upon \href{https://github.com/cvxgrp/sccf}{\texttt{sccf}} from \cite{Baratt20}), are available on our \href{https://github.com/guiguiom/SM_public}{GitHub} repository. Focusing on mathematical optimization contributions, we point out the following elements:  \\
\vspace{-10pt}
 \begin{itemize} 
     \item First, we detail how and when \eqref{eq:min_problem} can be rewritten as a \emph{big-M} (MICP) (Section \ref{sec:smc_reformulations}). In stark contrast with the perspective reformulation \eqref{eq:clipped_micp} of \cite{Baratt20}, for which we also provide an extension with arbitrary $n_s\in \mathbb{N}$ for every $s\in[N]$ in \eqref{eq:gen_perspective}, our (MICP) does not involve $\sum_{s=1}^{N}\,n_s$ copies of the decision vector $x$, hence $d$ (i.e. the dimension of $x$) times as many extra variables.\\
     \item Second, building upon (Proposition 1.4) from \cite{VanDessel24}, we restrict our global (MICP) formulation on local neighbourhoods (Section \ref{subsec:local_micp}). As a byproduct of our local (MICP) formulation, we state \emph{equivalent} \emph{local optimality} conditions (Lemma \ref{lemma_lo_smc}). From these latter, useful \emph{sufficient} conditions (Corollary \ref{coro:restart_check}) are derived to possibly go beyond \emph{criticality}, the usual standard in (DC) optimization. \\
     \item Third, we expose the bi-convex formulation of general \eqref{eq:min_problem} (Section \ref{subsec:global_bic}). We review plain alternating minimization (\texttt{AM}) applied on this formulation to underpin another new lemma (Lemma \ref{lemma:criticality}) describing \emph{criticality} sufficient conditions. This lemma together with sufficient descent conditions will serve as backbones to develop our new family of heuristics, i.e. relaxed alternating minimization (\texttt{r-AM}) methods. We prove that every accumulation point of the main sequence of \texttt{r-AM} iterates is \emph{critical} for problem \eqref{eq:min_problem} (Theorem \ref{theorem_ram}).\\
     \item Finally, we display the results of numerical experiments (Section \ref{sec:numerical_exp}), demonstrating the practical efficiency of two \texttt{r-AM} methods of ours when tackling \eqref{eq:min_problem} on \emph{piecewise-linear regression} or \emph{facility location} problems. In addition, thanks to the aforementioned \emph{sufficient} conditions (Corollary \ref{coro:restart_check}), we show that we were sometimes successful in certifying \emph{local optimality} of \texttt{r-AM} iterates (Section \ref{subsec:lo_certif}).
    \end{itemize}
    \vspace{-10pt}
\subsection{Preliminaries}\label{sec:prelim}
For any $d\in \mathbb{N}$,
we endow $\mathbb{R}^{d}$ with the usual dot product and the classical Euclidean norm. 
For any $p$-norm $||\cdot||_p$ ($p \geq 1$) on $\mathbb{R}^d$, $R \geq 0$ and $\bar{x} \in \mathbb{R}^d$, we define the corresponding ball
\vspace{-10pt}
\begin{equation}
\mathbb{B}_{p}(\bar{x}\,;\,R) := \big\{x \in \mathbb{R}^d\,|\, ||x-\bar{x}||_p \leq R \big\}.
\end{equation}
\noindent The notation $\mathbf{1}_d$ (respectively $\mathbf{0}_d$) represents the \emph{all-ones} (respectively \emph{all-zeroes}) vector of size $d$. Let $h : \mathbb{R}^d \to \mathbb{R} \cup \{\infty\}$, the \emph{domain} of $h$, noted $\text{dom} \,h$, is defined as
\begin{equation}
\text{dom} \,h := \{x\in\mathbb{R}^d\,|\,h(x)<\infty\}.
\label{def:domain}
\end{equation}
\begin{definition}[Critical point]\label{def:cp}
 Let $\mathcal{X} \subseteq \mathbb{R}^d$ be a convex set and $f_1,f_2$ two proper, convex and lower semicontinuous functions. A point $\hat{x} \in \mathcal{X}$ is called \emph{critical} for \begin{equation*}F^*_{\text{DC}}=\min_{x\,\in\,\mathcal{X}}\,f_1(x)-f_2(x) \label{eq:DC_straight} \end{equation*} \vspace{-5pt}if and only if there exists\footnote{$\mathcal{N}_{\mathcal{X}}(x)$ depicts the \emph{normal cone} of $\mathcal{X}$ at $x$.} $\hat{v} \in \mathcal{N}_{\mathcal{X}}(\hat{x})$, $g_1 \in \partial f_1(\hat{x})$ and $g_2 \in \partial f_2(\hat{x})$ such that 
\begin{equation}
\hat{v} + g_1 - g_2 = \mathbf{0}_d.
\label{eq:critical_DC}
\end{equation}
\end{definition}
\begin{definition}[Local optimality] \label{def:locopt} A point $\hat{x} \in \mathcal{X}$ is called a \emph{local minimizer} of \eqref{eq:min_problem} if $F(\hat{x})<\infty$, i.e. $\hat{x} \in \text{dom}\,F$, and there exists a radius $\alpha>0$ such that 
\begin{equation}
 F(\hat{x}) = \min_{x\, \in \,\mathcal{X} \,\cap \,\mathbb{B}_2(\hat{x},\alpha)}\,F(x).
\end{equation}
\end{definition}
\noindent Now, we extend the informal definition of \eqref{eq:active} to capture not only active \emph{components} at $x \in \mathcal{X}$ but also $\rho$-close to active ones. 
\begin{definition}[$\rho$-active sets] \label{def:active_set} Let $\rho \in [0,1]$. For every $s\in [N]$, we call the set \begin{equation}
\mathcal{A}^{(s)}_{\rho}(x) := \Bigg\{ \tilde{l}\,\in \,[n_s]\,\bigg|\, h^{(s)}_{\tilde{l}}(x) -h^{(s)}(x) \leq \rho \cdot \bigg(\max_{l\,\in\,[n_s]}\,h^{(s)}_l(x) - h^{(s)}(x)\bigg)\Bigg\}
\label{eq:as_min}
\end{equation}
the ($\rho$ -)\emph{active set} at $x \in \mathcal{X}$ of the $s$-th term.
\end{definition}
\noindent By convention, $\mathcal{A}^{(s)}(x) = \mathcal{A}^{(s)}_{0}(x)$ is simply called the \emph{active set} at $x$ of the $s$-th term.
\begin{proposition}
\label{local_non_expansiveness}
Let $\hat{x}\in\mathcal{X}$. There exists  $\alpha >0$ such that for every $x \in \mathcal{X} \cap \mathbb{B}_2(\hat{x}\,;\,\alpha)$, \begin{equation}\mathcal{A}^{(s)}(x) \subseteq \mathcal{A}^{(s)}_\rho(\hat{x}) \quad\forall \rho\in[0,1],\quad \forall s \in [N]\label{eq:stable_active_sets}
\end{equation}
\end{proposition}
\begin{proof}
According to (\cite{VanDessel24}, Proposition 1.4) that applies verbatim, for every $s \in [N]$ there exists $\alpha^{(s)}>0$ such that $\mathcal{A}^{(s)}(x)\subseteq \mathcal{A}^{(s)}_\rho(\hat{x})$ for every $x \in \mathcal{X}\,\cap\,\mathbb{B}_2(\hat{x}\,;\,\alpha^{(s)})$. \\Then, \eqref{eq:stable_active_sets} simply holds by taking $\alpha = \min_{s\,\in\,[n]}\,\alpha^{(s)}$.
\end{proof}
\begin{definition}[Standard simplex]
Let $\bar{n} \in \mathbb{N}$. The $\bar{n}$-\emph{standard simplex} reads
\begin{equation}
\Delta^{\bar{n}} := \{q \in \mathbb{R}^{\bar{n}}\,|\,\langle\mathbf{1}_{\bar{n}},q\rangle = 1,\, q\geq \mathbf{0}_{\bar{n}}\}.
\end{equation}
\end{definition}
\begin{assumption}
We end this section with the following blanket assumption. 
\begin{itemize}
\item One has access to a deterministic black-box algorithm that will return for any fixed $Q=(q^{(1)},\dots,q^{(N)}) \in \bigtimes_{s=1}^{N}\,\Delta^{(n_s)}$ the same output \begin{equation} x^*_{|Q} \in \arg \min_{x\,\in\,\mathcal{X}}\,\bar{h}(x) + \frac{1}{N}\sum_{s=1}^{N}\sum_{l=1}^{n_s}\,q^{(s)}_l\, h^{(s)}_l(x). \label{eq:full_oracle_problem} \end{equation}
\end{itemize}
\label{A1}
\end{assumption}
\noindent Note that if every $q^{(s)}$ is a vertex of $\Delta^{(n_s)}$, i.e. there exists $\sigma_s \in [n_s]$ such that $q^{(s)}_{\sigma_s}=1$ and $q^{(s)}_l = 0$ for every $l \in [n_s]\backslash \{\sigma_s\}$, \eqref{eq:full_oracle_problem} solves \eqref{eq:F_sigma_sel} with $\sigma = (\sigma_1,\dots,\sigma_N)$. 


\newpage 

\section{Equivalent representations of (SMC)}

\label{sec:smc_reformulations}

In this section we aim at explaining the core procedure to reformulate problem \eqref{eq:min_problem} as a \emph{big-M} (MICP). We first need to introduce a localized definition of \emph{big-M} bounds in what concerns \eqref{eq:min_problem}. That is, given a zone $\mathcal{S} \subseteq \mathbb{R}^d$, one needs a finite bound on the maximal difference of value (on $\mathcal{S}$) achievable by every pair of \emph{components}. These bounds will come into play to provide an epigraph representation of objective $F$,  restricted to $\mathcal{S}$,  using convex constraints and binary variables. To construct an equivalent formulation of \eqref{eq:min_problem}, one needs a representation of $F$ that is, \emph{at least}, valid on the basic feasible domain $\mathcal{X}$, hence the necessity of Assumption \ref{A2}. However, when one only wants a local representation, computing these bounds on a small $\mathcal{S} \not\supseteq \mathcal{X}$ will prove useful (see Section \ref{subsec:lo_certif}).  
That being said, this new assumption is satisfied for our applications, including the interesting \emph{piecewise-linear regression} (Example \ref{DC_fitting}). 

\begin{definition}[$\mathcal{S}$-bounds]
\label{def:s_bounds}
Let $\mathcal{S} \subseteq \mathbb{R}^d.$ We call $\mathcal{S}$-bounds for a given \eqref{eq:min_problem} values $M^{(s)}_{l_+,l}<\infty$ for every $s \in [N]$ and $(l_+,l)\in[n_s]^2$ such that $M^{(s)}_{l,l}=0=M^{(s)}_{l_+,l_+}$ and
\begin{equation} M^{(s)}_{l_+,l} \geq \max_{u \,\in\, \mathcal{S}}\, h^{(s)}_{l_+}(u)-h^{(s)}_{l}(u) \quad \text{s.t.} \quad h^{(s)}_{l}(u)= h^{(s)}(u).
\label{eq:bound_problem_M}
\end{equation}
By convention, if $\{u \, \in \, \mathcal{S} \,|\,h_l^{(s)}(u) = h^{(s)}(u)\} = \emptyset$, i.e. when \emph{component} $h^{(s)}_l$ is nowhere \emph{active} on $\mathcal{S}$ regarding the term $h^{(s)}$, one should set $M^{(s)}_{l_+,l} = -\infty$. 
\end{definition}

\begin{assumption} 
We dispose of $\mathcal{S}$-bounds for problem \eqref{eq:min_problem} with $\mathcal{S} \supseteq \mathcal{X}$.
\label{A2}
\end{assumption}

\remark Definition \ref{def:s_bounds} fixes which properties our own so-called \emph{big-M} constants must satisfy. As we will see in Section \ref{sec:glob_micp_model}, akin to \cite{bigM15}, their role is to model a selection between multiple incentives. Here, if $l$ is not the index of the selected \emph{component} for the $s$-th term ($s \in [N]$), then it must be some $l' \in [n_s] \backslash \{l\}$ and the inequality $h_{l}^{(s)}(x) \leq \nu^{(s)} + M_{l,l'}^{(s)}$ should hold at any $x \in \mathcal{X}$ (Assumption \ref{A2}) for $\nu^{(s)}=h^{(s)}_{l'}(x)$.

\paragraph*{Computability of $\mathcal{S}$-bounds}
As such, even if we relax the nonconvex constraints on the right-hand side of \eqref{eq:bound_problem_M}, the problem defining our \emph{big-M} bounds remains of (DC) nature. If the dimension $d$ is (very) small and the topology of $\mathcal{X}$ polyhedral, one might resort to global optimization techniques \cite{PiraniUlus23} to get a tight value of the corresponding \emph{big-M} constant. Otherwise, the general problem is not tractable. We identify however generic cases for which one can compute upper-bounds with $\mathcal{S} = \mathcal{X}$.

\begin{enumerate}
\item $h_l^{(s)}(\cdot)$ has $L$-Lipschitz continuous gradient with respect to a norm $||\cdot||$ and the feasible domain is bounded with diameter $D = \sup_{(x,y) \in \mathcal{X}^2}\,||x-y|| < \infty$.
\begin{equation} M^{(s)}_{l_+,l} \gets \max_{u \,\in \,\mathcal{X}}\,h_{l_+}^{(s)}(\bar{x}) + \langle \nabla h_{l_+}^{(s)}(\bar{x}), u-\bar{x}\rangle + \frac{L}{2}D^2 -h_l^{(s)}(u) 
\label{eq:bound_M_smooth}
\end{equation}
for any reference $\bar{x} \in \mathcal{X}$ chosen. We emphasize that problem \eqref{eq:bound_M_smooth} is convex. \\ This latter might eventually be solved for more than one reference $\bar{x}$, the \emph{big-M} constant would then be chosen as the smallest encountered. \\
\item $h_l^{(s)}(\cdot) = \max_{e\,\in\,[B]}\,\gamma^{(s)}_{l,e}+\langle \beta^{(s)}_{l,e}, \cdot \rangle$ and $\mathcal{X}$ is bounded.
\begin{equation} M^{(s)}_{l_+,l} \gets \max_{e\,\in\,[B]}\,\bigg\{\max_{u\, \in\, \mathcal{X}}\, \gamma^{(s)}_{l,e}+\langle \beta^{(s)}_{l,e}, u \rangle -h_l^{(s)}(u)\bigg\}.
\label{eq:bound_M_max}
\end{equation}
Thus, one gets $M^{(s)}_{l_+,l}$ by solving $B$ convex programs, e.g.\@ LPs if $\mathcal{X}$ is polyhedral.\\
\item Any Trust Region Subproblem (TRS) like setup $(h_{l_+}^{(s)}(\cdot),h_l^{(s)}(\cdot),\mathcal{X})$ matching requirements of \cite{Nguyen17}. For example, let us describe the classical (TRS) case. \\Let $R>0$ and $\mathcal{X} = \mathbb{B}_2(\mathbf{0}_d\,;\,R)$. If for every $x \in \mathcal{X}$, $$ h_{l_+}^{(s)}(x)-h_l^{(s)}(x) = \frac{1}{2} \langle \mathbf{V}_{l_+,l}^{(s)} x, x \rangle + \langle \bar{v}_{l_+,l}^{(s)}, x \rangle + \gamma_{l_+,l}^{(s)}$$ for a symmetric matrix $\mathbf{V}_{l_+,l}^{(s)} \in \mathbb{R}^{d\times d}$, vector $\bar{v}_{l_+,l}^{(s)} \in \mathbb{R}^d$ and constant $\gamma_{l_+,l}^{(s)}$ then 
\begin{equation}
-M_{l_+,l}^{(s)} \gets \min_{u\,\in\,\mathcal{X}}\, -\bigg(\frac{1}{2} \langle \mathbf{V}_{l_+,l}^{(s)} u, u \rangle + \langle \bar{v}_{l_+,l}^{(s)}, u \rangle + \gamma_{l_+,l}^{(s)}\bigg) + \min\bigg\{\frac{\lambda_{l_+,l}^{(s)}}{2},0\bigg\}\,(R^2-||u||_2^2).
\label{eq:bound_M_TRS}
\end{equation}
where $\lambda_{l_+,l}^{(s)} = \lambda_{\text{min}}(-\mathbf{V}_{l_+,l}^{(s)})$. Again, the resulting problem \eqref{eq:bound_M_TRS} is convex \cite{Song22}.\\ 
\item $h_l^{(s)}(\cdot)$ is bounded from below by $\check{H}_{l}^{(s)} >- \infty$ and $h_{l_+}^{(s)}(\cdot)$ is bounded from above by $\hat{H}_{l_+}^{(s)} < \infty $. Then, it trivially comes
\begin{equation}
M_{l_+,l}^{(s)} \gets \hat{H}_{l_+}^{(s)}-\check{H}_{l}^{(s)}.
\label{eq:bound_M_crude}
\end{equation}
\end{enumerate}
\begin{remark} \label{rem:big_M_number}
If $n_s=\bar{n}$ for every $s\in [N]$, as it is the case for e.g. \eqref{eq:SMC_data}, there are $\mathcal{O}(N\bar{n}^2)$ \emph{big-M} parameters to compute to obtain the whole collection of $\mathcal{S}$-bounds. 
\end{remark}

\subsection{Global (MICP) Model} \label{sec:glob_micp_model} Now that the adequate tools have been introduced, we recast \eqref{eq:min_problem} as a (MICP). 
\paragraph*{\emph{big-M} formulation}
For the purpose of this formulation, we apply the standard \emph{epigraph} technique on every term $h^{(s)}$ of \eqref{eq:min_problem}'s sum. Specifically, for every $s \in [N]$, we introduce a continuous variable $\eta^{(s)}$ and we need to rewrite
\begin{equation} h^{(s)}(x) \leq \eta^{(s)} \label{eq:min_convex_ineq}
\end{equation}
as an equivalent set of constraints, eventually including binary variables, whose continuous relaxations are convex. Then, instead of the initial objective, we can minimize
\begin{equation} \bar{h}(x) + \frac{1}{N}\sum_{s=1}^{N}\,\eta^{(s)} \label{eq:new_min_problem_obj}
\end{equation}
on the new feasible set, i.e.\@ constituted of the aforementioned constraints as well as the initial feasible set $\mathcal{X}$. 

\noindent To encode \eqref{eq:min_convex_ineq}, we resort to $n_s$ new binary variables $t^{(s)} = (t^{(s)}_1,\dots,t^{(s)}_{n_s}) \in \{0,1\}^{n_s}$ and impose that exactly one of them equals $1$, i.e.
\begin{equation}
\sum_{l=1}^{n_s}\,t^{(s)}_l = 1.
\label{eq:single_t}
\end{equation}
Let $t_\sigma^{(s)}=1$ for some $\sigma \in [n_s]$. We recall that, by definition, if $h_l^{(s)}(x)\leq \eta^{(s)}$ for any $l \in [n_s]$ then $h^{(s)}(x) \leq \eta^{(s)}$. We consider the following inequalities for every $l_+ \in [n_s]$,
\begin{equation}
h_{l_+}^{(s)}(x) \leq \eta^{(s)} + \sum_{l=1}^{n_s}\,t_l^{(s)}\,M_{l_+,l}^{(s)}.\label{eq:almost_there}
\end{equation}
Among the inequalities \eqref{eq:almost_there} is hidden the required implication from \eqref{eq:min_convex_ineq} since 
\begin{equation} h_{\sigma}^{(s)}(x) \leq \eta^{(s)}+ M_{\sigma,\sigma}^{(s)} = \eta^{(s)} \Rightarrow h^{(s)}(x) \leq \eta^{(s)} \label{eq:first_strong_implication} \end{equation}
where we used \eqref{eq:single_t} and the fact that, by definition, $M_{l,l}^{(s)} = 0$ for every $l \in [n_s]$.\\
For every $l_+ \in [n_s] \backslash \{\sigma\}$, we verify now that inequalities \eqref{eq:almost_there} are not restrictive.\\ We start by noticing that for any $l \in \mathcal{A}^{(s)}(x)$, it \textbf{always} holds 
\begin{align}h_{l_+}^{(s)}(x)-h^{(s)}_{\sigma}(x) &\leq h_{l_+}^{(s)}(x)-h^{(s)}_{l}(x)\nonumber \\ &\leq \max_{u\,\in\,\mathcal{X}}\, h_{l_+}^{(s)}(u)-h^{(s)}_{l}(u)\quad \text{s.t.}\quad h_l^{(s)}(u) = h^{(s)}(u)\nonumber \\ &\leq M_{l_+,l}^{(s)}.\label{eq:non_restrictive}\end{align}
Then, it readily follows from \eqref{eq:first_strong_implication} and \eqref{eq:non_restrictive} that
\begin{equation} h_{l_+}^{(s)}(x) \leq  h_\sigma^{(s)}(x) + M^{(s)}_{l_+,l} \leq \eta^{(s)} + M_{l_+,l}^{(s)}. \label{eq:clearer_version_ineq} \end{equation} In the mean time, at any minimizer $x$ of \eqref{eq:new_min_problem_obj}, it is \textbf{impossible} that $t^{(s)}_\sigma=1$ and $\sigma \not \in \mathcal{A}^{(s)}(x)$ (which would mean $h^{(s)}_l(x)<h^{(s)}_{\sigma}(x)$) because choosing $t^{(s)}_{l}=1$ instead would allow $\eta^{(s)}$ to decrease, achieving thereby a strictly better objective value. This alternative choice is possible since $x$ variable would remain feasible. Indeed, for every $l_+ \in [n_s]\backslash \{l\}$, inequality \eqref{eq:almost_there} boils down to \eqref{eq:clearer_version_ineq} that is \textbf{always} satisfied under these conditions (i.e. $t^{(s)}_=1$ with $l \in \mathcal{A}^{(s)}(x)$) and \eqref{eq:first_strong_implication} simply turns into $h^{(s)}_{l}(x) \leq \eta^{(s)}$.
Thus, we can conclude that $t^{(s)}_\sigma=1$ enforces $\sigma \in \mathcal{A}^{(s)}(x)$ at the optimum and, thereby, we can safely lay down the equivalence
\begin{equation}
h^{(s)}(x)\leq \eta^{(s)} \Leftrightarrow \sum_{l=1}^{n_s}\,t_l^{(s)} = 1,\hspace{3pt} h^{(s)}_{l_+}(x) \leq \eta^{(s)} + \sum_{l=1}^{n_s}\,t_l^{(s)}\,M_{l_+,l}^{(s)}\quad \forall l_+ \,\in\,[n_s].
\label{eq:epigraph_fund}
\end{equation}
We have now a formulation that completely avoids the nonconvex term $h^{(s)}$ in \eqref{eq:min_convex_ineq}. 
\newpage 
Based on \eqref{eq:epigraph_fund}, we can formulate the following parametric mixed-integer convex program. For every $x \in \mathcal{X}$ and $C \in [0,1]$, $F^*(C) = \min_{x \in \mathcal{X}}\,V_C(x)$ where
\vspace{-10pt}
\begin{align}
V_C(x) := \min_{(\eta^{(1)},\dots,\eta^{(N)})\, \in\, \mathbb{R}^N} \quad & \bar{h}(x) + \frac{1}{N}\sum_{s=1}^{N}\,\eta^{(s)}& \label{eq:new_min_problem}
 \\
\textrm{s.t.} \quad & t^{(s)} \in \{0,1\}^{n_s},\hspace{15pt}\sum_{l=1}^{n_s}\,t_l^{(s)} = 1 & \forall s \in [N] \nonumber \\
 &  h^{(s)}_{l_+}(x) \leq \eta^{(s)} + C \cdot \sum_{l=1}^{n_s}\,t_l^{(s)}\,M_{l_+,l}^{(s)}
& \forall s \in [N],\quad \forall l_+ \in [n_s]\nonumber
\end{align}
One obtains the (MICP) formulation of \eqref{eq:min_problem} by setting $C=1$, i.e. $F^*(1)=F^*$. 

\vspace{-10pt}
\paragraph*{Continuum from \texttt{max} to \texttt{min}} One can interpret the goal of minimizing the sum of maximums of convex functions (i.e. \textbf{max} hereafter) as $F^*(0)$ regarding \eqref{eq:new_min_problem} above.
\vspace{-12pt}
$$ \textbf{max} \equiv \min_{x \in \mathcal{X}}\,\bar{h}(x) + \frac{1}{N}\sum_{s=1}^{N}\,\max_{l\,\in\,[n_s]}\,h_l^{(s)}(x) \quad \text{(convex)}$$
\vspace{-20pt}
$$\textbf{mean} \equiv \min_{x \in \mathcal{X}}\,\bar{h}(x) + \frac{1}{N}\sum_{s=1}^{N}\,\sum_{l=1}^{n_s}\,\frac{1}{n_s}\,h_l^{(s)}(x) \quad \text{(convex)}$$
\vspace{-20pt}
$$ \textbf{min} \equiv \min_{x \in \mathcal{X}}\,\bar{h}(x) + \frac{1}{N}\sum_{s=1}^{N}\,\min_{l\,\in\,[n_s]}\,h_l^{(s)}(x) \quad \text{(nonconvex)}$$
\begin{figure}
\centering
\includegraphics[width=0.8\textwidth]{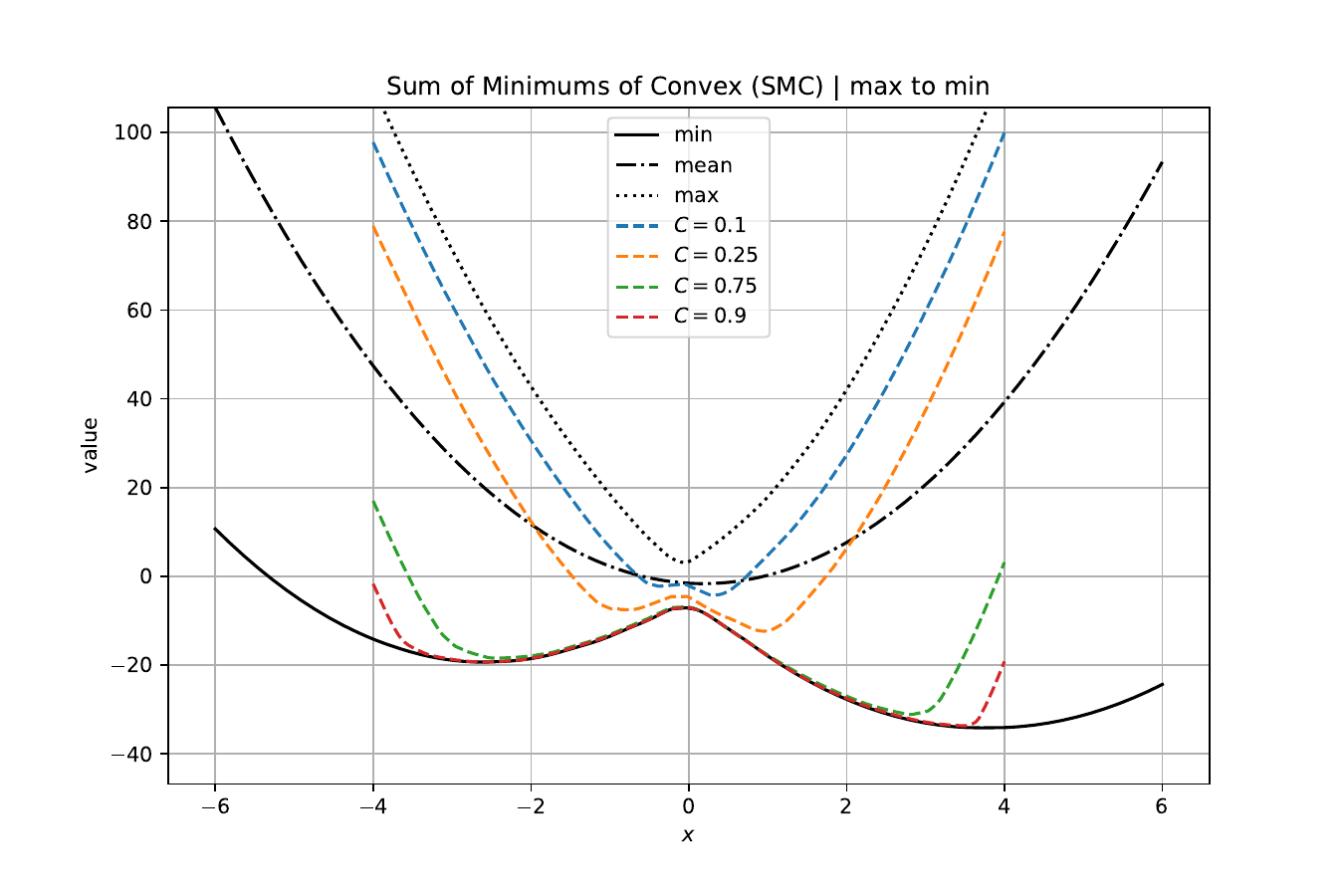}
\vspace{-10pt}
\caption{\emph{value functions} |  $C=0 \rightarrow $ sum of maximums, $C=1 \rightarrow $ sum of minimums}
\label{fig:SMC_interpolation}
\end{figure}
\vspace{-15pt}
\remark On Figure \ref{fig:SMC_interpolation}, all the \emph{component functions} were chosen to be 1D convex quadratics. In our computations of the \emph{value function} $V_C$, we used the non-uniform \emph{big-M} constants as devised in the previous section at bullet point (3). Here, $\mathcal{X} = [-4,4]$.\\In multiple colors, we have displayed the interpolating (from \textbf{max} to \textbf{min}) $V_C$ functions for $C \in \{10^{-1}, 2.5\cdot 10^{-1}, 7.5 \cdot 10^{-1}, 9 \cdot 10^{-1}\}$, all nonconvex in this case.\newpage
\paragraph*{\emph{Perspective} formulation} For the sake of completeness, as a minor contribution, we extend the \emph{perspective} eformulation \cite{Baratt20}, summarized in \eqref{eq:clipped_micp}, when $n_s\in \mathbb{N}$ for every $s\in [N]$. Again, we resort to $n_s$-dimensional vectors of binary variables, i.e. $$q^{(s)} = (q_1^{(s)},\dots,q_{n_s}^{(s)}) \in \{0,1\}^{n_s} \quad \forall s \in[N].$$ We also need $n_s$ copies $\{z^{(s)}_l\}_{l=1}^{n_s}$ of the main variable $x$. According to \eqref{eq:perspective_redef}, imposing 
\vspace{-5pt}$$\sum_{l=1}^{n_s}\,q^{(s)}_{l} = 1 \quad \forall s \in [N]$$
will force, for every $(s,\sigma_s) \in [N] \times [n_s]$ such that $q^{(s)}_{\sigma_s}=1$, the equality $x = z_{\sigma_s}^{(s)}$. On the other hand, $q_{l}^{(s)}=0$ implies that $z^{(s)}_{l} = \mathbf{0}_d$.
Therefore, we write in a compact form 
\begin{align} F^* = \min_{x\,\in\,\mathcal{X}} \quad &\, \frac{1}{N}\,  \sum_{s=1}^N\, \sum_{l=1}^{n_s}\,(\bar{h})^\dagger(x-z_l^{(s)},1-q_l^{(s)})+(\bar{h})^\dagger(z_l^{(s)},q_l^{(s)}) \label{eq:gen_perspective} \\ & +\frac{1}{N}\,  \sum_{s=1}^N\,\sum_{l=1}^{n_s}\, (h_l^{(s)})^\dagger(x-z_l^{(s)},1-q_l^{(s)}) +(h_l^{(s)})^\dagger(z_l^{(s)},q_l^{(s)})\nonumber\\
\textrm{s.t.} \quad & q^{(s)} \in \{0,1\}^{n_s},\hspace{15pt}\sum_{l=1}^{n_s}\,q_l^{(s)} = 1,\quad z^{(s)} \in \mathcal{X}^{n_s} \quad \forall s \in [N]. \nonumber 
\end{align}
\subsection{Local (MICP) Model} \label{subsec:local_micp} In this subsection, we outline how \emph{simpler} the description of the objective $F$ becomes locally. Using continuity arguments, we will show that only terms $h^{(s)}$ for which $s$ belongs to the \emph{degeneracy set} (see Definition \ref{def:degeneracy_set}) need to be expressed using binary variables as in \eqref{eq:epigraph_fund}. This consideration in mind, checking the \emph{local optimality} of a candidate $\hat{x} \in \mathcal{X}$ comes at a (hopefully) reduced computational expense.
\begin{definition}[degeneracy factor \& degeneracy set]\label{def:degeneracy_set} Let $\hat{x} \in \mathcal{X}$. We call the \emph{degeneracy factor} at $\hat{x}$, the size of the set $\mathcal{A}^{(1)}(\hat{x})\times \dots \times \mathcal{A}^{(N)}(\hat{x})$, i.e. $\Pi_{s=1}^{N} \big|\mathcal{A}^{(s)}(\hat{x})\big|$.
In addition, we call \emph{degeneracy set} $I(\hat{x})$ at $\hat{x}$ the set of indices $s\in[N]$ such that $|\mathcal{A}^{(s)}(\hat{x})|\geq 2$, i.e.\@ more than one \emph{component function} is active at $\hat{x}$ with respect to \eqref{eq:active}. 
\end{definition}
\begin{remark} \label{rem:easy_expr}According to Proposition \ref{local_non_expansiveness}, there exists $\alpha>0$ such that for every $s \in [N]$, $\mathcal{A}^{(s)}(x)\subseteq \mathcal{A}^{(s)}(\hat{x})$ for all $x$ belonging to the neighbourhood $\mathcal{X}\,\cap \,\mathbb{B}_2(\hat{x}\,;\,\alpha)$. Then, instead of involving $N$ terms each of which being the pointwise minimum of $n_s$ \emph{components}, the analytical expression of objective $F$ requires only $|I(\hat{x})|$ of them to be pointwise minima of $|\mathcal{A}^{(s)}(\hat{x})| $ \emph{components}. For the other indices, i.e. $s \not \in I(\hat{x})$, $\mathcal{A}^{(s)}(\hat{x})=\{\sigma_s\}$ for some $\sigma_s \in [n_s]$. Indeed, in light of these considerations, we develop
\begin{align}
F(x) &= \bar{h}(x) + \frac{1}{N}\, \sum_{s\, \in\, [N]}\, \min_{l\, \in \,\mathcal{A}^{(s)}(x)}\, h^{(s)}_l(x)\nonumber \\ &=\bar{h}(x)+\frac{1}{N}\,\Bigg( \sum_{s\in [N]\backslash I(\hat{x})}\, h^{(s)}_{\sigma_s}(x)\Bigg)+\frac{1}{N}\,\Bigg( \sum_{s \in I(\hat{x})}\, \min_{l\, \in \,\mathcal{A}^{(s)}(\hat{x})}\, h^{(s)}_l(x)\Bigg). \label{eq:objective_reduced_0}
\end{align}
\end{remark}

\noindent Building upon our \emph{big-M} (MICP) reformulation \eqref{eq:new_min_problem} and \eqref{eq:objective_reduced_0}, we prove a first lemma.
\begin{lemma}[\emph{local optimality}: equivalent conditions] \label{lemma_lo_smc}Let $\hat{\mathcal{X}}(\hat{\alpha}) := \mathcal{X} \cap \mathbb{B}_2(\hat{x}\, ;\, \hat{\alpha})$ for any $\hat{\alpha}>0$ and $\hat{x} \in \mathcal{X}$.  The point $\hat{x}$ is a local minimizer of \eqref{eq:min_problem} if and only if there exists $\hat{\alpha}>0$ such that $F(\hat{x})=\min_{x\,\in\,\hat{\mathcal{X}}(\hat{\alpha})}\,\hat{F}(x|\hat{x})$ where
\begin{align}
\hat{F}(x|\hat{x}):= \min_{\eta\, \in\, \mathbb{R}^N} \quad & \bar{h}(x) + \frac{1}{N}\sum_{s \in [N]\backslash I(\hat{x})}\,h^{(s)}_{\sigma_s}(x) + \frac{1}{N}\sum_{s\,\in\,I(\hat{x})}\,\eta^{(s)} & \label{eq:new_min_problem_local}
 \\
\hspace{-40pt} \forall s \in I(\hat{x}),\quad & t^{(s)} \in \{0,1\}^{n_s},\hspace{15pt}\sum_{l=1}^{n_s}\,t_l^{(s)} = 1&\nonumber  \\
 &  h^{(s)}_{l_+}(x) \leq \eta^{(s)} + \sum_{l=1}^{n_s}\,t_l^{(s)}\,M_{l_+,l}^{(s)} & \forall l_+ \in \mathcal{A}^{(s)}(\hat{x}) \nonumber \\
 & t^{(s)}_{l_+} = 0 & \forall l_+ \not \in \mathcal{A}^{(s)}(\hat{x})\nonumber
\end{align}
\label{lemma_equivalence_local_optimality}
where $\mathcal{A}^{(s)}(\hat{x}) = \{\sigma_s\}$ for every $s \in [N]\backslash I(\hat{x})$.
\end{lemma}
\begin{proof} Remark \ref{rem:easy_expr} states that there exists $\alpha>0$ so that $F$ can be expressed as in \eqref{eq:objective_reduced_0} for any $0<\hat{\alpha}\leq\alpha$, locally on $\hat{\mathcal{X}}(\hat{\alpha})$. To obtain the formulation at the right hand side of \eqref{eq:new_min_problem_local}, we apply our \emph{epigraph} reformulation trick \eqref{eq:epigraph_fund} to every term $s\in I(\hat{x})$ in order to rewrite the constraints\vspace{-15pt}
$$\vspace{-10pt}\min_{l\, \in\, \mathcal{A}^{(s)}(\hat{x})}\,h^{(s)}_{l}(x) \leq \eta^{(s)}$$
while when $s\not \in I(\hat{x})$, one just needed to set $h^{(s)}(x) = h_{\sigma_s}^{(s)}(x)$ since $\mathcal{A}^{(s)}(\hat{x})=\{\sigma_s\}$.
\\\vspace{-5pt} \\ $(\Rightarrow)$ Let $\hat{x}$ be a \emph{local minimizer} of $F$ then there exists $\bar{\alpha}>0$ such that $F(\hat{x}) = \min_{x \in \hat{\mathcal{X}}(\bar{\alpha})}\,F(x)$. Let $\hat{\alpha} = \min\{\bar{\alpha},\alpha\}$. Thereby, $\hat{x}$ is also the minimum of $F$ over $\hat{\mathcal{X}}(\hat{\alpha})$. On $\hat{\mathcal{X}}(\hat{\alpha})$, $\hat{F}(\cdot|\hat{x})$ and $F$ coincide so that equality \eqref{eq:new_min_problem_local} holds. \\ \vspace{-5pt} \\ $(\Leftarrow)$ If $\hat{\alpha}>0$ exists so that $F(\hat{x})=\min_{x\,\in\,\hat{\mathcal{X}}(\hat{\alpha})}\,\hat{F}(x|\hat{x})$ then for any $\bar{\alpha} \in (0,\min\{\alpha,\hat{\alpha}\}]$, $\hat{x}$ will remain a \emph{global minimum} of $\hat{F}(\cdot|\hat{x}) = F$ over $\hat{\mathcal{X}}(\hat{\alpha})$. 
\end{proof}
\noindent Although interesting theoretically speaking, one can not try every $\hat{\alpha}>0$ in Lemma \ref{lemma_equivalence_local_optimality} to prove or disprove the \emph{local optimality} of a candidate $\hat{x}\in \mathcal{X}$. Nevertheless, fixing a suitable neighbourhood $\mathcal{S}$ around $\hat{x}$, depending on the geometry of the problem, and solving the problem at the right hand side of \eqref{eq:new_min_problem_local} will span two different outcomes both from which one can learn something. We detail this formally in the next corollary. 
\begin{corollary} \label{coro:restart_check} Let $\hat{\alpha}>0$ be fixed in Lemma \ref{lemma_lo_smc} and $\mathcal{S}\supseteq \mathcal{X} \,\cap\,\mathbb{B}_2(\hat{x}\, ;\, \hat{\alpha})$ and   
\begin{equation}
\hat{F}^*_{\hat{x},\mathcal{S}} = \min_{x\,\in\,\mathcal{S}}\,\hat{F}(x|\hat{x}). \label{eq:new_min_problem_local_2}
\end{equation}
Let $x^*(\hat{x},\mathcal{S})$ denote a minimizer of \eqref{eq:new_min_problem_local_2}. Either $F(\hat{x}) = \hat{F}_{\hat{x},\mathcal{S}}^*$ and $\hat{x}$ is a local minimizer of \eqref{eq:min_problem} or $F(x^*(\hat{x},\mathcal{S}))\leq \hat{F}^*_{\hat{x},\mathcal{S}}<F(\hat{x})$. 
\end{corollary}
\begin{proof}
In any case, $\hat{F}(\hat{x}|\hat{x})=F(\hat{x})$ and for every $x \in \mathcal{S}$, \begin{equation} \hat{F}(x|\hat{x}) = \min_{\sigma \in\,\bigtimes_{s=1}^{N}\,\mathcal{A}^{(s)}(\hat{x})}\,F_{\sigma}(x) \geq F(x) =  \min_{\sigma \in\,\bigtimes_{s=1}^{N}\,[n_s]}\,F_{\sigma}(x). \label{eq:over_approx} \end{equation} If $\hat{F}^*_{\hat{x},\mathcal{S}} = F(\hat{x})$ then it is clear that $F(\hat{x}) = \min_{x\,\in\,\mathcal{X}\,\cap\,\mathbb{B}_2(\hat{x}\,;\,\hat{\alpha})}\,\hat{F}(x|\hat{x})$ and Lemma \ref{lemma_lo_smc} applies. Otherwise, $\hat{F}(\hat{x}|\hat{x})<F(\hat{x})$ and $F(x^*(\hat{x},\mathcal{S}))\leq \hat{F}^*_{\hat{x},\mathcal{S}}$ according to \eqref{eq:over_approx}. 
\end{proof}
 \remark Optimization problem in \eqref{eq:new_min_problem_local_2} involves $Z \in \mathbb{N}$ binary variables\footnote{Dummy binary variables set to $0$, i.e. $t^{(s)}_{l_+}$ for every $l_+ \not \in \mathcal{A}^{(s)}(\hat{x})$ with $s \in I(\hat{x})$, are not counted.}, \begin{equation} Z=\sum_{s\,\in\,I(\hat{x})}\,\big|\mathcal{A}^{(s)}(\hat{x})\big|.\label{eq:reduced_bin} \end{equation} Yet, among the \emph{possible} assignments for these variables only $\Pi_{s=1}^N\, \big|\mathcal{A}^{(s)}(\hat{x})\big|<2^Z$ are \emph{feasible} and worth consideration. $\Pi_{s=1}^N\, \big|\mathcal{A}^{(s)}(\hat{x})\big|$ is nothing else than the \emph{degeneracy factor} at $\hat{x}$, showing the intrinsic link between the hardness of certifying \emph{local optimality} and this specific number. In the worst-case, this latter could be as large as $\Pi_{s=1}^N\, n_s$. As we observed in our experiments, this is far from being true in practice. 
\\
  
 \noindent Our true motivation to express $\hat{F}(\cdot|\hat{x})$ as above \eqref{eq:new_min_problem_local} is to take advantage of state-of-the-art (MICP) solvers, e.g. \texttt{MOSEK} \cite{Mosek24}, \texttt{BARON} \cite{Baron24}, that might compute a minimizer of \eqref{eq:new_min_problem_local_2} faster than a full combinatorial enumeration that we recall now
\begin{equation}
\hat{F}_{\hat{x},\mathcal{S}}^* = \min_{\sigma \,\in\,\bigtimes_{s=1}^{N}\,\mathcal{A}^{(s)}(\hat{x})}\,\Bigg\{\min_{x\,\in\,\mathcal{S}}\,\bar{h}(x) + \frac{1}{N}\,\sum_{s=1}^{N}\,h^{(s)}_{\sigma_s}(x)\Bigg\}
\label{eq:enumeration_local}
\vspace{-2pt}
\end{equation}
 Yet, if the \emph{degeneracy factor} at $\hat{x}$ is small enough, it might be more convenient to get $\hat{F}_{\hat{x},\mathcal{S}}^*$ in a straightforward enumerative fashion since (MICP) formulations require, under the hood, additional computational work (e.g. \emph{canonicalization}, \emph{factorizations}, etc.). This is especially true when every $\sigma$-selection problem defined by
 $$ \vspace{-5pt}\nu_{|\mathcal{S}}(\sigma) := \min_{x\,\in\,\mathcal{S}}\,\bar{h}(x) + \frac{1}{N}\,\sum_{s=1}^{N}\,h^{(s)}_{\sigma_s}(x)$$
 admits a closed-form expression. Therefore, when resorting to Corollary \ref{coro:restart_check} in practice, one should first question the time it would take to adopt the approach of \eqref{eq:enumeration_local}. If it is presumably too long then one should launch a global solver to compute $\hat{F}^*_{\hat{x},\mathcal{S}}$.
 \begin{remark} \label{rem:choose_s_bounds}
Let $\hat{\mathcal{S}} \supseteq \mathbb{B}_2(\hat{x}\,;\,\hat{\alpha})$ so that $\mathcal{S}= \mathcal{X} \,\cap\,\hat{\mathcal{S}} \supseteq \mathcal{X}\, \cap \,\mathbb{B}_2(\hat{x}\, ;\, \hat{\alpha})$ in Corollary \ref{coro:restart_check}. One should, if possible, use refined $\hat{\mathcal{S}}$-bounds (Definition \ref{def:s_bounds}) instead of the assumed $\mathcal{X}$-bounds (Assumption \ref{A2}), when building the model of $\hat{F}(\cdot|\hat{x})$, i.e. \eqref{eq:new_min_problem_local}, involved in the problem \eqref{eq:new_min_problem_local_2}. The choice of $\hat{\mathcal{S}}$ significantly affects the practical solvability of $F^*_{\hat{x}, \mathcal{S}}$. We illustrate this procedure in Section \ref{subsec:lo_certif} with two numerical examples for which, fortunately, $\hat{\mathcal{S}}$-bounds are available in closed-form. 
\end{remark}
\vspace{-15pt}
\subsection{Global (BIC) Model} \label{subsec:global_bic}
We proceed to the second global reformulation of \eqref{eq:min_problem}, this latter being of utmost practical interest.
Every discrete minimum among $\bar{n} \in \mathbb{N}$ elements from a vector $\textbf{h} = (h_1, \dots, h_{\bar{n}})$ satisfies the relationship 
\begin{equation}
\min_{l\,\in\,[\bar{n}]}\,h_l = \min_{q \,\in \,\Delta^{\bar{n}}}\, q^T \textbf{h}. 
\label{eq:reformulation_min_to_BC}
\end{equation}

\noindent Based on \eqref{eq:reformulation_min_to_BC}, looking through another lens, \eqref{eq:min_problem} is equivalent to
\begin{equation}
    F^* = \min_{x \,\in\, \mathcal{X},\,Q\,\in\, \mathcal{Q}}\hspace{3pt}\Bigg\{\bar{F}(x,Q):= \bar{h}(x) + \frac{1}{N}\,\sum_{s=1}^{N}\,\sum_{l=1}^{n_s}\,q_l^{(s)}h_{l}^{(s)}(x)\Bigg\} 
 \label{eq:big_smc}
\tag{BIC-SMC}
\end{equation}
with $\mathcal{Q} = \Delta^{n_1}\times\dots\times \Delta^{n_N}$ and $Q= (q^{(1)},\dots,q^{(N)})$.
Fixing the values of \emph{component weights} $\{q^{(s)}\}_{s=1}^N$ in their respective simplices, the problem \eqref{eq:upper_rel} becomes convex
\begin{equation}
\bar{F}^*_{|Q} = \min_{\tilde{x} \in \mathcal{X}} \bar{F}(\tilde{x},Q).
\label{eq:upper_rel}
\end{equation} 
Conversely, fixing $x \in \mathcal{X}$ and minimizing with respect to $Q$ represents an (LP)
\begin{equation}
\bar{F}^*_{|x} = \min_{\tilde{Q}\, \in\, \mathcal{Q}} \bar{F}(x,\tilde{Q}).
\label{eq:upper_rel_2}
\end{equation} 
Considering the notations above, we can finally write 
$$\min_{x\, \in\, \mathcal{X}}\,\bar{F}^*_{|x} = F^* = \min_{Q \,\in\, \mathcal{Q}}\,\bar{F}^*_{|Q}.$$
Together, equations \eqref{eq:upper_rel} and \eqref{eq:upper_rel_2} undoubtedly suggest a two-step alternating minimization (\texttt{AM}) \cite{Grippo00} map $\mathcal{M} : \mathcal{X} \times \mathcal{Q} \to \mathcal{X} \times \mathcal{Q},$ $(x,Q)\to \mathcal{M}(x,Q) = (x_+,Q_+)$  
\begin{equation}
\begin{cases}
x_+& \gets \arg \min_{\tilde{x} \in\mathcal{X}} \bar{F}(\tilde{x},Q) \\
Q_+& \gets \arg \min_{\tilde{Q} \in \mathcal{Q}} \bar{F}(x_+,\tilde{Q}) 
\end{cases}
\label{eq:am_spec}
\tag{\texttt{AM}}
\end{equation}
for which one looks after \emph{fixed-value pairs}, i.e. $x \in \mathcal{X}$ and $Q\in\mathcal{Q}$ such that $\bar{F}(x,Q) = \bar{F}(\mathcal{M}(x,Q))$. Let $x^*$ be a minimizer of \eqref{eq:min_problem} and let $Q^* \in \arg \min_{\tilde{Q}\,\in\,\mathcal{Q}}\,\bar{F}(x^*,\tilde{Q})$, then $(x^*,Q^*)$ must be a \emph{fixed-value pair} of $\mathcal{M}$, easily deduced from the fact that $\mathcal{M}$ decreases $\bar{F}$ monotonically. Actually, we will see in a while that being a \emph{fixed-value pair} $(x,Q)$ is sufficient to prove that $x$ is \emph{critical} for \eqref{eq:min_problem}. However, as we expose in Proposition \ref{lemma:criticality}, \emph{criticality} requires even less stringent conditions in general. To that end, we need to define our notion of \emph{gain} (when optimizing the \emph{weights}).
\begin{definition}[Gain] Let $Q \in \mathcal{Q}$ and $x_+ \in \arg \min_{\tilde{x} \in \mathcal{X}}\,\bar{F}(\tilde{x},Q)$. The \emph{gain} at $x_+$ is the positive quantity defined as \begin{equation} \mathcal{G}^*_{(x_+, Q)} = \bar{F}(x_+, Q) - \min_{\tilde{Q}\, \in\, \mathcal{Q}}\, \bar{F}(x_+, \tilde{Q})\geq 0.\label{def:gain_x} \end{equation}
\end{definition}
\begin{proposition} \label{lemma:criticality}
If $\mathcal{G}^*_{(x_+,Q)} = 0$ in \eqref{def:gain_x} then $x_+$ is a critical point for \eqref{eq:min_problem}.
\end{proposition}
\begin{proof}
This proof is deferred to Appendix \ref{crit_proof}.
\end{proof}
\remark \label{rem_false_converse} We crafted a simple toy example ($N=1$, $n_1=3$, $d=1$) showing the converse is not true. Consider the function $x \to F(x) = |x| + \min\{x-1/8,x^2,2x-1/16\}$, weights $Q= (0,1,0)$ (e.g. encoding \emph{active components} at $x=1/2$) and $\mathcal{X} = [-2,2]$. One easily obtains that $x_+ = 0 = \arg \min_{\tilde{x}\,\in\,[-2,2]}\,|\tilde{x}|+\tilde{x}^2$ is \emph{critical} in this context. According to the (DC) decomposition \eqref{eq:DC_decomposition}, $\partial f_1(x_+) = [2,4]$ and $\partial f_2(x_+) = \{2\} $.\\ Yet, the \emph{gain} is strictly positive with respect to $\bar{F}(x_+,Q)=0$ in the sense that $$F(x_+)= \min_{\tilde{Q} \in \Delta^3}\,\bar{F}(x_+,\tilde{Q})=-1/8 < 0\vspace{-10pt}$$ achieved for $\tilde{Q} = (1,0,0)$. Note that $x_+$ is not globally optimal since $F^* = -33/16$.  \\ \\
 \hspace{-17pt} Equation \eqref{eq:decrease_decomposition} tells explicitly that \emph{fixed-value pairs} of $\mathcal{M}$ produce \emph{critical points} $x_+$,
\begin{equation}\bar{F}(x,Q)-\bar{F}(\mathcal{M}(x,Q))  = \underbrace{(\bar{F}(x,Q)-\bar{F}(x_+,Q))}_{\geq 0} + \underbrace{(\bar{F}(x_+,Q)-\bar{F}(x_+,Q_+))}_{=\,\mathcal{G}^*_{(x_+,Q)}}
\label{eq:decrease_decomposition} \end{equation} due to the implication $\bar{F}(x,Q) - \bar{F}(\mathcal{M}(x,Q))=0\Rightarrow \mathcal{G}^*_{(x_+,Q)}=0$ and Proposition \ref{lemma:criticality}.\\ In Section \ref{sec:RAM} hereafter, we will ensure the sufficient decrease of our relaxed alternating minimization approaches (\texttt{r-AM}) with respect to the monitorable quantity $\mathcal{G}^*_{(x_+,Q)}$.

\subsection{Problem specific knowledge encoding} We end this section with a word on \emph{valid} and/or \emph{symmetry breaking} constraints, also playing a determinant role in (MICP)'s solvability \cite{Liberti12,Dias21}. In some situations, one is aware, \emph{a priori}, of some regularity conditions or specific structure that holds at any optimal solution of \eqref{eq:min_problem}. General \emph{valid} constraints implement cuts in the original search space whereas \emph{symmetry breaking} constraints essentially induce an ordering of the solutions and prevent the possibility of having multiple optimal solutions that are permutations of each other \cite{Sabo24}. We illustrate now both types of constraints hereafter. 
\example[$\ell_2$-\emph{clustering}] \label{l2_cluster} Let $\{\beta^{(s)}\}_{s \in [N]}$ be $N$ points in $\mathbb{R}^d$. A clustering task with $B \in \mathbb{N}$ centroids might ask to minimize the average distance to the nearest centroid among $\{x_l\}_{l \in[B]}$, \vspace{-15pt}$$x = (x_1, \dots, x_B) \to F(x) = \frac{1}{N}\sum_{s=1}^{N}\, \min_{l\,\in\,[B]}\,|x_l-\beta^{(s)}||_2.$$  
\begin{itemize} 
\item[(i)] For any $l \in [B]$, it can be shown that an optimally located centroid $x^*_l$ is contained in the convex hull of $\{\beta^{(s)}\}_{s=1}^{N}$. The following constraints are \emph{valid cuts}  
\begin{equation}
x_l^*\in \textbf{conv}\Big(\big\{\beta^{(1)},\dots,\beta^{(N)}\big\}\Big) \quad \forall l \in [B].
\label{eq:chull}
\end{equation}
Also, without loss of generality, an optimal solution can be assumed to satisfy the following conditions: 
\begin{equation}
x^*_{l_1} \not =x^{*}_{l_2}\quad \forall (l_1,l_2)\in [B]^2,\hspace{3pt} l_1\not=l_2.
\label{eq:differences}
\end{equation}
Unfortunately, this latter constraint is nonconvex. If such hypothesis applies, \\a more restrictive constraint that entails \eqref{eq:differences} would be, for a chosen $\updelta\geq0$,
\begin{equation}
\langle\mathbf{1}_d,x_{l_1} -x_{l_2}\rangle \geq \updelta\quad \forall (l_1,l_2)\in [B]^2,\hspace{3pt} l_1<l_2.
\label{eq:differences_stronger}
\end{equation}
These last \emph{symmetry breaking} constraints rank subvectors $\{x_l\}_{l\in[B]}$ according to their mean and impose a separation of $\updelta/d$ between two successive means.\\

\item[(ii)] For any $l \in[B]$, there should be at least one data point $\beta^{(s)}$ such that $x^*_l$ is the closest among $(x_1^*,\dots,x_B^*)$ from $\beta^{(s)}$. We write the \emph{valid} constraints \eqref{eq:selection_minimale} and \eqref{eq:selection_minimale_2}, respectively for our (MICP) and (BIC) reformulations of \eqref{eq:min_problem},
\begin{align}
\sum_{s\in [N]}\,t_l^{(s)} \geq 1 &\quad \forall l \in[B],  \label{eq:selection_minimale} \\
\vspace{-20pt}
\sum_{s\in [N]}\,q_l^{(s)} \geq 1& \quad \forall l \in [B]. \label{eq:selection_minimale_2}
\end{align}
\end{itemize}
\begin{remark} If \emph{a priori} information is encoded through constraints such as \eqref{eq:chull}, we implicitly assume that feasible sets $\mathcal{X}$, $\mathcal{T} = \bigtimes_{s=1}^{N}\,\{0,1\}^{n_s}$ and $\mathcal{Q} = \bigtimes_{s=1}^{N}\,\Delta^{n_s}$ for variable $x$, $T=(t^{(1)},\dots,t^{(N)})$ and $Q=(q^{(1)},\dots,q^{(N)})$ incorporate these constraints.\end{remark}
\vspace{-15pt}
\section{Relaxed Alternating Minimization (r-AM)} 
\label{sec:RAM}
Although alternating minimization applied on \eqref{eq:big_smc} is intuitive and works relatively well in practice (see e.g.\@ \cite{Cortes15,Chen22}), we emphasize a major drawback in Section \ref{subsec:rationale}. Fortunately, this drawback can be somehow mitigated and this is what we propose to explore in Section \ref{subsec:q_updates}. 
\vspace{-10pt}
\subsection{Rationale}
\label{subsec:rationale}
By default, the update of the weights in \eqref{eq:am_spec}, i.e. the minimization of $\bar{F}(x_+,\tilde{Q})$ with respect to $\tilde{Q}$ at fixed $x_+ \in \mathcal{X}$, assigns systematically for every $s\in[N]$ and $l\in[n_s]$,
\begin{equation} (q_+^{(s)})_l = \begin{cases} 1 & \text{when } $l$ \text{ is the smallest element of }\mathcal{A}^{(s)}(x_+) \\ 0 & \text{otherwise}.\end{cases}\label{eq:am_default_update} \end{equation}
In other words, the information about other \emph{component}'s values, i.e. $h^{(s)}_l(x_+)$ for every $s \in [N]$ and $l \not= \min \mathcal{A}^{(s)}(x_+)$, is totally disregarded. Lemma \ref{lemma_equivalence_local_optimality} underlines that directly next to $x_+$,  the representation of objective $F$ only relies on indices belonging to active sets $\mathcal{A}^{(s)}(x_+)$. However, \emph{components} close to active at $x_+$ play a key role when it comes to represent the landscape of each loss $h^{(s)}$ in a neighbouring valley, i.e.\@ regions $\mathcal{R}(\sigma)$ (recall Eq. \eqref{eq:packing}) for which there exists a $\tilde{s} \in [N]$ such that $\sigma_{\tilde{s}} \not \in \mathcal{A}^{(\tilde{s})}(x_+)$.\\
\vspace{-5pt}
\\
\noindent Explained differently, \texttt{AM} proceeds as follows. Given $x_+ \in \mathcal{X}$, it sets $\sigma=(\sigma_1,\dots,\sigma_N)$  where, for every $s\in[N]$, $\sigma_s = l$ if and only if $(q_+^{(s)})_l=1$ according to \eqref{eq:am_default_update}, achieving $$ F_{\sigma}(x_+)=\bar{h}(x_+)+\frac{1}{N}\sum_{s=1}^N\,h^{(s)}_{\sigma_s}(x_+) = F(x_+).$$ Then, $x_{++}$ is chosen as a global minimizer of $F_{\sigma}$. If it occurs that $F_{\sigma}(x_{++})=F(x_{++})$ then \texttt{AM} might already be stuck. Indeed, this means that $\sigma_s$ belongs to $\mathcal{A}^{(s)}(x_{++})$ for every $s\in[N]$. Furthermore if $\sigma_s$ is the smallest element\footnote{Note that it was the smallest element of $\mathcal{A}^{(s)}(x_+)$ according to \eqref{eq:am_default_update} and the fact $\sigma_s=l \Leftrightarrow (q_+^{(s)})_l=1$.} of $\mathcal{A}^{(s)}(x_{++})$ for every $s \in [N]$ then the procedure detailed above will loop forever.
 This suggests that \texttt{AM} is biased in the sense that $x$-updates only take into account the very local geometry of $F$ and, presumably,  tends to converge fast but towards \emph{critical points} not far away from the initial guess. Based on this observation, one may attempt to put strictly positive weights not only on indices from $\mathcal{A}^{(s)}(x_+)$ but also on other promising indices linked to close to active \emph{components}. As mentioned earlier, they better describe the surroundings of $x_+$ and, possibly, allow to escape from the current convex valley wherein \texttt{AM} will perhaps stay until convergence. We initiated this work with the belief that coupled, with a multi-start strategy, this approach would promote more exploration in the search space and lead to better final objective values. \\ \vspace{-8pt}\\
To end this subsection, we show a 1D toy example clearly displaying the valley escape phenomenon we strive for when relaxing plain alternating minimization. 
\begin{example}[Valley escape phenomenon] On Figure \ref{fig:SMC_valley_escape}, we depicted a \eqref{eq:min_problem} objective defined for every $x \in \mathcal{X} = [-5,5]$ by $F(x) = \frac{1}{2}\min\big\{x^2-3x-2,x^2+2,x^2+x-2,x^2+4x+2\big\} +\frac{1}{2}\min\big\{\frac{1}{2}x^2+2x-2,x^2+4x+2,x^2-1\big\}$. Its (unique) global minimizer is located at $x^* = -4$. We sampled $10^2$ equidistant starting points along $[-5,5]$ and applied plain alternating minimization \texttt{AM} subject to the election rule \eqref{eq:am_default_update}. We also tried an alternative version wherein for every other iteration, instead of \eqref{eq:am_default_update}, we set up the weights as 
\vspace{-6pt}
\begin{equation}
q^{(s)}_+ = \textbf{softmin}\Big(\kappa \cdot \big(h^{(s)}_1(x_+),\dots,h^{(s)}_{n_s}(x_+)\big)\Big)\quad \forall s \in [N].
\label{eq:prelude_softmin}
\end{equation}
where the \textbf{softmin}  transformation of $u = (u_1,\dots,u_{\bar{n}}) \in \mathbb{R}^{\bar{n}}$ reads
\begin{equation}
(v_1,\dots,v_{\bar{n}}) = \textbf{softmin}(u) \quad \text{with}\quad v_l = e^{-u_l}/\sum_{l'\,\in\,[\bar{n}]}\,e^{-u_{l'}}\quad \forall l \in [\bar{n}].
\label{def:softminmax}
\vspace{-8pt}
\end{equation}
Here, $N=2$ with $n_1=4$ and $n_2=3$. This alternative method with $\kappa= 1/4$ is dubbed as \texttt{ALTER} on Figure \ref{fig:SMC_valley_escape}. Note that $\kappa\to \infty$ in \eqref{eq:prelude_softmin} would have led to \texttt{AM} back again. Hence, using the \textbf{softmin} operator above relaxes exact optimization of the weights at fixed $x_+ \in \mathcal{X}$. One can observe that for every starting point, it allowed \texttt{ALTER} to converge towards $x^*$ although the relaxation mechanism \eqref{eq:prelude_softmin} was employed for only half of the iterations. It allowed jumping over a \emph{local minimum} near $x \simeq 3/2$ while \texttt{AM} inevitably got trapped therein.
\end{example}
\vspace{-20pt}
\subsection{$x$-updates} 
We emphasize that, compared to \texttt{AM}, the \texttt{r-AM} methods we will describe only differ in the way they handle $Q$-updates. Let $Q \in \mathcal{Q}$, the minimizer $x^*_{|Q}$ of $F^*_{|Q}$ (see \eqref{eq:upper_rel}) returned by our black-box (Assumption \ref{A1}) is chosen deterministically.
\begin{remark}When the mapping $Q \to x^*_{|\mathcal{Q}}$ is Lipschitz continuous, any algorithm (e.g. \texttt{AM}) alternating between $x$-updates and $Q$-updates enjoys the desirable feature that converging $Q$ weights also imply converging $x$ iterates. Nevertheless, one cannot usually assess the Lipschitz continuity of such a mapping beforehand. However, it might happen in practice that two successive iterates $x$ and $x_+$ define the exact same \emph{active sets}, i.e. $\mathcal{A}^{(s)}(x)=\mathcal{A}^{(s)}(x_+)$ for every $s \in [N]$. When $Q$-updates do not allow for symmetry, e.g. in \eqref{eq:am_default_update}, we note that this would induce equal weights $Q$ and $Q_+$. Then, as highlighted above, two consecutive calls to the black-box with inputs $Q$ and $Q_+$ where $Q=Q_+$ will provide the same $x^*_{|Q}=x^*_{|Q_{+}}$ so that a stopping criteria based on the distance between consecutive iterates will halt.
\end{remark}
\vspace{-5pt}
\begin{figure}[h]
\centering
\includegraphics[width=0.8\textwidth]{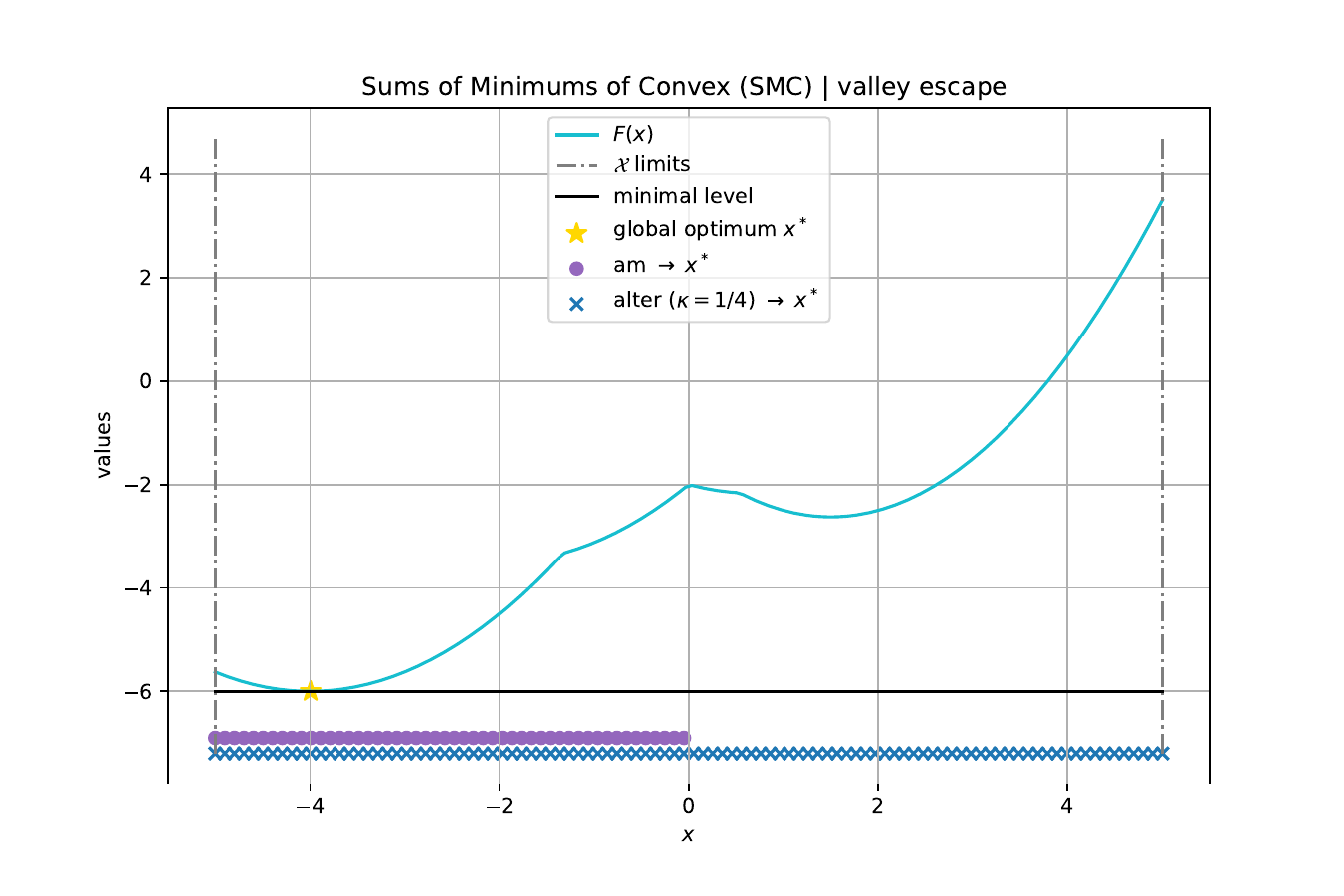}
\vspace{-20pt}
\caption{A first glimpse of relaxed Alternating Minimization (\texttt{r-AM}) methods}
\label{fig:SMC_valley_escape}
\end{figure}
\vspace{-20pt}
\subsection{$Q$-updates} \label{subsec:q_updates}
We explain now the threefold weight update mechanism in \texttt{r-AM} methods. \\ Let $x_+ \in \mathcal{X}$, we will extensively use the shorthand 
$$ \mathbf{h}^{(s)}(x_+) := \big(h^{(s)}_1(x_+),\dots,h^{(s)}_{n_s}(x_+)\big)\in \mathbb{R}^{(n_s)} \quad\forall s \in [N].$$
Informally, $Q$-updates in \texttt{r-AM} are simply convex combinations between the weights associated with \texttt{AM}, i.e. \eqref{eq:am_default_update}, and sound \textbf{candidate} weights that take into account the value of \emph{all} the \emph{components}, e.g. \eqref{eq:prelude_softmin}, as motivated in Section \ref{subsec:rationale}. 
We get $Q_+ = (q_+^{(1)},\dots,q_+^{(N)}) \in \mathcal{Q}$ thanks to the following election procedure:
\begin{equation}
\forall s \in [N],\quad
\begin{cases}
(q^*)^{(s)} &\gets \arg \min_{\tilde{q}\, \in\, \Delta^{n_s}}\,\langle \tilde{q},\textbf{h}^{(s)}(x_+)\rangle \quad \text{\textcolor{magenta}{// AM weights}}  \\
\hat{q}^{(s)}_+ &\gets \text{\textbf{candidate} from }\Delta^{n_s} \quad \text{\textcolor{purple}{// exploration}} \\
q^{(s)}_+ &= \varepsilon^{(s)}\cdot\hat{q}^{(s)}_+ + (1-\varepsilon^{(s)})\cdot (q^*)^{(s)} \quad \text{\textcolor{blue}{// convex combination}}
\end{cases}
\label{eq:RAM_Q_update}
\end{equation}
where $\varepsilon^{(s)} \in [0,1]$ depicts the \emph{exploration ratio} of the update. If $\varepsilon^{(s)} = 1$ (full exploration) then $q_+^{(s)} = \hat{q}^{(s)}_+$ whereas if $\varepsilon^{(s)}=0$ (greedy) then $q_+^{(s)} = (q^*)^{(s)}$ like \texttt{AM}. \\ \noindent We illustrate \eqref{eq:RAM_Q_update} before moving on. On Figure \ref{fig:SMC_valley_escape}, remembering its description from Section \ref{subsec:rationale}, \texttt{ALTER}  can be seen as a \texttt{r-AM} method with $\varepsilon^{(1)}=\varepsilon^{(2)}=0$ every odd iteration and $\varepsilon^{(1)}=\varepsilon^{(2)}=1$ every pair iteration with $\hat{q}_+^{(1)}$ and $\hat{q}_+^{(2)}$ computed as \eqref{eq:prelude_softmin}.\\ \vspace{-7pt}\\
Now, we state a lemma (proof in Appendix \ref{safeguard_proof}) that provides an explicit rule to set the parameters $\varepsilon = (\varepsilon^{(1)},\dots,\varepsilon^{(N)})$ so that a sufficient decrease of $\bar{F}(x_+,\cdot)$ is guaranteed as long as the \emph{gain} $\mathcal{G}^*_{(x_+,Q)}>0$, i.e. the weight update in \texttt{AM} itself decreases $\bar{F}(x_+,\cdot)$.
\begin{lemma}[Exploration Bound] \label{lemma:exploration_bound}
Let $C \in [0,1]$. Imposing, for every $s \in [N]$,
\begin{equation}
\varepsilon^{(s)} = \min\big\{1,C\cdot \langle q^{(s)}-(q^*)^{(s)},\textbf{h}^{(s)}(x_+)\rangle / \langle \hat{q}^{(s)}_+ - (q^*)^{(s)},\textbf{h}^{(s)}(x_+)\rangle  \big\}
\label{eq:max_safe_exploration}
\end{equation}
in \eqref{eq:RAM_Q_update} ensures that $ \bar{F}(x_+, Q) -  \bar{F}(x_+, Q_+) \geq (1-C)\cdot \mathcal{G}^*_{(x_+,Q)}$.
\end{lemma}
\newpage 
\paragraph*{Candidates}
Although the possibilities to choose \textbf{candidates} in \eqref{eq:RAM_Q_update} are limitless, we propose three different options that appear to work well in practice, compared to the alternating minimization \texttt{AM} baseline. Of course, some fine tuning has been attempted to set their hyper-parameter $\kappa \geq 0$ right. Nevertheless, to stay fair with respect to the parameter-free \texttt{AM}, throughout all our numerical experiments, we will adopt a fixed schedule $\{\kappa_k\}_{k \in \mathbb{N}}$ for the hyper-parameter value (see Table \ref{tab:candidates}).\\ \vspace{-20pt} \\
\begin{table}[H]
\centering
  \renewcommand*{\arraystretch}{2.1}
\begin{tabular}{|c||c|c|}
 \hline
\emph{method} & $\hat{q}_+^{(s)}$ \textbf{expr.}   & $\kappa$-\textbf{schedule}   \\
 \hline
 \hline
&\vspace{-5pt} $\textbf{proj}_{\Delta^{n_s}}(q^{(s)}+\kappa \cdot u^{(s)})$   &  \\
Baratt-Boyd (\texttt{BB}) & $ u_l^{(s)} = \begin{cases} 1 & \text{if } l = \min \mathcal{A}^{(s)}(x_+) \\ -1 & \text{otherwise}\end{cases}$ \vspace{-25pt} & $\kappa_k = 10^{-1}$ \\
&  & \\
\hline
 & $\textbf{softmin}\Big(\kappa\cdot\frac{\textbf{h}^{(s)}(x_+)\,+\,u^{(s)}}{\max\{10^{-4},|\langle \mathbf{1}_{n_s},\textbf{h}^{(s)}(x_+)\rangle|\}}\Big)$   &  \\
 (Norm.) Softmin (\texttt{SM})  & $ u^{(s)} \sim \textbf{Uni}\big([-5\cdot10^{-7} , 5\cdot 10^{-7}]^{n_s}\big)$ \vspace{-25pt} & $\kappa_k = \big(\frac{3}{2}\big)^{k^{\frac{3}{4}}}$ \\
&  & \\
\hline
\multirow{2}{*}{Max-Min (\texttt{MM})}   & \multirow{2}{*}{$\textbf{proj}_{\Delta^{n_s}}\Big(\kappa\cdot \frac{\textbf{1}_{n_s}\cdot\max_{l \in [n_s]}\,h^{(s)}_l(x_+)-\textbf{h}^{(s)}(x_+)}{\max_{l \in [n_s]}\,h^{(s)}_l(x_+) - \min_{l \in [n_s]}\,h^{(s)}_l(x_+)}\Big)$}  & \multirow{2}{*}{$\kappa_k = k^{\frac{2}{3}}$}  \\
&& \\
\hline
\end{tabular}\\ \vspace{5pt}
\caption{Candidates for \emph{exploration} in $Q$-updates.}
\label{tab:candidates}
\end{table}
\remark The Baratt-Boyd (\texttt{BB}) candidate above stands as our extension of the inexact $Q$-update introduced in \cite{Baratt20}, i.e. equation \eqref{eq:iAM_original} when $n_s=2$ for every $s\in [N]$. Unlike (\texttt{SM}) and (\texttt{MM}), (\texttt{BB}) candidate might fail to preserve the order of the values of the \emph{components}, i.e. the following property for every $s\in [N]$:
\begin{equation}
h^{(s)}_{l_2} (x_+)\leq h^{(s)}_{l_1}(x_+) \Rightarrow (\hat{q}^{(s)}_+)_{l_2}\geq (\hat{q}^{(s)}_+)_{l_1} \quad \quad \forall (l_1,l_2) \in [n_s]^2\label{eq:order_preserved}
\end{equation}
This is due to the fact that (\texttt{BB}) is not memoryless, its update involves the last $Q$ used. It might happen, especially in early iterations when iterates are still far from convergence, that there exists $s\in[N]$ such that $(\hat{q}^{(s)}_+)_l = q^{(s)}_l + \kappa < (\hat{q}^{(s)}_+)_{\tilde{l}} = q^{(s)}_{\tilde{l}}-\kappa$ although $l$ would be the smallest element of $\mathcal{A}^{(s)}(x_+)$. Indeed, it would occur if the weight previously put on $\tilde{l}$, i.e. $q_{\tilde{l}}^{(s)}$, was significant enough, i.e. $q_{\tilde{l}}^{(s)} > q^{(s)}_l + 2\kappa$. \\ \vspace{-5pt}\\
\noindent Max-Min (\texttt{MM}) maps $\mathbf{h}^{(s)}(x_+)$ to $v^{(s)} \in [0,1]^{n_s}$ where $v^{(s)}_l=1$ (respectively $0$) indicates that $h^{(s)}_l(x_+)$ achieves the minimal (respectively maximal) value of  $\mathbf{h}^{(s)}(x_+)$. Then, the resulting $\kappa\cdot v^{(s)}$ is projected back on the simplex $\Delta^{n_s}$. Property \eqref{eq:order_preserved} holds. \\
\vspace{-5pt}
\\
\noindent Normalized Softmin (\texttt{SM}) is \eqref{eq:prelude_softmin} but the input argument is divided by the (absolute of the) sum of the values of the \emph{components}, i.e. $|\langle \mathbf{1}_{n_s}, \mathbf{h}^{(s)}(x_+)\rangle|$, if not too small. \\We underline that the small stochastic perturbations $\{u^{(s)}\}_{s\,\in\, [N]}$ aim at breaking close ties among the values of the \emph{components}. This promotes candidates fulfilling \begin{equation}\min_{q \,\in\,\{0,1\}^{n_s}}\,||\hat{q}^{(s)}_+-q|| \to 0 \quad\text{as}\quad \kappa \to \infty. \label{eq:converging_candidate} \end{equation}
\subsection{Convergence}
We start by providing the pseudocode of \texttt{r-AM} methods in Algorithm \ref{alg:ram}.
\begin{algorithm} [H]
\caption{$\texttt{r-AM}\, (Q_{\text{init}})$}\label{alg:ram}
\begin{algorithmic}
\Require $\{C_k\}_{k\in\mathbb{N}}$, $\{\kappa_k\}_{k \in \mathbb{N}}$,  tolerance $\delta>0$, $Q_1 = Q_{\text{init}} \in \mathcal{Q}$, $\upsilon = \infty$, $K_{\text{max}}\in\mathbb{N}$
\For{$k=1,\dots,K_{\text{max}}$}
\State $x_{k} \gets \arg \min_{\tilde{x} \in\mathcal{X}}\,\bar{F}(\tilde{x},Q_{k})$ \quad\quad\quad\quad\quad\quad\quad\quad\quad// $x$-update
\For{$s=1,\dots,N$}\quad\quad\quad \quad\quad\quad\quad\quad\quad\quad\quad\quad\hspace{1pt}// $Q$-update
\State $(q^*)^{(s)} \gets \arg \min_{\tilde{q}\, \in\, \Delta^{n_s}}\,\langle \tilde{q},\textbf{h}^{(s)}(x_+)\rangle  $
\State $\hat{q}^{(s)}_+(\kappa_k) \gets \text{candidate from }\Delta^{n_s} $ \quad \quad \quad \quad \quad\hspace{3pt}// cfr. Table \ref{tab:candidates}
\State $ q^{(s)}_+ = \varepsilon^{(s)}\cdot \hat{q}^{(s)}_+ + (1-\varepsilon^{(s)})\cdot (q^*)^{(s)}$ \quad\quad\quad \hspace{2pt} // $\varepsilon^{(s)}(C_k,Q_k, \hat{Q}_+, Q^*)$
\EndFor
\If{$\upsilon-F(x_k)<\delta$}
\Return$ x$ such that $F(x) = \min_{\tilde{k}=1,\dots,k} \,F(x_{\tilde{k}})$
\EndIf
\State $(Q_{k+1},\upsilon) \gets (Q_+,\bar{F}(x_k,Q_{k}))$
\EndFor\\
\Return $ x$ such that $F(x) = \min_{\tilde{k}=1,\dots,k} \,F(x_{\tilde{k}})$
\end{algorithmic}
\end{algorithm}
\vspace{-5pt}
\noindent By choosing $(K_{\text{max}}, \delta) =(\infty, 0)$ in Algorithm \ref{alg:ram}, one would like to (asymptotically) converge towards \emph{critical points}. Under such setting, using the values for $\varepsilon=(\varepsilon^{(1)},\dots,\varepsilon^{(N)})$ prescribed in Lemma \ref{lemma:exploration_bound}, we state in Theorem \ref{theorem:main_conv} sufficient conditions for \texttt{r-AM} methods, independently from the choice of \textbf{candidate}, to converge in objective value. We also give a characterization of cluster points.
\begin{theorem}[\texttt{r-AM} - convergence]\label{theorem_ram} Let $Q_{\text{init}}\in \mathcal{Q}$, $\delta=0$ and $K_{\text{max}}=\infty$ within Algorithm \ref{alg:ram}. Let the values $\varepsilon^{(s)}$ follow the rule \eqref{eq:max_safe_exploration} for every $s\in [N]$ with $q^{(s)}=q^{(s)}_k$ and parameter $C$ chosen to be $C_k$ for every iteration $k\in\mathbb{N}$, let $\bar{C}:=\sup_{k\,\in\,\mathbb{N}}\,C_k<1$ and let $\hat{F}:=\bar{F}(x_1,Q_1)<\infty$. The sequence $\{(x_k,Q_k)\}_{k\in\mathbb{N}}$ produced by \texttt{r-AM} satisfies
\begin{itemize}
\item[(i).] $\{\bar{F}(x_k,Q_k)\}_{k \in \mathbb{N}}$ is non-increasing. 
\item[(ii).] Every cluster point of $\{x_k\}_{k \in \mathbb{N}}$ is a critical point of \eqref{eq:min_problem} and for every $K\geq2$
\vspace{-10pt}
\begin{equation}
\min_{k =1,\dots,K-1}\,\mathcal{G}^*_{(x_k,Q_k)} \leq \frac{\hat{F}-F^*}{\sum_{k=1}^{K-1}\,(1-C_k)} \leq \frac{\hat{F}-F^*}{(1-\bar{C})\cdot(K-1)} \label{eq:guarantee_criticality}
\end{equation}
\end{itemize}
\label{theorem:main_conv}
\end{theorem}
\vspace{-15pt}
\begin{proof}
(i). Let $k \in \mathbb{N}$ be fixed and let us show that $\bar{F}(x_k,Q_k)-\bar{F}(x_{k+1},Q_{k+1})\geq0$. 
$$\bar{F}(x_k,Q_k)-\bar{F}(x_{k+1},Q_{k+1})= \underbrace{\bar{F}(x_{k},Q_k)-\bar{F}(x_{k},Q_{k+1})}_{(1-C_k)\,\cdot \,\mathcal{G}^*_{(x_{k},Q_k)} \,(\text{Lemma}\,\ref{lemma:exploration_bound})} + \underbrace{\bar{F}(x_k,Q_{k+1})-\bar{F}(x_{k+1},Q_{k+1})}_{\max_{\tilde{x}\,\in\,\mathcal{X}}\,\bar{F}(x_k,Q_{k+1})-\bar{F}(\tilde{x},Q_{k+1})}$$
Since $x_k \in \mathcal{X}$ by construction, the two terms above are both nonnegative. \\
(ii). For any $(x,Q) \in \mathcal{X}\times\mathcal{Q}$, we have $\bar{F}(x,Q)\geq F(x)\geq F^*>-\infty$ hence one deduces that $\bar{F}$ is bounded from below. Summing the above equation for $k=1$ to $K-1$ yields
\vspace{-10pt}
\begin{align*}
\bar{F}(x_1,Q_1)-\bar{F}(x_K,Q_K) = \sum_{k=1}^{K-1}\,\bar{F}(x_k,Q_k)-\bar{F}(x_{k+1},Q_{k+1})&\geq \sum_{k=1}^{K-1}\,(1-C_k)\,\cdot\,\mathcal{G}^*_{(x_{k},Q_k)} \\
\end{align*}
from which \eqref{eq:guarantee_criticality} readily emerges as $\bar{C}<1$. Now, let us denote by $(x^{\perp},Q^{\perp})$ an arbitrary cluster point of the sequence $\{(x_k,Q_k)\}_{k\, \in\, \mathbb{N}}$. We would like to show that $x^{\perp}$ must be \emph{critical}. 
Let us assume by absurdum that $\mathcal{G}^*_{(x^{\perp},Q^{\perp})} \geq 2  \iota > 0$.
By continuity of $\mathcal{G}^*_{(x,Q)}$ as a function with input arguments $(x,Q) \in \mathcal{X}\times \mathcal{Q}$ and the definition of cluster point in first-countable sets, there exists an infinite subset of indices $\mathcal{K} \subseteq \mathbb{N}$ generating the subsequence $\{(x_{\tilde{k}},Q_{\tilde{k}})\}_{\tilde{k}\, \in \,\mathcal{K}}$ along which elements are close enough from $(x^\perp,Q^\perp)$ such that $\mathcal{G}^*(x_{\tilde{k}},Q_{\tilde{k}}) \geq \iota$. Hence, one would conclude that $\lim_{\tilde{k} \to \infty} \, \bar{F}(x_{\tilde{k}},Q_{\tilde{k}}) = - \infty$ (by subtracting at least $(1-\bar{C})\cdot \iota>0$ infinitely many times from $\hat{F}=\bar{F}(x_1,Q_1)$) which is impossible since $\bar{F}$ is bounded from below. Therefore, it holds that $\mathcal{G}^*_{(x^{\perp},Q^{\perp})}=0$, showing by Proposition \ref{lemma:criticality} that $x^\perp$ must be \emph{critical}.
\end{proof}
\paragraph*{Stopping criterion} Let $\delta >0$ for Algorithm \ref{alg:ram}. The main for-loop exits if quantity $\bar{F}(x_k,Q_k)-F(x_{k+1})$ falls below the threshold $\delta$, which in turn would imply that $\bar{F}(x_k,Q_k)-\bar{F}(x_{k+1},Q_{k+1}) \leq \delta$. In such circumstances, the guarantees at $x_k$, in terms of \emph{gains} (cfr. Proposition \ref{lemma:criticality}), become $$ \mathcal{G}^*_{(x_k,Q_k)} \leq \delta / (1-C_k). \vspace{-5pt}$$
\section{Numerical Experiments}
\label{sec:numerical_exp}
In this last section, we benchmark different local approaches on \emph{Piecewise-Linear Regression} (PLR), exactly as detailed in Example \ref{DC_fitting}, as well as on \emph{Restricted Facility Location} (RFL) problems somehow related to Example \ref{l2_cluster}. More specifically, we try out alternating minimization \texttt{AM} and make it compete against two of our \texttt{r-AM} methods, i.e. Algorithm \ref{alg:ram} with \textbf{candidates} chosen as (Norm.) Softmin (\texttt{SM}) and Max-Min (\texttt{MM}) (see Table \ref{tab:candidates}), and against our generalization (\texttt{BB}) of Baratt \& Boyd's work \cite{Baratt20}. Since \eqref{eq:min_problem} belongs to the less structured (DC) programming world, we also apply the standard (DC) algorithm, i.e. \texttt{DCA} (see e.g. \cite{LeThi18}), for comparison. Note that all the experiments were performed on real-world datasets. 
\paragraph*{Methodology}
We discuss here the implementations of the five aforementioned contenders. Except for \texttt{DCA}, every approach can be seen as a \texttt{r-AM} method with its own specifications in terms of $\varepsilon$ policies and $\{C_k\}_{k \in \mathbb{N}}$ schedules.
\begin{table}[H]
\centering
  \renewcommand*{\arraystretch}{1.2}
\begin{tabular}{|c||c|c|}
 \hline
\emph{method} & $\varepsilon$ rule   & $C$-\textbf{schedule}   \\
 \hline
 \hline
Baratt-Boyd (\texttt{BB}) & $\varepsilon = \mathbf{1}_{N}$ & /  \\
\hline
Alternating-Minimization (\texttt{AM})& $\varepsilon = \mathbf{0}_{N}$& / \\
\hline
Max-Min (\texttt{MM})  & \multirow{2}{*}{Lemma \ref{lemma:exploration_bound}}  & \multirow{2}{*}{$C_k = \frac{2}{\sqrt{k-1}+3}$}  \\
(Norm.) Softmin (\texttt{SM})  & & \\
\hline
\end{tabular}\\ \vspace{7pt}
\caption{Specifications of methods.}
\label{tab:specs}
\end{table}
\noindent For the (PLR) benchmark (respectively (RFL)), $50$ (respectively $150$) weights $Q_{\text{init}}$ were sampled uniformly at random on $\mathcal{Q}$ and fed as starting $Q$ iterates. We tracked the times taken by each method to terminate for $\delta = 10^{-8}$ and $K_{\text{max}}=400$, and recorded the best objective value obtained during each run. Interestingly enough, it turns out that usual \texttt{DCA} iterations for (PLR) are the same as \texttt{AM}. However, in what concerns (RFL), the two methods are no longer equivalent.
\newpage 
\subsection{Piecewise-Linear Regression} \label{sec:PLR} We chose the common datasets \href{https://archive.ics.uci.edu/dataset/1/abalone}{\underline{abalone}}, \href{https://archive.ics.uci.edu/dataset/186/wine+quality}{\underline{winequality}} and \href{https://www.kaggle.com/datasets/teertha/ushealthinsurancedataset}{\underline{insurance}} for this regression task. The raw data has first been worked out to transform the few categorical features into numerical ones and then to increase its dimensionality by incorporating
second-order interactions between explanatory features. That is, let $\tilde{\beta} \in \mathbb{R}^{\tilde{p}}$ be an observation in the original feature space, then we computed $\bar{\beta}\in \mathbb{R}^p$ with $p = \tilde{p}\cdot(3+\tilde{p})/2$ as 
\vspace{-7pt}
$$  \tilde{\beta} \rightarrow \bar{\beta}= (\tilde{\beta}, \tilde{\beta}_1 \cdot \tilde{\beta}_{1:p}, \tilde{\beta}_2\cdot \tilde{\beta}_{2:p},\dots,\tilde{\beta}_p\cdot \tilde{\beta}_{p:p}) \in \mathbb{R}^{\tilde{p}\cdot(3+\tilde{p})/2}$$ 
\vspace{-20pt}
\begin{center}
\begin{tabular}{l|c|c|c}
& insurance & winequality & abalone \\
 \hline
$p$ & $27$ & $77$ & $44$
\end{tabular}
\end{center}
In the end, we dispose of a dataset of tuples $(\bar{\beta}^{(s)},\gamma^{(s)}) \in \mathbb{R}^{p+1}$ for $s=1,\dots,N=750$. To predict $\gamma$ from $\bar{\beta}$, we fit the coefficients $x$ of a piecewise-linear model in $\bar{\beta}$, i.e.
\begin{align*}
\Delta \ell(\bar{\beta} \, ; x) &= \max_{e_1 \, \in \, [B_1]}\,\big\{\langle \bar{\beta}, x^{(1)}_{e_1} \rangle\big\} - \max_{e_2 \, \in \, [B_2]}\,\big\{\langle \bar{\beta}, x^{(2)}_{e_2} \rangle\big\}
\vspace{-40pt}
\end{align*}
where $B_1=6$, $B_2=5$ and $x=(x_1^{(1)},\dots,x_{B_1}^{(1)},x_1^{(2)},\dots,x_{B_2}^{(2)})$, by trying to minimize the averaged-$L_1$ prediction loss \eqref{eq:MAE_bis} on $\mathcal{X} = \mathbb{B}_{\infty}(\mathbf{0}_p\,;\,10^2)^{B_1 + B_2}$,
\vspace{-10pt}
\begin{equation}
\min_{x \,\in \,\mathcal{X}}\,\frac{1}{N}\,\sum_{s=1}^{N}\,|\gamma^{(s)} - \Delta \ell(\bar{\beta}^{(s)}\,;\, x)|.
\label{eq:MAE_bis}
\end{equation}
\vspace{-30pt}
\paragraph*{Results} As expected (cfr. Section \ref{subsec:rationale}), one can clearly observe on Figure \ref{fig:plr_experiment3} that \texttt{AM} usually finished the fastest but, most of the time, with the second worst average objective value after (\texttt{BB}). The latter was quite slow and did not manage to explore as much as the other contenders. We partially explain this by the number of $Q$-updates it takes (\texttt{BB}) to move from one vertex of $\mathcal{Q}$ to another when $n_s$ is big (here $n_s= B_1\cdot B_2=30$ for every $s\in [N]$, see Example \ref{DC_fitting}). On the other hand, this ability to try intermediate weights $Q$, i.e. not only vertices of $\mathcal{Q}$, also represents the exploration strength of the method compared to \texttt{AM}. Hence, a \emph{trade-off} symbolized by $\kappa$ should exist although our trials with $\kappa=1/10^\varsigma$ for $\varsigma \in \{1/4,1/2,1,2\}$ were not more successful. Based on Table \ref{tab:plr_results}, we can safely state that our methods (\texttt{MM}) and (\texttt{SM}) were the most competitive on the (PLR). One either chooses rapidity (\texttt{MM}) or better objective values (\texttt{SM}).  \\ \vspace{-2pt}
\begin{table}[H]
\centering
\begin{tabular}{c||c|c|c|c|c}
dataset &\emph{method} &avg. time [s] & min. value & avg. value & med. value\\
\hline
\multirow{4}{*}{insurance} &(\texttt{MM}) & 51.47& 0.0664& 0.0825 &0.0820\\
& \texttt{AM} / \texttt{DCA} &  \textbf{37.26} & \textcolor{red}{0.0704} &\textcolor{red}{0.0895} & \textcolor{red}{0.0895}\\
& (\texttt{BB}) & \textcolor{red}{219.34}&0.0632 & 0.0879 &0.0894\\
& (\texttt{SM}) &80.77& \textbf{0.0630}& \textbf{0.0779} &\textbf{0.0783}\\
\hline
\multirow{4}{*}{winequality} & (\texttt{MM})& 186.85& 0.0450& 0.0712 &0.0712\\
& \texttt{AM} / \texttt{DCA} &  \textbf{108.43} & 0.0638 &0.0906 & 0.0916\\
& (\texttt{BB}) & \textcolor{red}{659.53}&\textcolor{red}{0.0703} & \textcolor{red}{0.0913} &\textcolor{red}{0.0935}\\
& (\texttt{SM}) &479.71&\textbf{0.0175}&  \textbf{0.0407} &\textbf{0.0348}\\
\hline
\multirow{4}{*}{abalone} & (\texttt{MM}) & 112.41& 0.1686& 0.1886 &0.1893\\
& \texttt{AM} / \texttt{DCA} & \textbf{63.28} &0.1754&0.1976 & 0.1965\\
& (\texttt{BB}) & \textcolor{red}{439.67}&\textcolor{red}{0.1810} & \textcolor{red}{0.1995} &\textcolor{red}{0.2005}\\
& (\texttt{SM}) &186.03& \textbf{0.1611}&\textbf{0.1830} &\textbf{0.1808}
\end{tabular}
\vspace{10pt}
\caption{Piecewise-Linear Regression: results}
\label{tab:plr_results}
\end{table}
\newpage 
\begin{figure}[H]%
    \hspace{-30pt}
    \vspace{-20pt}
    \subfloat{{\includegraphics[scale=0.45]{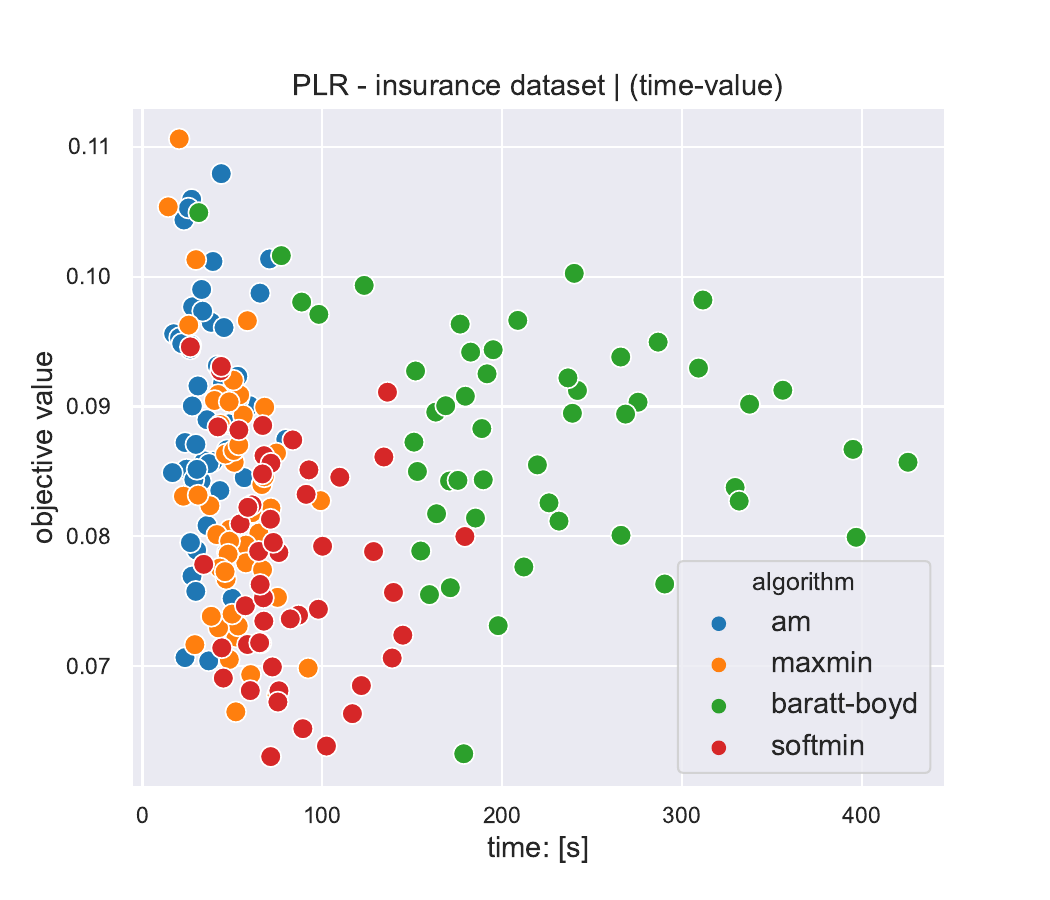} }}%
    \subfloat{{\includegraphics[scale=0.45]{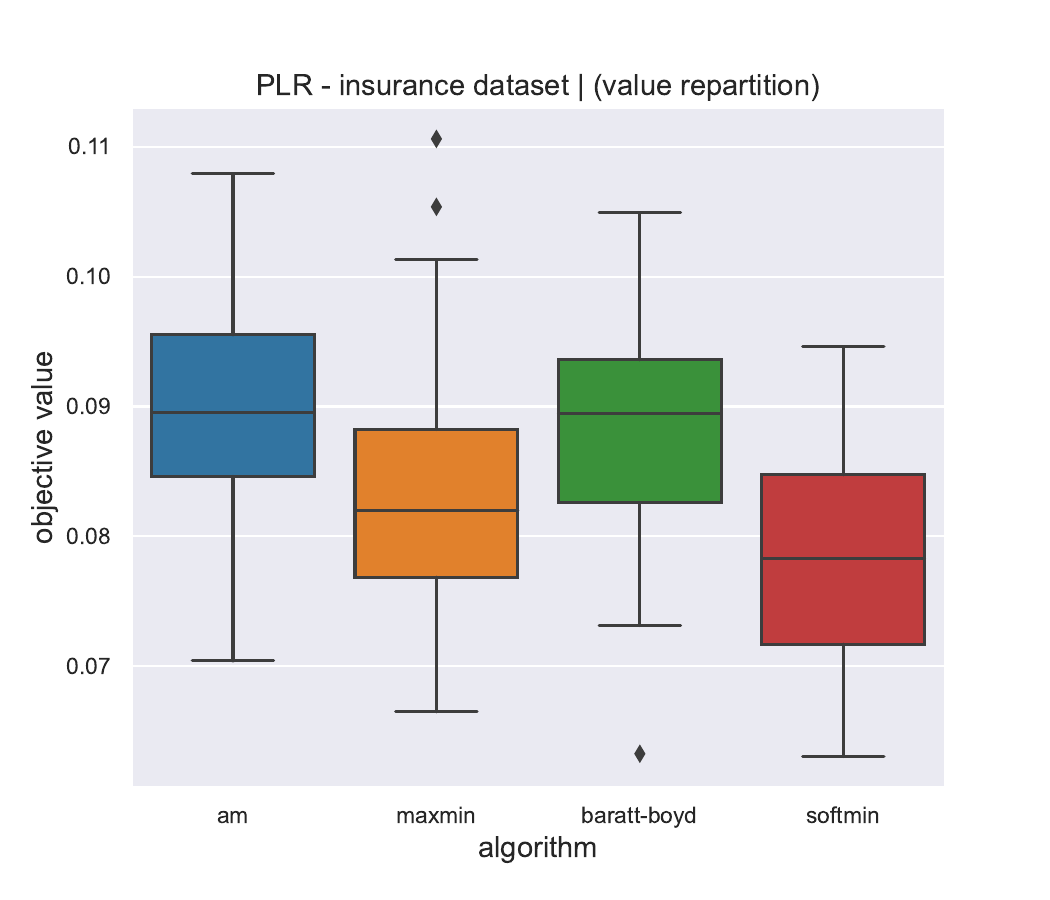} }}%
\label{fig:plr_experiment}%
\end{figure}
\begin{figure}[H]%
    \hspace{-30pt}
    \vspace{-25pt}
    \subfloat{{\includegraphics[scale=0.45]{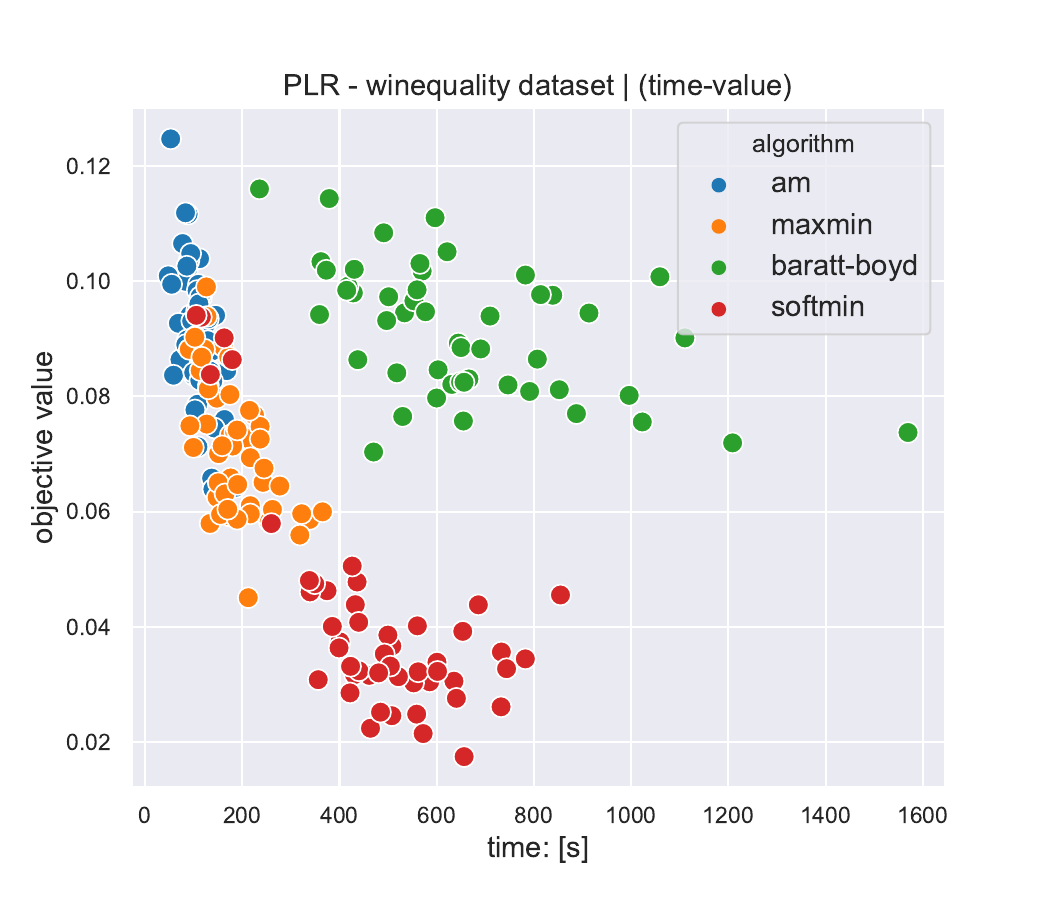} }}%
    \subfloat{{\includegraphics[scale=0.45]{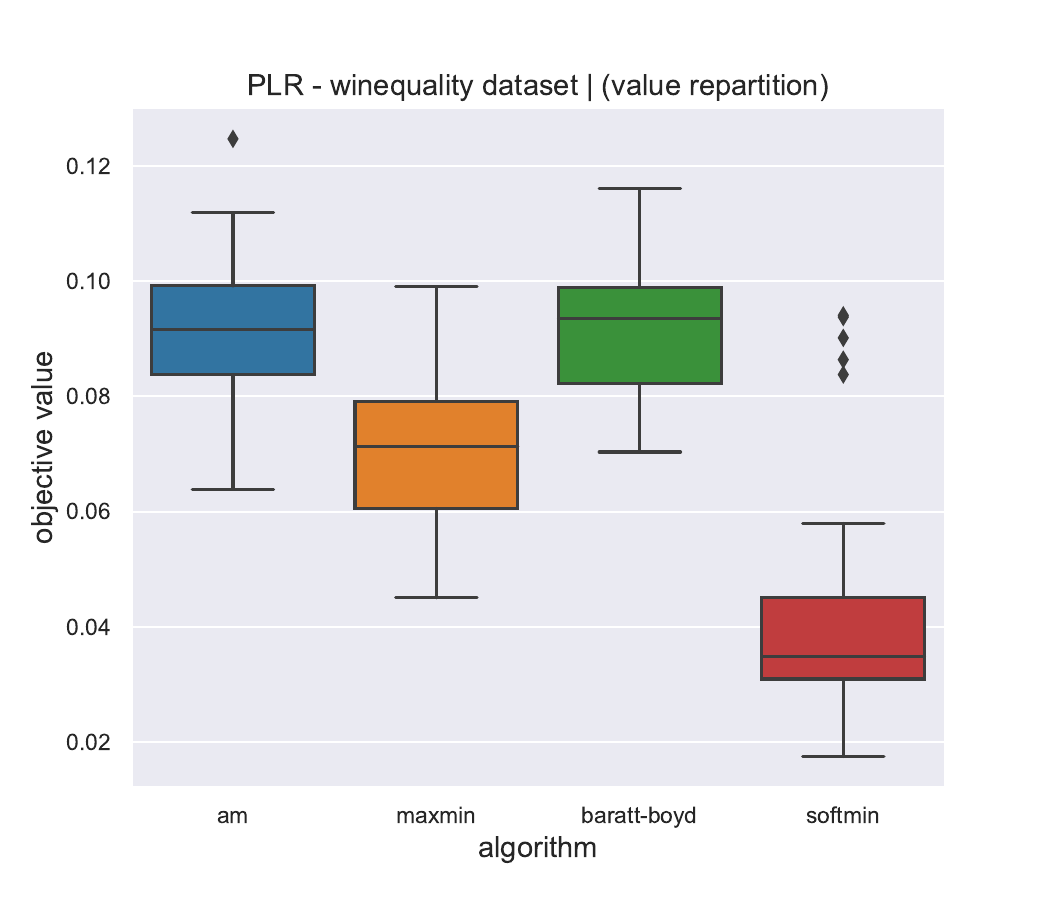} }}%
\label{fig:plr_experiment2}%
\end{figure}
\begin{figure}[H]%
    \hspace{-30pt}
    \vspace{-25pt}
    \subfloat{{\includegraphics[scale=0.45]{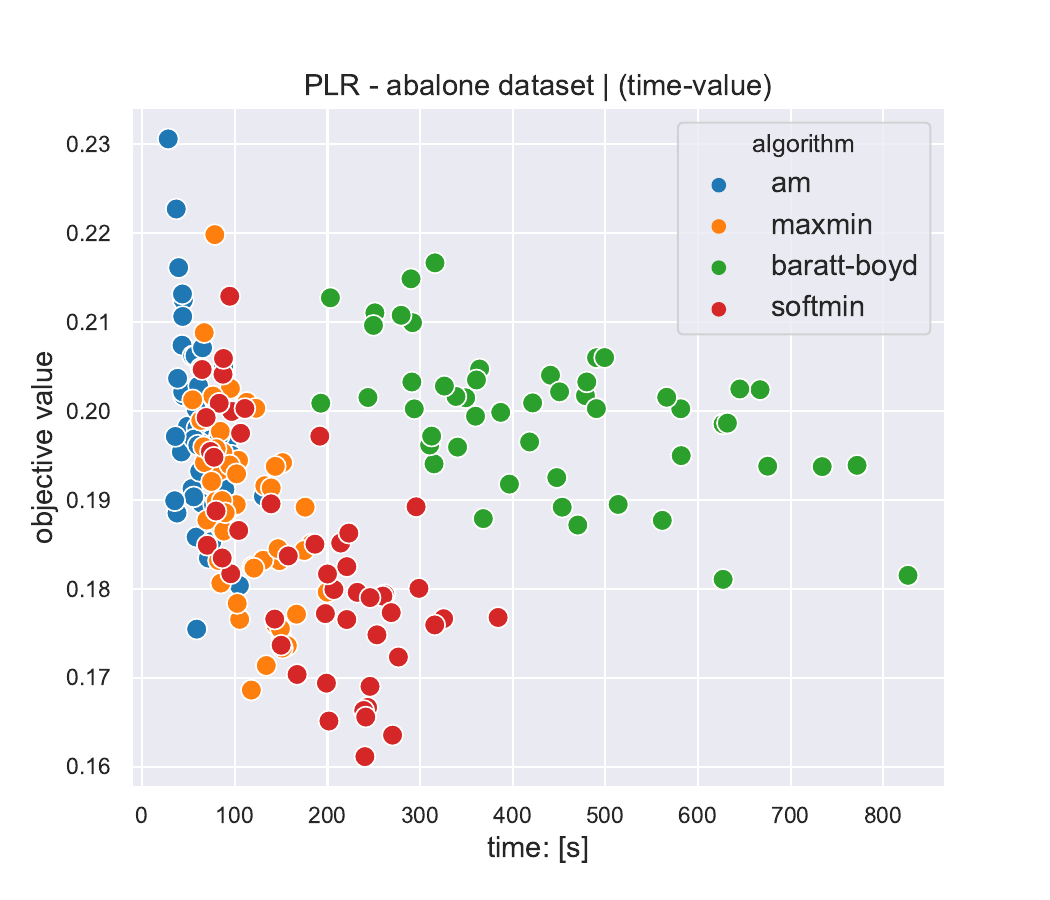} }}%
    \subfloat{{\includegraphics[scale=0.45]{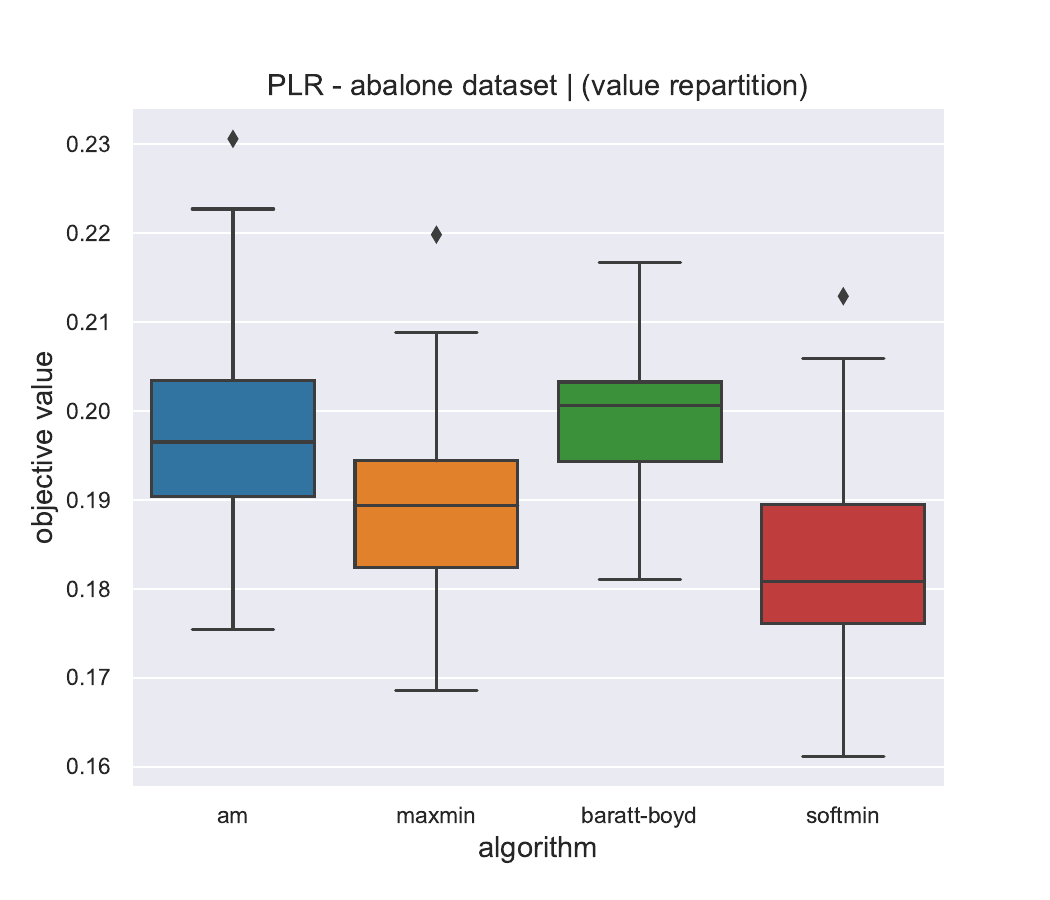} }}%
\vspace{15pt}
\caption{Piecewise-Linear Regression: (time-value) and (value distribution) plots}
\label{fig:plr_experiment3}%
\end{figure}

\subsection{Restricted Facility Location} \label{sec:RFL} We picked up the datasets \href{https://simplemaps.com/data/at-cities}{\underline{austria}}, \href{https://simplemaps.com/data/be-cities}{\underline{belgium}} and \href{https://simplemaps.com/data/cr-cities}{\underline{costarica}} for this facility location task. A dataset is composed of $N$ cities, each of which identified by an index $s \in [N]$ with respective population $P^{(s)} \in \mathbb{N}$ and coordinates $\beta^{(s)} \in \mathbb{R}^2$ (latitude, longitude). 
\vspace{2pt}
 \begin{center}
\begin{tabular}{l|c|c|c}
& austria & belgium & costarica   \\
 \hline
$N$ &$183$ & $53$ & $140$
\end{tabular}
\end{center}
\vspace{5pt}
One plans to build $B = 10$ stores, namely $\{x_{l}\}_{l \in [B]}$, and one distribution center, $x_0 \in \mathbb{R}^2$, so that considering the distance induced by $||\cdot||_1$ on $\mathbb{R}^2$,  \begin{itemize}
\item the averaged distance between populations and their closest store is minimized 
\item the stores are located within a $R_{\text{ref}}>0$ distance from the distribution center 
\end{itemize}
We set a cost by unit of excess distance $\Lambda>0$ to be incurred if one chooses to relax the proximity constraint above. In the spirit of \cite{Hamacher09}, the problem is formulated as follows
\begin{equation}
\min_{x \in \mathcal{X}}\,\Lambda\cdot \max\{R-R_{\text{ref}},0\} + \Bigg(\sum_{\tilde{s}=1}^{N}\,P^{(\tilde{s})}\Bigg)\,^{-1}\,\sum_{s=1}^{N}\,\min_{l\,\in\,[B]}\,P^{(s)}\,||\beta^{(s)}-x_l||_1
\label{eq:raw_facility}
\end{equation}
where the description of the basic feasible set is as follows,
$$\mathcal{X} = \Big\{(R,x_0,x_1,\dots,x_B) \in \mathbb{R}\times \mathbb{R}^2 \times (\mathbb{R}^2)^B\,|\,R>0,\,||x_0-x_l||_1\leq R \hspace{5pt} \forall l \in [B]\Big\}.$$

\paragraph*{Results} Table \ref{tab:rfl_results} displays the results obtained when $\Lambda = 0$, i.e. when we do not penalize the placement of stores outside the $R_{\text{ref}}$ distance window from the distribution center $x_0$. This time, performances are more spread out. Depending on the dataset, a different winner arises, e.g. \texttt{AM} dominated for the case \underline{austria}. Nevertheless, \texttt{DCA} clearly fell short suggesting it is worthwhile to use the structure behind \eqref{eq:big_smc} in \texttt{r-AM} methods. Looking closer to Table \ref{tab:rfl_results}, we point out that for datasets \underline{belgium} and \underline{costarica}, our (\texttt{SM}) and (\texttt{MM}) were again the best in terms of average/median objective value returned. Within these smaller scale problems (compared to (PLR)), (\texttt{BB}) slightly outperforms \texttt{AM}. Initially devised as $n_s=\bar{n}=2$ for every term $s\in [N]$ and beating \texttt{AM} in such case according to the authors of \cite{Baratt20}, it is not totally surprising to observe such behaviour as $\bar{n}$ drops from $30$ (PLR) to $10$ here. \\ \\
Interpreting Table \ref{tab:rfl_results_2} for which we picked $\Lambda = 10$, one draws relatively similar conclusions. The solutions returned by all methods were satisfying the implicit constraint $R\leq R_{\text{ref}}$, meaning that the penalty implemented by parameter $\Lambda$ was prohibitive enough. The median and minimal values obtained by each method among the $150$ trials were bigger for $\Lambda = 10$ than for $\Lambda =0$, as anticipated. A single exception exists when it comes to the performances of \texttt{AM} on \underline{costarica}'s dataset. We noticed that constraints dramatically changed the solutions about \underline{austria} and \underline{belgium} datasets. \\In the case of \underline{austria}, methods performed equally well and their respective performance variability diminished. \\ \\Overall, unlike for (PLR), one can observe that \texttt{AM} provided satisfactory results on (RFL) tasks taking into account its speed. The method with the most variability, (\texttt{BB}), provided the best one-shot solution in the majority of the cases. \newpage 

\begin{table}[h]
\vspace{60pt}
\centering
\begin{tabular}{c||c|c|c|c|c}
dataset &\emph{method} &avg. time [s] & min. value & avg. value & med. value\\
\hline
\multirow{5}{*}{austria} &(\texttt{MM}) & 1.0837 & 0.1279 & 0.3223 &  \textbf{0.2862}\\
& \texttt{AM} & \textbf{0.2821} &  \textbf{0.1008} &  \textbf{0.3078} & 0.3024\\
& \texttt{DCA} & \textcolor{red}{2.7869} & \textcolor{red}{0.1807} & 0.3268 & 0.3241\\
& (\texttt{BB}) &1.3152 & 0.1049 & 0.3280 & 0.3148\\
& (\texttt{SM}) &1.9238 & 0.1755 & \textcolor{red}{0.4144} & \textcolor{red}{0.3275}\\
\hline
\multirow{5}{*}{belgium} &(\texttt{MM}) & 0.6201 & 0.0924 & 0.1992 & 0.1901\\
& \texttt{AM} &   \textbf{0.1787} &  \textbf{0.0781} & 0.2257 & 0.2382\\
& \texttt{DCA} &  \textcolor{red}{3.3046} & \textcolor{red}{0.1131} & \textcolor{red}{0.2333} & \textcolor{red}{0.2438}\\
& (\texttt{BB}) &0.8764 & 0.0795 & 0.2237 & 0.2367\\
& (\texttt{SM}) &1.0632 & 0.0950 &  \textbf{0.1754} &  \textbf{0.1678}\\\hline
\multirow{5}{*}{costarica} &(\texttt{MM}) & 0.9845 & 0.1232 &  \textbf{0.1624} &  \textbf{0.1614}\\
& \texttt{AM} &  \textbf{0.3744} & 0.1242 & 0.1805 & 0.1829\\
& \texttt{DCA} & \textcolor{red}{6.6098} & \textcolor{red}{0.1474} & \textcolor{red}{0.2067} & \textcolor{red}{0.2113}\\
& (\texttt{BB}) & 2.1687 &  \textbf{0.1197} & 0.1699 & 0.1747\\
& (\texttt{SM}) &2.1629 & 0.1259 & 0.1670 & 0.1677\\
\end{tabular}
\vspace{10pt}
\caption{Restricted Facility Location: results for $\Lambda =0$}
\label{tab:rfl_results}
\vspace{50pt}
\end{table}

\begin{table}[H]
\centering
\begin{tabular}{c||c|c|c|c|c}
dataset &\emph{method} &avg. time [s] & min. value & avg. value & med. value\\
\hline
\multirow{5}{*}{austria} &(\texttt{MM}) & 1.0598 & \textbf{0.5463} & \textbf{0.5943} & \textbf{0.6041}\\
& \texttt{AM} & \textbf{0.3171} & 0.5556 & 0.6028 & 0.6054\\
& \texttt{DCA} & \textcolor{red}{4.6697} &  \textcolor{red}{0.5671} &  \textcolor{red}{0.6473} &  \textcolor{red}{0.6094}\\
& (\texttt{BB}) &1.3908 & 0.5548 & 0.5997 & 0.6058\\
& (\texttt{SM}) &1.8707 & 0.5540 & 0.6073 & 0.6079\\
\hline
\multirow{5}{*}{belgium} &(\texttt{MM}) &0.5360 & 0.1301 & 0.2188 & 0.2331\\
& \texttt{AM} & \textbf{0.2018} & 0.1545 & 0.2294 & 0.2357\\
& \texttt{DCA} & \textcolor{red}{3.1733} &  \textcolor{red}{0.1659} &  \textcolor{red}{0.2349} &  \textcolor{red}{0.2439}\\
& (\texttt{BB}) &0.9186 &\textbf{ 0.1197} & 0.2313 & 0.2367\\
& (\texttt{SM}) &0.9724 & 0.1304 & \textbf{0.1894} & \textbf{0.1856}\\\hline
\multirow{5}{*}{costarica} &(\texttt{MM}) &1.0066 & 0.1255 & \textbf{0.1680} & \textbf{0.1694}\\
& \texttt{AM} &  \textbf{0.4917} & 0.1311 & 0.1724 & 0.1746\\
& \texttt{DCA} & \textcolor{red}{7.7209} &  \textcolor{red}{0.1499} &  \textcolor{red}{0.2109} &  \textcolor{red}{0.2137}\\
& (\texttt{BB}) & 2.9921 & \textbf{0.1240} & 0.1710 & 0.1726\\
& (\texttt{SM}) &1.8744 & 0.1267 & 0.1707 & 0.1703\\
\end{tabular}
\vspace{10pt}
\caption{Restricted Facility Location: results for $\Lambda =10$}
\label{tab:rfl_results_2}
\end{table}

\newpage 
\begin{figure}[H]%
    \hspace{-30pt}
    \vspace{-20pt}
    \subfloat{{\includegraphics[scale=0.45]{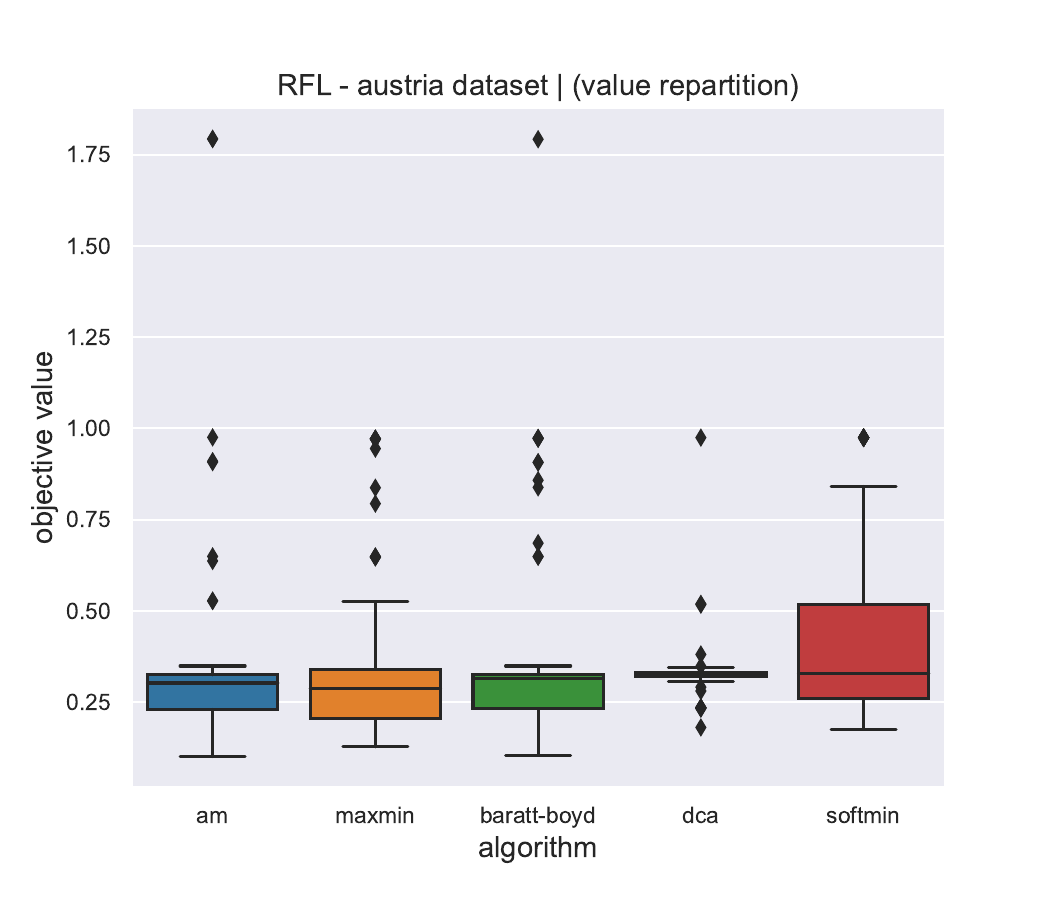} }}%
    \subfloat{{\includegraphics[scale=0.45]{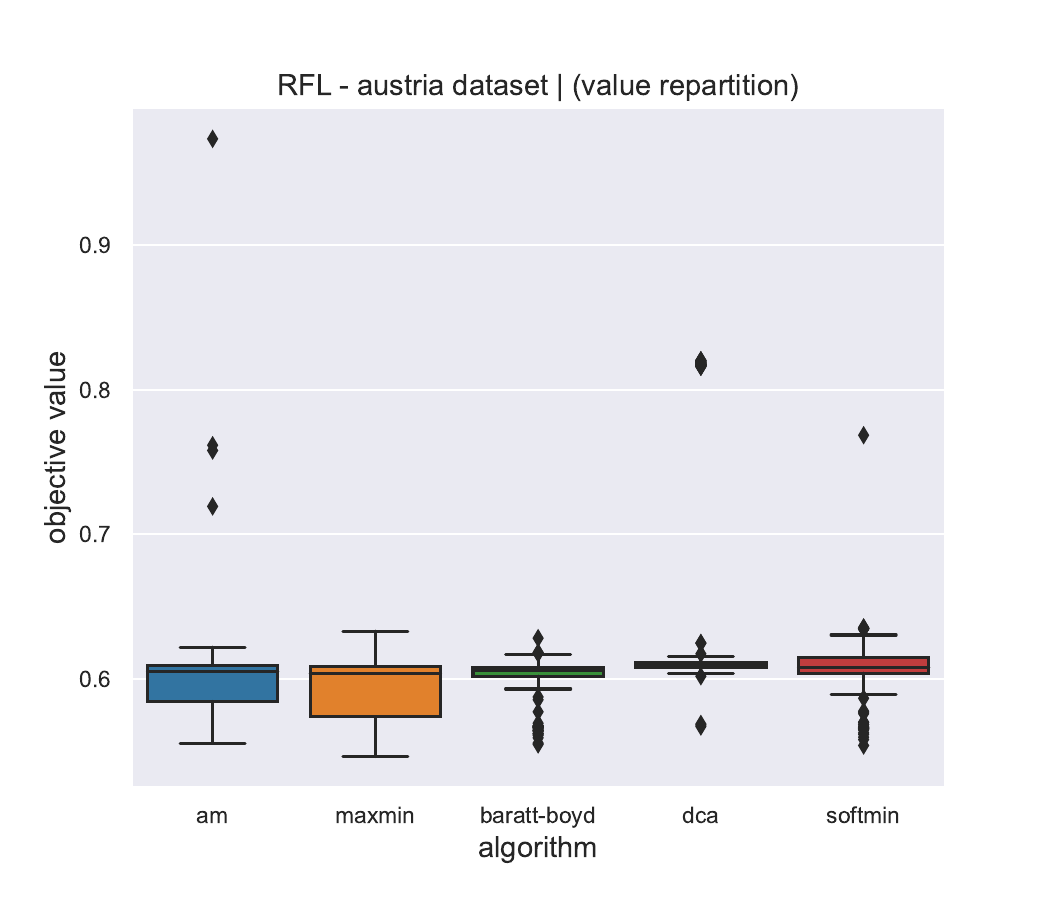} }}%
\label{fig:rfl_experiment}%
\end{figure}
\begin{figure}[H]%
   \hspace{-30pt}
    \vspace{-25pt}
    \subfloat{{\includegraphics[scale=0.45]{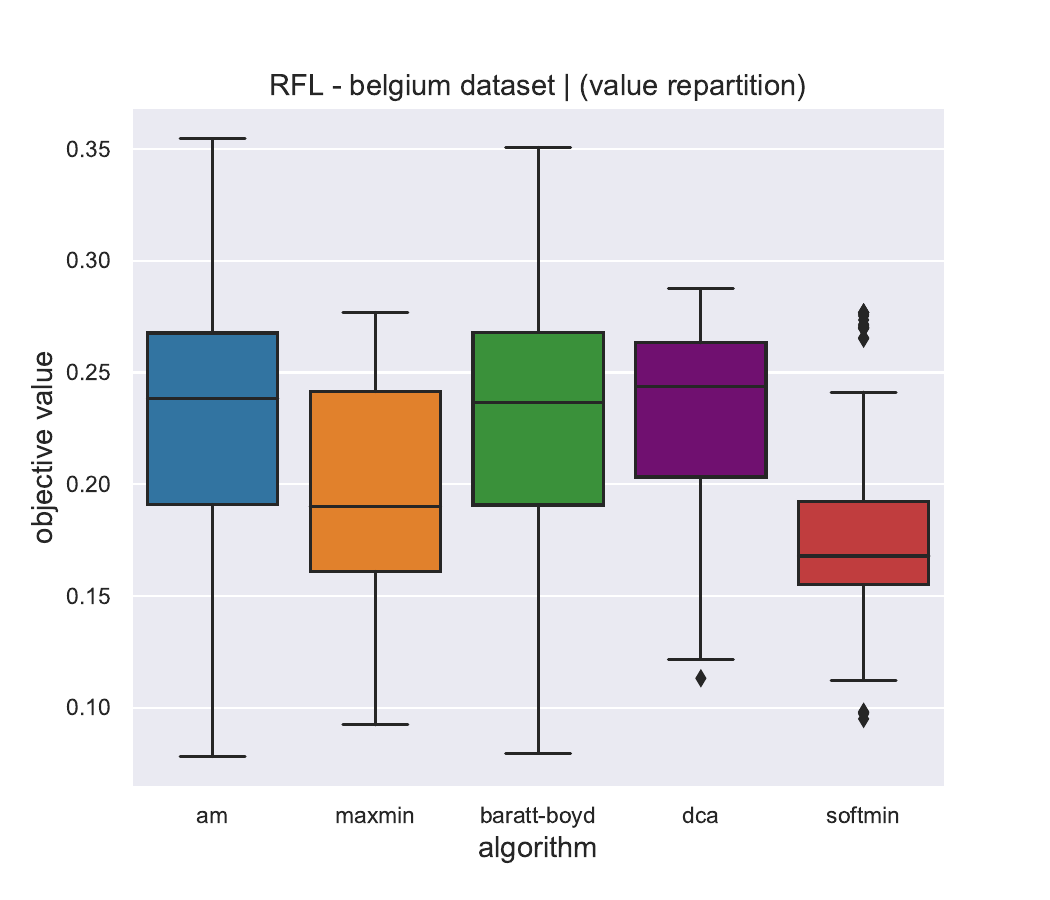} }}%
    \subfloat{{\includegraphics[scale=0.45]{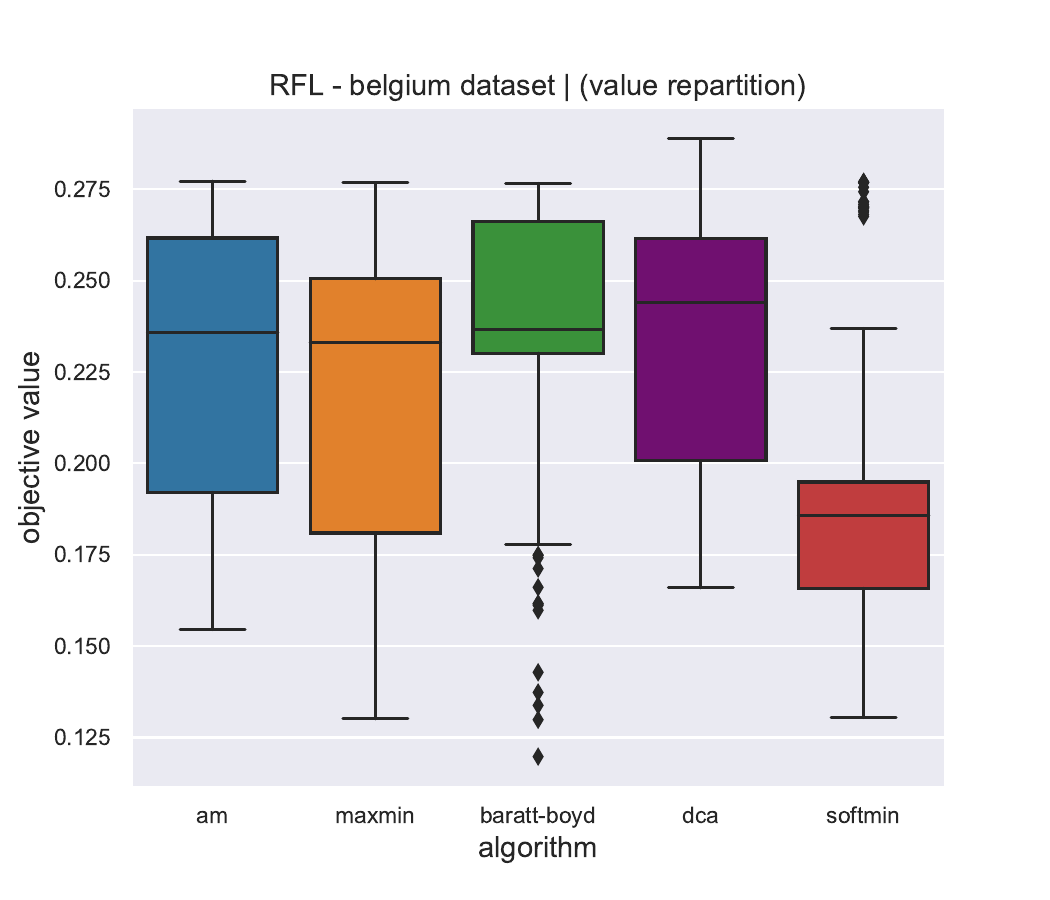} }}%
\label{fig:rfl_experiment2}%
\end{figure}
\begin{figure}[H]%
    \hspace{-30pt}
    \vspace{-25pt}
    \subfloat{{\includegraphics[scale=0.45]{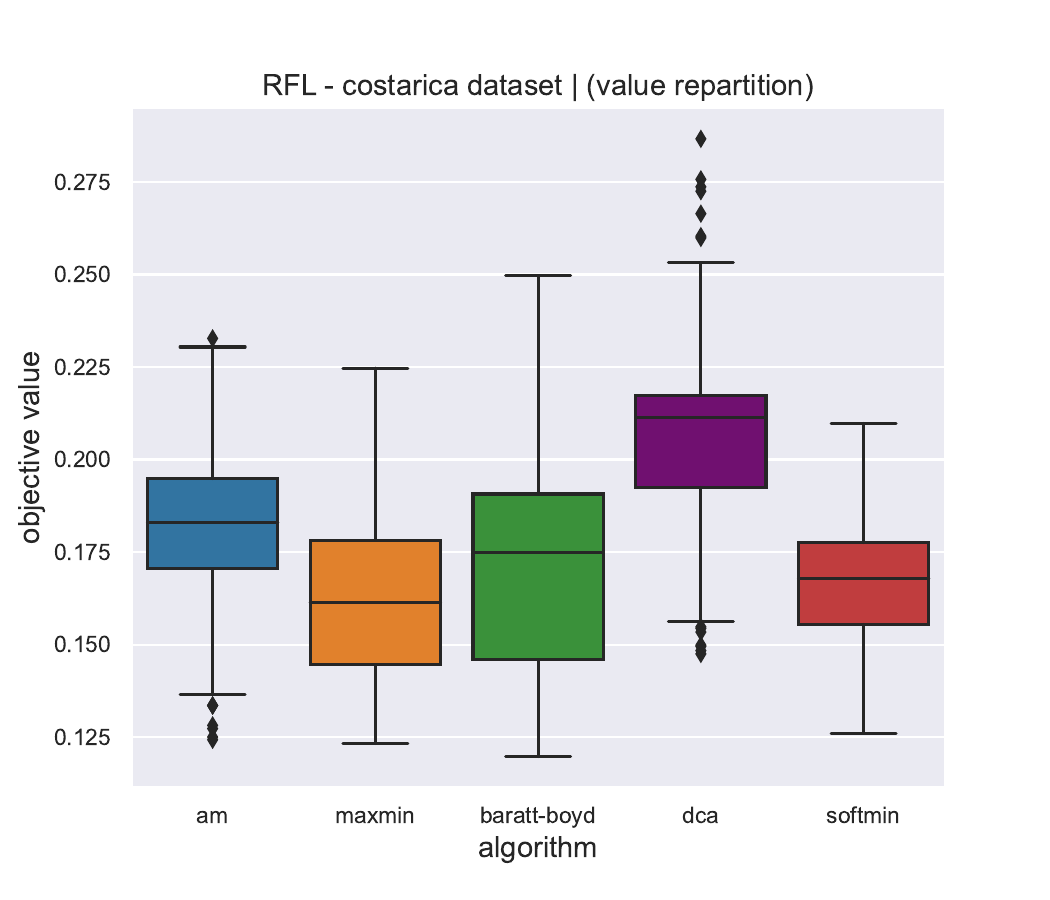} }}%
    \subfloat{{\includegraphics[scale=0.45]{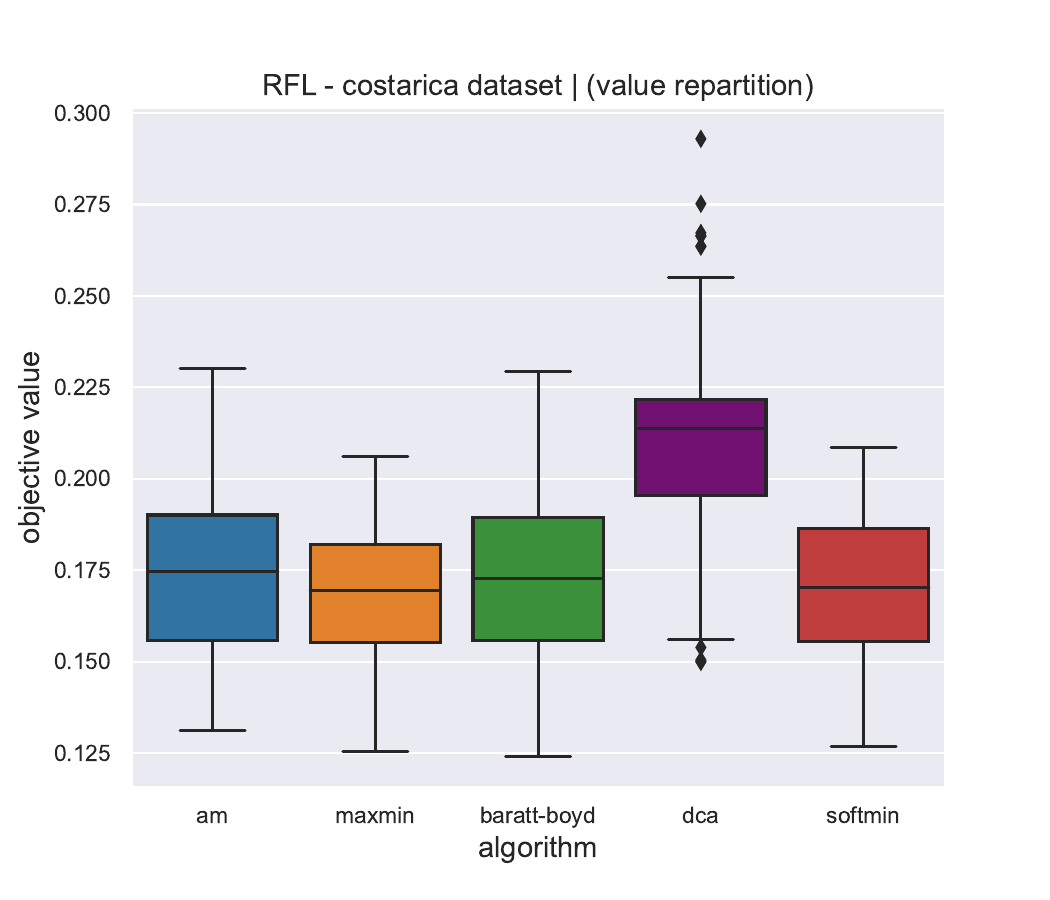} }}%
\vspace{15pt}
\caption{Restricted Facility Location: (value distribution) $\Lambda =0$ (left) and $\Lambda=10$ (right)}
\label{fig:rfl_experiment3}%
\end{figure}

\newpage

\subsection{Local Optimality Certifications} \label{subsec:lo_certif} Let $\hat{x} \in \mathcal{X}$ be the output iterate of a single run of any of the methods tested above. Fortunately enough for both (PLR) and (RFL), we can define a suitable neighbourhood $\hat{\mathcal{S}}$ around $\hat{x}$ so that the computation of $\hat{\mathcal{S}}$-bounds (see Definition \ref{def:s_bounds}) becomes closed-form. Then, according to Remark \ref{rem:choose_s_bounds},  we can try to solve problem \eqref{eq:new_min_problem_local_2} with $\mathcal{S} = \mathcal{X} \cap \hat{\mathcal{S}}$ from Corollary \ref{coro:restart_check} using these local $\hat{\mathcal{S}}$-bounds as the \emph{big-M} values. For numerical stability purposes, we used $\rho$-\emph{active sets} at $\hat{x}$ (see Definition \ref{def:active_set} with $\rho =10^{-12}$) instead of mere \emph{active sets} in formulation \eqref{eq:new_min_problem_local}.\\ \\
We fixed a time limit of $120$ [s] and a sufficient decrease threshold of $\delta_{\text{glob}} = 5 \cdot 10^{-7}$. \\ From there, two possible outcomes may happen: \begin{enumerate}\item the solver outputs a point $x_{\text{new}} \in \mathcal{S}$ such that $F(x_{\text{new}})\leq F(\hat{x})-\delta_{\text{glob}}$ 
\item the solver did not find such point
\end{enumerate}
When $(1)$ occurs, we restart the method with $q_{\text{init}}^{(s)} \in \arg \min_{\tilde{q}\,\in\,\Delta^{n_s}}\,\langle \tilde{q}, \textbf{h}^{(s)}(x_{\text{new}})\rangle$ for every $s\in [N]$, i.e. \texttt{AM} weights. We proceed all over again until the method reaches the next $\hat{x}$ approximate \emph{critical} point. When $(2)$ happens, we are done. Sometimes, it would mean that the solver found the optimal solution $x^*(\hat{x},\mathcal{S})$ of \eqref{eq:new_min_problem_local_2} and based on the value $F(x^*(\hat{x},\mathcal{S}))$, the \emph{local optimality} of $\hat{x}$ might be certified (see Corollary \ref{coro:restart_check}).

\paragraph*{$\hat{\mathcal{S}}$-bounds for (PLR)} Let $R>0$ and let $(||\cdot||,||\cdot||_*)$ denote a norm and its dual defined on $\mathbb{R}^p$.
Looking back to Section \ref{sec:PLR}, let $\hat{x} =  \big(\hat{x}_1^{(1)},\dots,\hat{x}_{B_1}^{(1)},\hat{x}_1^{(2)},\dots,\hat{x}_{B_2}^{(2)}\big)$. \\If one takes
\vspace{-5pt}
\begin{equation}\hat{\mathcal{S}} = \Bigg(\bigtimes_{e_1 \,\in\,[B_1]}\, \mathbb{B}_{||\cdot||}(\hat{x}^{(1)}_{e_1}\, ;\, R) \Bigg)\nonumber \bigtimes \Bigg(\bigtimes_{e_2 \,\in\,[B_2]}\, \mathbb{B}_{||\cdot||}(\hat{x}^{(2)}_{e_2}\, ;\, R)\Bigg), \label{eq:cf_s_plr}\end{equation} 
then valid $\hat{\mathcal{S}}$-bounds for (PLR) are given for for every $(l_+,l)\in[n_s]^2$ ($s \in [N]$) by
\begin{equation} M^{(s)}_{l_+,l} \gets h^{(s)}_{l_{+}}(\hat{x})-h^{(s)}_{l}(\hat{x}) + 2\cdot\Big(\mathbbm{1}_{\{e_1(l_+)\not=e_1(l)\}}+\mathbbm{1}_{\{e_2(l_+)\not=e_2(l)\}}\Big)\cdot\big(R \cdot ||\bar{\beta}^{(s)}||_*\big) \label{eq:big_M_plr} \end{equation}
where $e_1(l)  =  \lceil l/B_2\rceil$ and $e_2(l)  =  1+\big(l-1) \,\textbf{mod}\,B_2$.
\vspace{-3pt}
\paragraph*{$\hat{\mathcal{S}}$-bounds for (RFL)}
Let $R_\infty>0$. Sticking to the notations used in Section \ref{sec:RFL},\\ we can write $\hat{x} = \big(\hat{R},\hat{x}_0,\hat{x}_1,\dots,\hat{x}_B\big)$. 
If one takes
\begin{equation}\hat{\mathcal{S}} = [\hat{R}-R_\infty,\hat{R}+R_\infty] \bigtimes \Bigg(\bigtimes_{l=0}^{B}\,\mathbb{B}_{||\cdot||_\infty}(\hat{x}_{l}\, ;\, R_\infty)\Bigg)\label{eq:cf_s_rfl} 
\vspace{-5pt}
\end{equation}
then valid $\hat{\mathcal{S}}$-bounds for (PLR) are given for for every $(l_+,l)\in[n_s]^2$ ($s \in [N]$) by
\begin{align} M^{(s)}_{l_+,l} &\gets N\cdot P^{(s)} / \bar{P}\,\cdot\Big(||\beta^{(s)}-\hat{x}_{l_+} + R_\infty\cdot\textbf{sign}\big(\beta^{(s)}-\hat{x}_{l_+}\big) ||_1 \nonumber \\
 &- || \beta^{(s)}-\hat{x}_l -\textbf{min}\big(R_\infty\cdot \mathbf{1}_2,\,\textbf{abs}(\beta^{(s)}-\hat{x}_l)\big)\cdot \textbf{sign}\big(\beta^{(s)}-\hat{x}_l\big)||_1\Big)\label{eq:big_M_rfl} \end{align}
 where $\bar{P} = \sum_{s=1}^{N} P^{(s)}$ and $\textbf{sign}$ represents the vectorized version of $\text{sign} : \mathbb{R} \to \{-1,1\}$ defined for every scalar $u$ as 
 $\text{sign}(u) = 1$ if $u\geq 0$ or $\text{sign}(u)=-1$ otherwise.
\newpage 
\noindent As a final contribution, paving the way for further investigations, we tried the procedure described at the beginning of this subsection using only (\texttt{SM}), our most promising local-search technique. We conducted $5$ independent trials represented by $5$ initial weights $Q_{\text{init}}$ for every experiment. For each trial, we recorded the difference between the final best value obtained at the end of a minimization of \eqref{eq:new_min_problem_local_2} with the prescribed $\hat{\mathcal{S}}$-bounds (either \eqref{eq:big_M_plr} for (PLR) or \eqref{eq:big_M_rfl} for (RFL)) and the objective value obtained after the first halt of (\texttt{SM}) (i.e. after the initial \texttt{r-AM} run) within that trial.\\ This difference is a metric of performance and will be called \emph{value enhancement}. Note that if the \emph{value enhancement} is strictly positive, (\texttt{SM}) has been restarted at least once, i.e. it must have restarted after its first halt and the associated encounter of \eqref{eq:new_min_problem_local_2}. \\ \vspace{-5pt}\\
Using $R = 10^{-1}$, and $(||\cdot||, ||\cdot||_*) = (||\cdot||_1, ||\cdot||_\infty)$,  we present in the table below the averaged performances for the (PLR) experiments.

\begin{table}[H]
\centering
\begin{tabular}{c||c|c}
dataset  & avg. \emph{val. enhancement} [$\%$] &  \# \emph{local optimality} certifications \\
\hline
insurance & 4.79& \textcolor{Green}{5/5}\\
\hline
winequality & \textcolor{red}{0}& \textcolor{red}{0/5}\\
\hline
abalone  & 2.79& \textcolor{Blue}{2/5}\\ 
\hline
\end{tabular}
\vspace{10pt}
\label{tab:plr_results_extended}
\vspace{-10pt}
\end{table}
\noindent In what concerns (RFL) experiments, we fixed $R_\infty =2\cdot10^{-1}$. 
\begin{table}[H]
\centering
\begin{tabular}{c||c|c}
dataset  & avg. \emph{val. enhancement} [$\%$] &  \# \emph{local optimality} certifications \\
\hline
austria $(\Lambda =0$)& \textcolor{Green}{8.76}& \textcolor{Green}{5/5}\\
austria $(\Lambda =10$)& \textcolor{red}{0}& \textcolor{Green}{5/5}\\
\hline
belgium $(\Lambda =0$)& 1.03& \textcolor{Green}{5/5}\\
belgium $(\Lambda =10$)& 0.7& \textcolor{Green}{5/5}\\
\hline
costarica $(\Lambda =0$)& 1.52& \textcolor{Green}{5/5}\\
costarica $(\Lambda =10$)& \textcolor{red}{0}& \textcolor{Green}{5/5}\\
\hline
\end{tabular}
\vspace{10pt}
\label{tab:rfl_results_extended}
\vspace{-20pt}
\end{table}
\paragraph*{Results} The ability to certify \emph{local optimality} relies on the radii used for defining the neighbourhood $\hat{\mathcal{S}}$, e.g. $R_\infty$, as well as the capabilities of the (MICP) invoked. Also, it can happen that \texttt{r-AM} methods return points $\hat{x} \in \mathcal{X}$ that are not \emph{local minima}. Arguably, our restart methodology did not provide substantial gains in terms of objective value (i.e. \emph{value enhancement}) but, unexpectedly, managed to certify \emph{local optimality} for many instances, especially in lower-dimensional settings such as in (RFL).
\vspace{-15pt}
\section{Conclusion} 
\label{sec:conclusion}
\vspace{-8pt}
In this paper, we studied a class of nonconvex nonsmooth optimization problems involving the sum of pointwise minima of convex functions. We derived (MICP) reformulations of the original problem as well as new strong \emph{local optimality} conditions.\\ To circumvent the computational difficulty associated with (MICP), we devised a second, simple,  bi-convex reformulation of the problem. This latter allowed us to successfully develop a new framework for local-search methods that can be applied as good heuristics, beating most of the time baseline alternating minimization. On small neighbourhoods, our (MICP) formulation can be alleviated of many binary variables and proved to be useful to certify \emph{local optimality} in certain scenarios. It remains unclear to us how and when to use these reduced scale (MICP) formulations. A possible future research direction is to investigate a smoothing of the objective as a tool to render substantial portions of its landscape benign.
In addition, this would also make the problem amenable for second-order methods, not exploited in this context so far. 
\newpage 
\section*{Funding}

Guillaume Van Dessel is funded by the UCLouvain university as a teaching assistant. 
\vspace{-2pt}
\bibliographystyle{abbrv}
\bibliography{SMC.bib}
\newpage 
\appendix
\addcontentsline{toc}{section}{Appendices}

\section{(SMC) formulation of \emph{Piecewise-Linear $L_1$-Regression}}
    \label{dev_SMC}

As such, \eqref{eq:raw_DC_fitting} is a nonsmooth nonconvex optimization problem. One can actually rewrite \eqref{eq:raw_DC_fitting} in a more convenient way. Let us consider a tuple $(\gamma,\bar{\beta}) \in \mathbb{R} \times \mathbb{R}^p$. 
Now, let $\epsilon = |\gamma-\Delta \ell(\bar{\beta}\, ;\, x)|$. Two scenarios are possible, either $\epsilon = \epsilon_+$ or $\epsilon = \epsilon_{-}$ with
\begin{equation*}
\begin{cases}
\epsilon_{+} = \gamma +\ell_2(\bar{\beta}\,;\,x)-\ell_1(\bar{\beta}\,;\,x) \geq 0, \\
\epsilon_{-} = -\gamma +\ell_1(\bar{\beta}\,;\,x)-\ell_2(\bar{\beta}\,;\,x)  \geq 0. 
\end{cases}
\end{equation*}
We recall from Example \ref{DC_fitting} that $x = (x_1^{(1)},\dots,x_{B_1}^{(1)},x_1^{(2)},\dots,x_{B_2}^{(2)})$ and $$\ell_\star(\cdot\,;\,x) := \max_{e_\star\,\in\,[B_\star]}\,\langle \cdot, x_{e_\star}^{(\star)}\rangle\quad \quad \forall \star \in \{1,2\}.$$ 
It is easy to see that $\epsilon = \max\{0,\epsilon_+\} + \max\{0,\epsilon_{-}\}$, the sum of two polyhedral (DC) functions in $x$ hence itself (DC). Let us rewrite $\max\{0,\epsilon_+\}$ and $\max\{0,\epsilon_-\}$ as follows
\begin{align*}
\max\{0,\epsilon_+\} &=\max\big\{\gamma + \ell_2(\bar{\beta}\,;\,x),\ell_1(\bar{\beta}\,;\,x)\big\} -\ell_1(\bar{\beta}\,;\,x) \\
&=\max\big\{\gamma + \ell_2(\bar{\beta}\,;\,x),\ell_1(\bar{\beta}\,;\,x)\big\} + \min_{e_1\,\in\,[B_1]}\,-\langle \bar{\beta}, x_{e_1}^{(1)}\rangle\\
\max\{0,\epsilon_-\} &=\max\big\{-\gamma + \ell_1(\bar{\beta}\,;\,x),\ell_2(\bar{\beta}\,;\,x)\big\} -\ell_2(\bar{\beta}\,;\,x) \\
&=\max\big\{-\gamma + \ell_1(\bar{\beta}\,;\,x),\ell_2(\bar{\beta}\,;\,x)\big\} + \min_{e_2\,\in\,[B_2]}\,-\langle \bar{\beta}, x_{e_2}^{(2)}\rangle
 \end{align*}
Applying the above decomposition on every term of objective's sum yields
\begin{align*}
   F(x) &=\frac{1}{N}\,\sum_{s=1}^{N}\,\epsilon^{(s)}\\
   &=\frac{1}{N}\,\sum_{s=1}^{N}\, \max\{0,\epsilon_+^{(s)}\} + \max\{0,\epsilon_{-}^{(s)}\}\\
   &= \bar{h}(x) + \frac{1}{N}\,\sum_{s=1}^N\,\min_{(e_1,e_2)\,\in\,[B_1]\times[B_2]}\,-\langle \bar{\beta}^{(s)} , x_{e_1}^{(1)}+ x_{e_2}^{(2)}\rangle \\
   &= \bar{h}(x)+\frac{1}{N}\,\sum_{s=1}^{N}\,\min_{l\,\in\,[\bar{n}]}\,h_l\big(x\,|\,\bar{\beta}^{(s)}\big)
 \end{align*}
 with $\bar{n} = B_1\cdot B_2$, $e_1(l) = \lceil l/B_2\rceil$ and $e_2(l) = 1+(l-1) \,\textbf{mod} \,B_2$ for every $l \in [\bar{n}]$ and
  \begin{align}
  \bar{h}(x) = &\, \frac{1}{N}\,\sum_{s=1}^N 
 \max\Big\{\gamma^{(s)} + \ell_2(\bar{\beta}^{(s)}\,;\,x),\ell_1(\bar{\beta}^{(s)}\,;\,x)\Big\} \nonumber \\&\,+ \max\Big\{-\gamma^{(s)} + \ell_1(\bar{\beta}^{(s)}\,;\,x),\ell_2(\bar{\beta}^{(s)}\,;\,x)\Big\}, \nonumber \\
h_l(x\,|\,\bar{\beta}) =& -\big\langle \bar{\beta}, x^{(1)}_{e_1(l)}+x^{(2)}_{e_2(l)}\big\rangle.  \nonumber
\end{align}
\hspace{300pt} \qed

\newpage 

\section{Criticality revisited} \label{crit_proof}

\begin{proposition}
Let $Q \in \mathcal{Q}$ and $x_+ \in \arg \min_{\tilde{x} \in \mathcal{X}}\,\bar{F}(\tilde{x},Q)$. If there is no gain, i.e. $$\mathcal{G}^*_{(x_+, Q)} := \bar{F}(x_+, Q) - \min_{\tilde{Q}\, \in\, \mathcal{Q}}\, \bar{F}(x_+, \tilde{Q}) =0 $$ then $x_+$ is a critical point for \eqref{eq:min_problem}.
\end{proposition}
\begin{proof}
In the absence of gain, it comes $Q \in \arg \min_{\tilde{Q}\,\in\,\mathcal{Q}}\,\bar{F}(x_+,\tilde{Q})$. \\We first notice that for every $s\in [N]$ and $l \in [n_s]$, $q^{(s)}_l>0 \Rightarrow l \in \mathcal{A}^{(s)}(x_+)$. \\Otherwise, there would exist $\bar{s} \in [N]$, $\bar{l} \in [n_s]$ such that $q^{(\bar{s})}_{\bar{l}}>0$ with $\bar{l}$ not in $\mathcal{A}^{(\bar{s})}(\hat{x})$. \\One could then spread the weight of $q^{(\bar{s})}_{\bar{l}}$ on indices of $\mathcal{A}^{(\bar{s})}(\hat{x})$ and achieve a strictly better objective, contradicting the optimality of $Q$. Moreover, first-order necessary optimality conditions for the convex subproblem defining $x_+$ imply the existence of 
$\hat{v}\in\mathcal{N}_{\mathcal{X}}(x_+)$, $\bar{g}\in\partial \bar{h}(x_+)$, $g^{(s)}_l(x_+) \in \partial h^{(s)}_l(x_+)$ for every $s\in[N]$, $l \in [n_s]$ such that
\vspace{-5pt}
$$\textbf{0} = \hat{v} + \bar{g} + \frac{1}{N} \sum_{s=1}^{N}\,\sum_{l \,\in \,\mathcal{A}^{(s)}(x_+)}\,q^{(s)}_l\,g^{(s)}_l=\hat{v}+g_1-g_2\vspace{-7pt}$$
where $g_1 \in \partial f_1(x_+)$ and $g_2 \in \partial f_2(x_+)$ with
\vspace{-5pt}
$$
g_1 = \bar{g} + \frac{1}{N}\,\sum_{s=1}^N\,\sum_{l\, \in\, [n_s]}\,g^{(s)}_l,\hspace{10pt}g_2=\frac{1}{N}\,\sum_{s=1}^N\,\sum_{\tilde{l} \,\in\,\mathcal{A}^{(s)}(x_+)}\,q_{\tilde{l}}^{(s)}\sum_{l\,\in\,[n_s]\backslash\{\tilde{l}\}}\,g^{(s)}_l\vspace{-5pt}$$
\end{proof}

\section{Safeguarded exploration} \label{safeguard_proof}
\begin{lemma}[Exploration Bound]
Let $C \in [0,1]$. Imposing, for every $s \in [N]$,
\begin{equation}
\varepsilon^{(s)} =\min\big\{1,C\cdot \langle q^{(s)}-(q^*)^{(s)},\textbf{h}^{(s)}(x_+)\rangle / \langle \hat{q}^{(s)}_+ - (q^*)^{(s)},\textbf{h}^{(s)}(x_+)\rangle  \big\}\footnote{If $\hat{q}_+^{(s)} = (q^*)^{(s)}$ then $\varepsilon^{(s)}\gets 1$.}
\label{eq:max_safe_exploration_appendix}
\end{equation}
in \eqref{eq:RAM_Q_update} ensures that $ \bar{F}(x_+, Q) -  \bar{F}(x_+, Q_+) \geq (1-C)\cdot \mathcal{G}^*_{(x_+,Q)}$.
\end{lemma}
\begin{proof}
Let $Q_+(\epsilon) = (q^{(1)}_+(\varepsilon^{(1)}),\dots,q^{(N)}_+(\varepsilon^{(N)}))$ with, for every $s\in[N]$, $$q_+^{(s)}(\varepsilon^{(s)}) = \varepsilon^{(s)}\cdot\hat{q}^{(s)}_+ + (1-\varepsilon^{(s)})\cdot (q^*)^{(s)} $$ as suggested by \eqref{eq:RAM_Q_update}. We start by recalling that \begin{align*}
\mathcal{G}^*_{(x_+,Q)} &= \frac{1}{N}\,\sum_{s=1}^N\,\langle q^{(s)}-(q^*)^{(s)},\textbf{h}^{(s)}(x_+)\rangle\\
\bar{F}(x_+,Q)-\bar{F}(x_+,Q_+) &= \frac{1}{N}\,\sum_{s=1}^N\,\langle q^{(s)}-q_+^{(s)},\textbf{h}^{(s)}(x_+)\rangle
\end{align*}Based upon this decomposition, if for every $s\in [N]$, we show that 
\begin{equation} \langle q^{(s)}-q_+^{(s)},\textbf{h}^{(s)}(x_+)\rangle \geq (1-C)\cdot \langle q^{(s)}-(q^*)^{(s)},\textbf{h}^{(s)}(x_+)\rangle \label{eq:to_prove_seq} \end{equation}
then the proof is completed. We need to consider two cases. First, let $\varepsilon^{(s)}=1$. \\In such circumstances, $q_+^{(s)} = \hat{q}_+^{(s)}$ and according to \eqref{eq:max_safe_exploration_appendix},
$$  \langle \hat{q}^{(s)}_+ - (q^*)^{(s)},\textbf{h}^{(s)}(x_+)\rangle =   \langle q^{(s)}_+ - (q^*)^{(s)},\textbf{h}^{(s)}(x_+)\rangle \leq C\cdot \langle q^{(s)}-(q^*)^{(s)},\textbf{h}^{(s)}(x_+)\rangle $$
Multiplying the above inequality by $-1$ and adding $\langle q^{(s)}-(q^*)^{(s)}, \textbf{h}^{(s)}(x_+)\rangle$ on both sides yields \eqref{eq:to_prove_seq}. Now, let $\varepsilon^{(s)}<1$. We notice that $\varepsilon^{(s)}\geq 0$ by definition of $(q^*)^{(s)}$.
\begin{align*}
\langle q_+^{(s)}(\varepsilon^{(s)}),\textbf{h}^{(s)}(x_+)\rangle &=  \varepsilon^{(s)}\cdot \langle \hat{q}^{(s)}_+,\textbf{h}^{(s)}(x_+)\rangle + (1-\varepsilon^{(s)}) \cdot \langle (q^*)^{(s)},\textbf{h}^{(s)}(x_+)\rangle \\
&= \varepsilon^{(s)}\cdot \langle \hat{q}^{(s)}_+ - (q^*)^{(s)},\textbf{h}^{(s)}(x_+)\rangle + \langle (q^*)^{(s)},\textbf{h}^{(s)}(x_+)\rangle\\
&= C\cdot \langle q^{(s)}-(q^*)^{(s)},\textbf{h}^{(s)}(x_+)\rangle+ \langle (q^*)^{(s)},\textbf{h}^{(s)}(x_+)\rangle
\end{align*}
multiplying both sides by $-1$ and adding $\langle q^{(s)},\textbf{h}^{(s)}(x_+)\rangle$ on both sides yields \eqref{eq:to_prove_seq}.
\end{proof}

\end{document}